\documentclass[11pt,english]{article}
\usepackage[latin9]{inputenc}
\usepackage{geometry}
\usepackage{color}
\usepackage{babel}
\usepackage{amsthm}
\usepackage{amsmath}
\usepackage{amssymb}
\usepackage{graphicx}
\usepackage{esint}
\usepackage{multicol}
\usepackage[unicode=true,pdfusetitle,
 bookmarks=true,bookmarksnumbered=false,bookmarksopen=false,
 breaklinks=false,pdfborder={0 0 1},backref=false,colorlinks=true, linktocpage=true]
 {hyperref}
\hypersetup{
 hyperfootnotes=true, unicode=true}

\makeatletter
\numberwithin{equation}{section}
\theoremstyle{plain}
\newtheorem{thm}{\protect\theoremname}
  \theoremstyle{remark}
  \newtheorem{rem}[thm]{\protect\remarkname}
  \theoremstyle{definition}
  \newtheorem{example}[thm]{\protect\examplename}
  \theoremstyle{definition}
  \newtheorem{defn}[thm]{\protect\definitionname}
  \theoremstyle{plain}
  \newtheorem{prop}[thm]{\protect\propositionname}
  \theoremstyle{plain}
  \newtheorem{cor}[thm]{\protect\corollaryname}
  \theoremstyle{plain}
  \newtheorem{lem}[thm]{\protect\lemmaname}

\date{}

\usepackage{amsthm}
\usepackage{epsfig}
\usepackage{amsbsy}\usepackage{amsfonts}\usepackage{tikz}\usetikzlibrary{calc,fadings,decorations.pathreplacing}
\usepackage{tikz}
\usepackage{epsfig}
\usepackage{appendix}
\usepackage{endnotes}

\usetikzlibrary{arrows}
\usetikzlibrary{shapes}

\usepackage{epsfig}\usepackage{booktabs}

\newcommand{\onenorm}[1]{\left\lVert#1\right\rVert_1}
\newcommand{\core}{\Delta_{\text{core}}}
\newcommand{\cusp}{\Delta_{\text{cusp}}}
\newcommand{\deltao}{\Delta_{0}}
\newcommand{\on}{\text{ on }}

\newcommand{\wdenom}{\mu (w)}
\newcommand{\wdenomm}{\theta (w)}
\newcommand{\wdenommm}{\kappa (w)}
\newcommand{\wdenomij}{\psi (w)}
\newcommand{\wsum}{\beta (w)}

\newenvironment{claim}[1]{\par\noindent\underline{Claim:}\space#1}{}

\usepackage{caption}
\usepackage{blkarray}
\usepackage{kbordermatrix}

\theoremstyle{plain}

\hypersetup{pdfpagemode=UseNone}

\author{}
\date{\today}
\usepackage[leftcaption]{sidecap}

\newcommand{\FigBesBeg}[1][1.0]{%
 \let\MyFigure\figure
 \let\MyEndfigure\endfigure
 \renewenvironment{figure}[1]{\begin{SCfigure}[#1]##1}{\end{SCfigure}}}

\newcommand{\FigBesEnd}{%
 \let\figure\MyFigure
 \let\endfigure\MyEndfigure}

\makeatother

  \providecommand{\corollaryname}{Corollary}
  \providecommand{\definitionname}{Definition}
  \providecommand{\examplename}{Example}
  \providecommand{\lemmaname}{Lemma}
  \providecommand{\propositionname}{Proposition}
  \providecommand{\remarkname}{Remark}
\providecommand{\theoremname}{Theorem}
\renewcommand*{\thefootnote}{\fnsymbol{footnote}}

\begin{document}

\title{An asymptotic formula for integer points on Markoff-Hurwitz   varieties }

\author{Alex Gamburd\footnote{A. Gamburd was supported in part by NSF award DMS-1603715.}, Michael Magee\footnote{M. Magee was supported in part by NSF award DMS-1701357.}, and Ryan Ronan}

\renewcommand*{\thefootnote}{\arabic{footnote}}

\maketitle
\begin{abstract}
We establish an asymptotic formula for the number of integer solutions
to the Markoff-Hurwitz equation
\[
x_{1}^{2}+x_{2}^{2}+\ldots+x_{n}^{2}=ax_{1}x_{2}\ldots x_{n}+k.
\]
When $n\geq4$ the previous best result is by Baragar (1998) 
that gives an exponential rate of growth with exponent $\beta$ that
is not in general an integer when $n\geq4$. We give a new interpretation
of this exponent of growth in terms of the unique parameter for which
there exists a certain conformal measure on projective space. 

\end{abstract}
\global\long\def\A{\mathcal{A}}
\global\long\def\Haus{\mathrm{Haus}}
\global\long\def\SL{\mathrm{SL}}
\global\long\def\Z{\mathbf{Z}}
\global\long\def\D{\mathcal{\mathfrak{D}}}
\global\long\def\modtor{\mathbb{T}_{mod}}
\global\long\def\tr{\mathrm{tr}}
\global\long\def\Aut{\mathrm{Aut}}
\global\long\def\O{\mathcal{O}}
\global\long\def\P{\mathcal{P}_{r}}
\global\long\def\E{\mathcal{E}}
\global\long\def\o{\vec{o}}
\global\long\def\R{\mathbf{R}}
\global\long\def\H{\mathcal{H}}
\global\long\def\order{\mathsf{order}}
\global\long\def\xtuple{(x_{1},\ldots,x_{n})}
\global\long\def\ytuple{(y_{1},\ldots,y_{n})}
\global\long\def\scale{n^{1/(n-2)}}
\global\long\def\M{\mathcal{M}}
\global\long\def\ztuple{(z_{1},\ldots,z_{n})}
\global\long\def\ordered{(\R_{\geq0}^{n})_{\mathrm{ord}}}
\global\long\def\Log{\mathrm{Log}}
\global\long\def\e{\epsilon}
\global\long\def\acc{\texttt{\texttt{>>}}}
\global\long\def\L{\mathcal{L}}
\global\long\def\Int{\mathrm{Int}}
\global\long\def\N{\mathbf{N}}
\global\long\def\op{\mathrm{op}}
\global\long\def\K{\mathcal{K}}
\global\long\def\LL{\tilde{\L}}
\global\long\def\C{\mathbf{C}}
\global\long\def\g{\gamma}
\global\long\def\image{\mathrm{image}}
\global\long\def\G{\Gamma}
\global\long\def\Jac{\mathrm{Jac}}

\tableofcontents{}

\section{Introduction}

For integer parameters $n\geq3,$ $a\geq1,$ and 
$k\in \Z$
consider
the Diophantine equation

\begin{equation}
x_{1}^{2}+x_{2}^{2}+\ldots+x_{n}^{2}=ax_{1}x_{2}\ldots x_{n}+k.\label{eq:markoff-hurwitzeq}
\end{equation}
We call this the generalized\footnote{Normally $k=0$ is considered.}
\emph{Markoff-Hurwitz equation. }In this paper we count solutions
to (\ref{eq:markoff-hurwitzeq}) in 
integers, which we we
call \emph{Markoff-Hurwitz tuples}. More precisely, let $V$ be the
affine subvariety of $\C^{n}$ cut out by $\eqref{eq:markoff-hurwitzeq}$.
We are interested in the asymptotic size
of the set
\[
V(\Z)\cap B(R)
\]
 where $B(R)$ is the ball of radius $R$ in the $\ell^{\infty}$
norm on $\R^{n}\subset\C^{n}.$

When $n=3,a=3$ and $k=0$ solutions to (\ref{eq:markoff-hurwitzeq}) 
 in positive integers 
are called Markoff triples, and the numbers that appear therein are
called Markoff numbers\footnote{A long standing conjecture of Frobenius asserts that each Markoff
number appears as the maximal entry of only one triple, up to reordering.
If one assumes this conjecture, then the problems of counting Markoff
triples and numbers are the same.}. The Markoff numbers are intimately connected with Diophantine properties
of the rationals via the Markoff spectrum \cite{BOMBIERI}, and also
with hyperbolic geometry and free groups \cite{AIGNER}.

The question of counting $|V(\Z)\cap B(R)|$ for Markoff triples
was first investigated in the thesis of Gurwood 
\cite{GURWOOD} who established an asymptotic formula using the correspondence
between Markoff and Farey trees. An improved error
term was obtained by Zagier in \cite[pg. 711]{ZAGIER}, and a very
clean proof of a slightly weaker result can be found in Belyi \cite{BELYI}.
The current best result is due to McShane and Rivin \cite{RIVINMCSHANE}:
\begin{thm}[McShane-Rivin]
\label{thm:zagier}The number $M(R)$  of Markoff triples $(x,y,z)$ with
$x\leq y\leq z\leq R$ is given by
\[
M(R)=C(\log R)^{2}+O(\log R\log\log R)
\]
as $R\to\infty$, with $C>0.$ 
\end{thm}

Perhaps somewhat surprisingly, the asymptotic growth for $n\geq 4$ is not of the order $(\log R)^{n-1}$, as was first noticed by Baragar \cite{BARAGAR2}, who obtained the following result

\begin{thm}[Baragar]
\label{thm:Baragar-theorem}There is a number $\beta=\beta(n)$ such
that when $k=0,$ if 
$V(\Z)-\{(0,0,0)\}$
 is nonempty then for every $\varepsilon >0$ 
\begin{equation}
\Omega((\log R)^{\beta(n) - \varepsilon} )\leq |V(\Z)\cap B(R)| \leq O((\log R)^{\beta(n) + \varepsilon}) .\label{eq:barager-expoential-growth}
\end{equation}
\end{thm}
This was strengthened by Baragar in \cite{BARAGAR3} under the same hypotheses to 
\begin{equation}
|V(\Z)\cap B(R)|=(\log R)^{\beta+o(1)}.\label{eq:barager-expoential-growth}
\end{equation}
In \cite{BARAGAR3} the following bounds for the exponents $\beta(n)$ were also obtained
\begin{eqnarray*}
\beta(3) & = & 2,\\
\beta(4) & \in & (2.430,2.477),\\
\beta(5) & \in & (2.730,2.798),\\
\beta(6) & \in & (2.963,3.048),
\end{eqnarray*}
 and in general
\[
\frac{\log(n-1)}{\log2}<\beta(n)<\frac{\log(n-1)}{\log2}+o(n^{-0.58}).
\]

 In 1995 \cite{Silverman}, it was asked by Silverman  whether in the setting of $k=0$
  \begin{enumerate}
 \item there is a true asymptotic formula for $|V(\Z)\cap B(R)|$ with main term proportional to $\log(R)^\beta$, and
 \item furthermore, $\beta(n)$ is irrational?
 \end{enumerate}
The irrationality of $\beta$ remains a tantalizing open question and one may wonder whether it is even algebraic. On the other hand, our methods do give some further insight into the nature of this mysterious number (cf. Theorem \ref{thm:beta-characterization} below). 
The main goal of this paper is to extend Baragar's exponential rate of growth estimate to a true asymptotic formula\footnote{The techniques in \cite{BARAGAR3} ``were inspired in part by Boyd's work on the Apollonian packing problem \cite{Bo1, Bo2, Bo3}.''  Boyd's result was extended to a true asymptotic formula in the work of Kontorovich and Oh \cite{KO}.}.

When $k>0$ there are certain exceptional families of solutions to
(\ref{eq:markoff-hurwitzeq}) that have a different quality of growth.
We describe these families in Definition \ref{def:exceptional} and
for fixed $k,a,n$ we write $\E$ for the set of exceptional tuples. 
We obtain the following theorem for the asymptotic number of Markoff-Hurwitz
tuples.
\begin{thm}
\label{thm:main-counting}For each $(n,a,k)$ with $V(\Z)-\E$
infinite, there is a positive constant $c=c(n,a,k)$ such that
\[
|(V(\Z)-\E)\cap B(R)|=c(\log R)^{\beta}+o((\log R)^{\beta}).
\]
Here $\beta$ is the same constant as Theorem \ref{thm:Baragar-theorem}.
\end{thm}

\begin{rem}
We explain in Section \ref{sub:basic-properties} that removing $\E$
is necessary in Theorem \ref{thm:main-counting} since the exceptional
families have
$|\E\cap B(R)|\geq cR$, $c>0$ for  $R\geq R_0(n,a,k)$
 when they are non-empty.
On the other hand, $\E$ is non-empty only when $k-n+2$ or $k-n-1$
is a square.
\end{rem}

\begin{rem} The issue of the existence and infinitude of integral solutions for general $a, k$, even for $n=3$,  is quite subtle: see \cite{MORDELL, SM}. In recent work of Ghosh and Sarnak \cite{GSMARKOFF}, the Hasse principle is  established to hold for Markoff-type cubic surfaces $x_1^2+x_2^2+x_3^2-x_1 x_2 x_3 = k$  for almost all $k$.
\end{rem}

\begin{rem}
In Theorem \ref{thm:beta-characterization} we give a new characterization
of $\beta$ as the unique parameter for which there exists a conformal
measure for the action of a linear semigroup on projective space.
\end{rem}

Our counting arguments, as in \cite{ZAGIER} and \cite{BARAGAR2,BARAGAR3},
depend on an infinite descent for solutions to (\ref{eq:markoff-hurwitzeq})
that goes back to Markoff \cite{MARKOFF} in the case of Markoff triples
and Hurwitz \cite{HURWITZ} in the higher dimensional setting of $n>3$,
$k=0.$
In Section \ref{sub:basic-properties} we explain how the counting problem for $V(\Z)$
can be related to the analogous one for $V(\Z_+)$, where $\Z_+$ are the positive integers.

Given $x\in V(\Z_{+}),$ fixing all of the coordinates of
$x$ except $x_{j}$ and viewing (\ref{eq:markoff-hurwitzeq}) as
a quadratic polynomial in $x_{j}$, the other root is given by 
\[
x'_{j}=a\prod_{i\neq j}x_{i}-x_{j}.
\]
Therefore for each $j$ one has the \emph{Markoff-Hurwitz move}
\[
m_{j}(x_{1},x_{2},\ldots,x_{n})=(x_{1},x_{2},\ldots,\underbrace{a\prod_{i\neq j}x_{i}-x_{j}}_{j},\ldots,x_{n})
\]
that yields a new solution to (\ref{eq:markoff-hurwitzeq}). Infinite
descent for the Markoff-Hurwitz equation says that any unexceptional
tuple 
in $V(\Z_+)$
can be reduced to one in a compact set $K_{0}=K_{0}(n,a,k)$
by a series of Markoff-Hurwitz moves (cf. Corollary \ref{cor:infinite-descent}).

After renormalizing (\ref{eq:markoff-hurwitzeq}), the Markoff-Hurwitz
moves $\{m_{j}\}$ induce the moves

\[
\lambda_{j}(z_{1},\ldots,z_{n})=\left(z_{1},\ldots,\widehat{z_{j}},\ldots,z_{n},\prod_{i\neq j}z_{i}-z_{j}\right),\quad1\leq j\leq n-1,
\]
on ordered tuples, where $\widehat{\bullet}$ denotes omission. If
enough of  the $z_{i}$ are large, the move $\lambda_{j}$
can be approximated by
\[
z\mapsto\left(z_{1},\ldots,\widehat{z_{j}},\ldots,z_{n},\prod_{i\neq j}z_{i}\right)
\]
to high accuracy relative to the largest entries of $z$. At the level
of logarithms this corresponds to 
\[
(\log z_{1},\log z_{2},\ldots,\log z_{n})\mapsto\left(\log z_{1},\ldots,\widehat{\log z_{j}},\ldots,\log z_{n},\sum_{i\neq j}\log z_{i}\right).
\]
Thus one is naturally led to study the linear semigroup generated
by linear maps
\begin{equation}
\g_{j}(y_{1},y_{2},\ldots,y_{n})=\left(y_{1},\ldots,\widehat{y_{j}},\ldots,y_{n},\sum_{i\neq j}y_{i}\right)\label{eq:gamma-generators-defn}
\end{equation}
on ordered $n$-tuples $(y_1,\ldots,y_n)$. Indeed, this is the approach of Zagier \cite{ZAGIER} in the setting
of Markoff triples and Baragar \cite{BARAGAR2} for general $n,a$
with $k=0$. Let 
\[
\Gamma=\langle\g_{1},\ldots,\g_{n-1}\rangle_{+}
\]
where we have written a `$+$' to indicate we are generating a semigroup,
not a group.

An important idea in this work that explains why we are able to
make progress on the counting problem is that we replace\footnote{See our discussion in Section \ref{sub:acceleration} about the benefits
of this replacement. It is inspired by the `Time Acceleration Machine' described by Zorich in \cite[Section 5.3] {zorich}.} the generators of $\Gamma$ with the countably infinite generating
set
\[
T_{\Gamma}=\left\{ \:\gamma_{n-1}^{A}\gamma_{j}\::\:A\in\Z_{\geq0},\:1\leq j\leq n-2\:\right\} 
\]
and then consider the semigroup 
\[
\Gamma'=\langle\:T_{\Gamma}\:\rangle_{+}.
\]

Both $\Gamma$ and $\Gamma'$ are freely generated by their respective
generating sets\footnote{This follows from a similar argument to the proof of Lemma \ref{lem:Lambda-is-free}
we give below.}. Notice that $\Gamma$ and $\Gamma'$ preserve the nonnegative ordered
hyperplane

\begin{equation}
\H\equiv\left\{ \:(y_{1},\ldots,y_{n})\in\R_{\geq0}^{n}\::\:y_{1}\leq y_{2}\leq\ldots\leq y_{n,\:}\sum_{j=1}^{n-1}y_{j}=y_{n}\:\right\} \subset\R_{\geq0}^{n}\label{eq:hyperplane-defn}
\end{equation}
and that any element of $\Gamma$ maps ordered tuples in $\R_{\geq0}^{n}$
into $\H$. Therefore the study of orbits of $\Gamma$ and $\Gamma'$
on ordered tuples boils down to the study of orbits in $\H$. 
\begin{example}
When $n=3,$ the linear map $\sigma:\H\to\H$ defined by
\begin{equation}
\sigma(a,b,a+b)=\order\left(b-a,a,b\right),\label{eq:sigma_0-formual-1}
\end{equation}
where $\order$ puts a tuple in ascending order from left to right,
is such that for $j=1,2$ we have
\[
\sigma\g_{j}.y=y
\]
for all $y\in\H.$ Repeatedly applying the map $\sigma$ to a triple
$(a,b,a+b)$ with $a\leq b\in\Z$ performs the Euclidean algorithm
on $a,b.$ However, one application of $\sigma$ corresponds in general
to less than one step of the algorithm. Replacing $\Gamma$ with $\Gamma'$
corresponds to speeding this up so one whole step of the Euclidean
algorithm corresponds to one semigroup generator. As for counting,
the orbit of $(0,1,1)$ under $\Gamma$ is precisely those $(a,b,a+b)$
with $(a,b)=1$ and thus can be counted by elementary methods. This
is exploited in Zagier's paper \cite{ZAGIER}. 
\end{example}
We can use the basis 
\[
e_{j}=(0,\ldots,0,\underbrace{1}_{j},0,\ldots,0,1)
\]
for the subspace spanned by $\H$. This basis clarifies the action
of $\Gamma'$. 
\begin{example}
When $n=3$ the semigroup $\Gamma'$ is generated by the
\[
g_{A}:=\gamma_{2}^{A}\gamma_{1}=\left(\begin{array}{cc}
0 & 1\\
1 & A+1
\end{array}\right)
\]
with respect to the basis $\{e_{1},e_{2}\}.$ These generators are
classically connected with continued fractions by the formulae

\[
\left(\begin{array}{cc}
0 & 1\\
1 & A_{1}
\end{array}\right)\left(\begin{array}{cc}
0 & 1\\
1 & A_{2}
\end{array}\right)\ldots\left(\begin{array}{cc}
0 & 1\\
1 & A_{k}
\end{array}\right)=\left(\begin{array}{cc}
\star & b\\
\star & d
\end{array}\right),\quad\frac{b}{d}=\cfrac{1}{A_{1}+\cfrac{1}{A_{2}+\ddots\cfrac{1}{A_{k}}}}.
\]

\end{example}

\begin{example}\label{ex:Rauzy}

When $n=4$ the semigroup $\Gamma$ acts in the basis given by the
$e_{i}$ as
\[
\gamma_{1}=\left(\begin{array}{ccc}
0 & 1 & 0\\
0 & 0 & 1\\
1 & 1 & 1
\end{array}\right),\:\gamma_{2}=\left(\begin{array}{ccc}
1 & 0 & 0\\
0 & 0 & 1\\
1 & 1 & 1
\end{array}\right),\:\gamma_{3}=\left(\begin{array}{ccc}
1 & 0 & 0\\
0 & 1 & 0\\
1 & 1 & 1
\end{array}\right).
\]
This semigroup appears naturally in different areas of mathematics.
In most situations that this semigroup appears, as will also be the
case in this paper, the dynamics of the projective linear action of
$\Gamma$ on $\R_{+}^{3}/\R_{+}$ becomes relevant. Up to the minor modification
of possibly multiplying the generators on
the left or right by permutation matrices, the iterated function system
given by the projective linear action of $\Gamma$ on $\R_{+}^{3}/\R_{+}$
has a fractal attracting set that is known as the \emph{Rauzy gasket. }

The Rauzy gasket first appears in the literature in a paper of Levitt
\cite{LEVITT} in connection with the dynamics of partially defined
rotations of the circle. The Rauzy gasket has been rediscovered by
different groups of mathematicians, including De Leo and Dynnikov
\cite{DD} in connection to a conjecture of Novikov \cite{NOV} on
triply periodic surfaces, Arnoux and Starosta \cite{AS} (wherein
the Rauzy gasket was given its name) in relation to generalizations
of Sturmian words to three letters and the `fully subtractive' continued
fractions algorithm, and now, in this paper, in connection to Diophantine
geometry.

The Rauzy gasket was proven by Avila, Hubert, and Skripchenko \cite{AHS1}
to have Hausdorff dimension less than 2, answering a question of Arnoux.
The acceleration, replacing $\Gamma$ by $\Gamma'$, that we perform
here is also carried out (in the context of iterated function systems)
by Arnoux and Starosta \cite{AS} and Avila, Hubert, and Skripchenko
\cite{AHS1}, where the acceleration is viewed as analogous to Zorich's acceleration (see  \cite[Section 5.3] {zorich})
 of Rauzy-Veech induction that is well known in Teichm\"{u}ller
dynamics. 

It is also worth pointing out that higher dimensional versions of
the Rauzy gasket have been defined \cite{AS,DELEO2}, and the branches
of the corresponding iterated function system, after the same simple
modifications as before, match with our $\Gamma$ for $n>4$.

Some of our technical results in Sections \ref{sec:linear-semigroup-count} and \ref{sec:dynamics} can be
closely compared to, intersect with, or generalize, results obtained
by Avila, Hubert, and Skripchenko for the Rauzy gasket in \cite{AHS2,AHS1}.
We point out these intersections throughout the paper. 

\end{example}

So our semigroups $\Gamma$ and $\Gamma'$ are natural extensions
of the Euclidean algorithm and continued fractions semigroup to higher
dimensions\footnote{See  \cite{zorich} for the discussion of such an extension in the context of translation surfaces.}.  We write $\Delta=\H/\R_{+}$ and we can view $\Delta$
as a subset of $\R^{n-2}$ (see Section \ref{sec:dynamics} for details).
The key distinction that appears when $n\geq4$ is that
\[
\Delta\neq\bigcup_{j=1}^{n-1}\gamma_{j}(\Delta)
\]
 and so the induced dynamics on $\H/\R_{+}$ has `holes' as we illustrate
in Figure \ref{fig:Color-image}.

\begin{figure}

\includegraphics[scale=0.4]{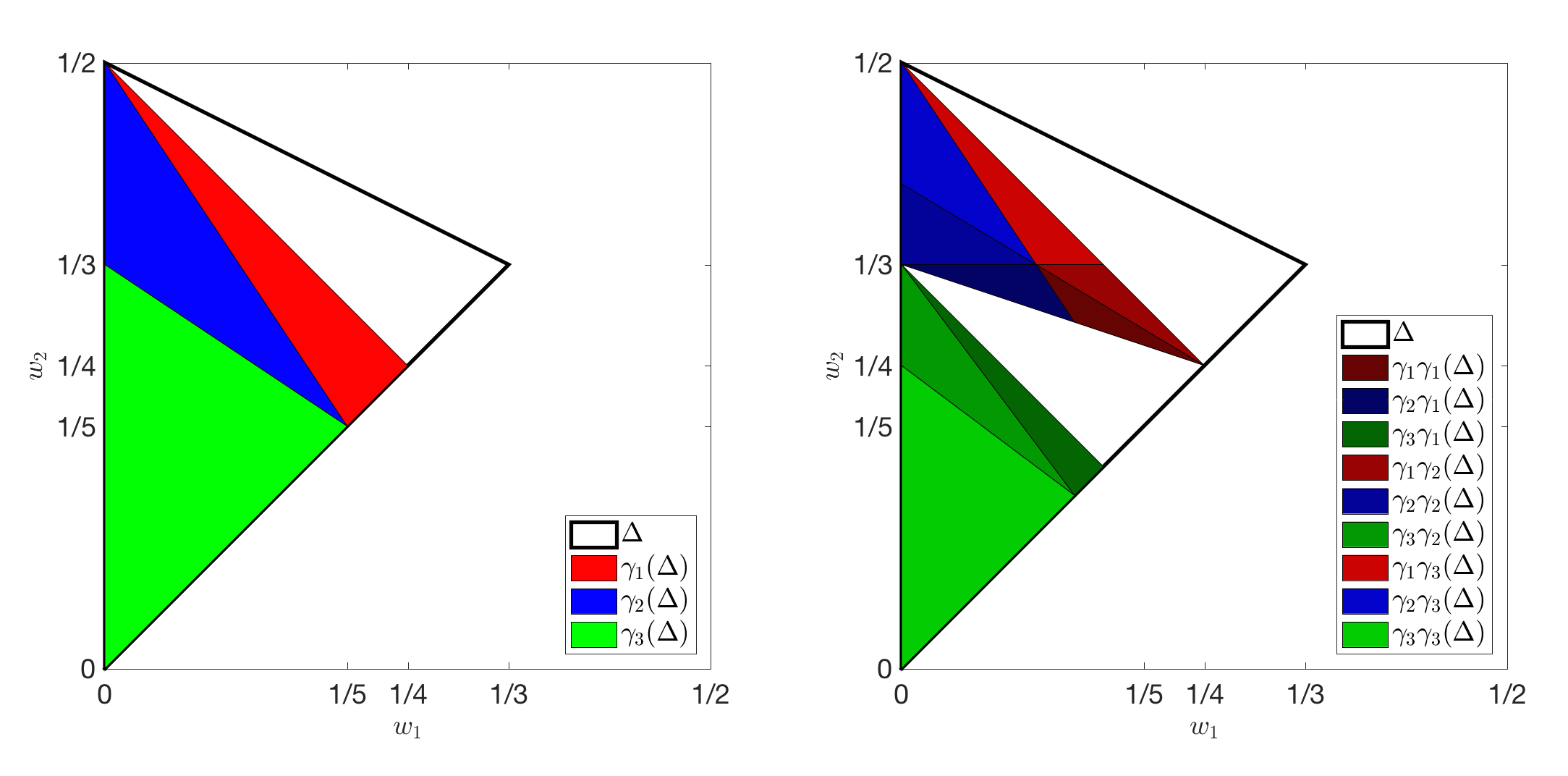}\protect\caption{\label{fig:Color-image}When $n=4,$ the semigroup elements map $\Delta=\protect\H/\protect\R_{+}$
into a strictly smaller subset. After iteration this leads to more
and more empty space (see also Figure \ref{fig:fractal}). This doesn't
occur when $n=3,$ as one can also see from the picture: the action
of the group elements $\gamma_{2}$ and $\gamma_{3}$ on the vertical
coordinate axis is a copy of the $n=3$ dynamics.}
\label{color-picture}

\end{figure}

We get a new characterization of the parameter $\beta$ in terms of
the action of $\G'$ on $\H/\R_{+}.$
\begin{thm}
\label{thm:beta-characterization}The $\beta$ from Theorem \ref{thm:Baragar-theorem}
is the unique parameter in $(1,\infty)$ such that there exists a
probability measure $\nu_{\beta}$ on $\Delta=\H/\R_{+}$ with the
property

\[
\int_{w\in\Delta}f(w)\:d\nu_{\beta}(w)=\sum_{\gamma\in T_{\G}}\int_{w\in\Delta}f(\gamma.w)|\Jac_{w}(\gamma)|^{\frac{\beta}{n-1}}\:d\nu_{\beta}(w)
\]
for all $f\in C^{0}(\Delta).$ We call $\nu_{\beta}$ a conformal
measure. \end{thm}
\begin{rem}
Theorem \ref{thm:beta-characterization} can be viewed as a partial
analog of the connection between the exponent of growth of a finitely
generated Fuchsian group and the Hausdorff dimension of its limit
set as a result of Patterson-Sullivan theory \cite{PATTERSON, SUL79, SUL84}. In
our setting, the lack of any symmetric space means the parameter $\beta$
is not in any obvious way connected to the Hausdorff dimension of
the compact $\Gamma'$-invariant subset of $\Delta.$
\end{rem}
\begin{rem}
 In the case of $n=4$, the measure $\nu_\beta$ is essentially the same as the measure obtained for the Rauzy gasket by Avila, Hubert, and Skripchenko in \cite[Theorem 1]{AHS2} in the context of a problem of Novikov \cite{NOV} on triply periodic surfaces. 

\end{rem}
\begin{figure}
\begin{centering}
\includegraphics[scale=0.45]{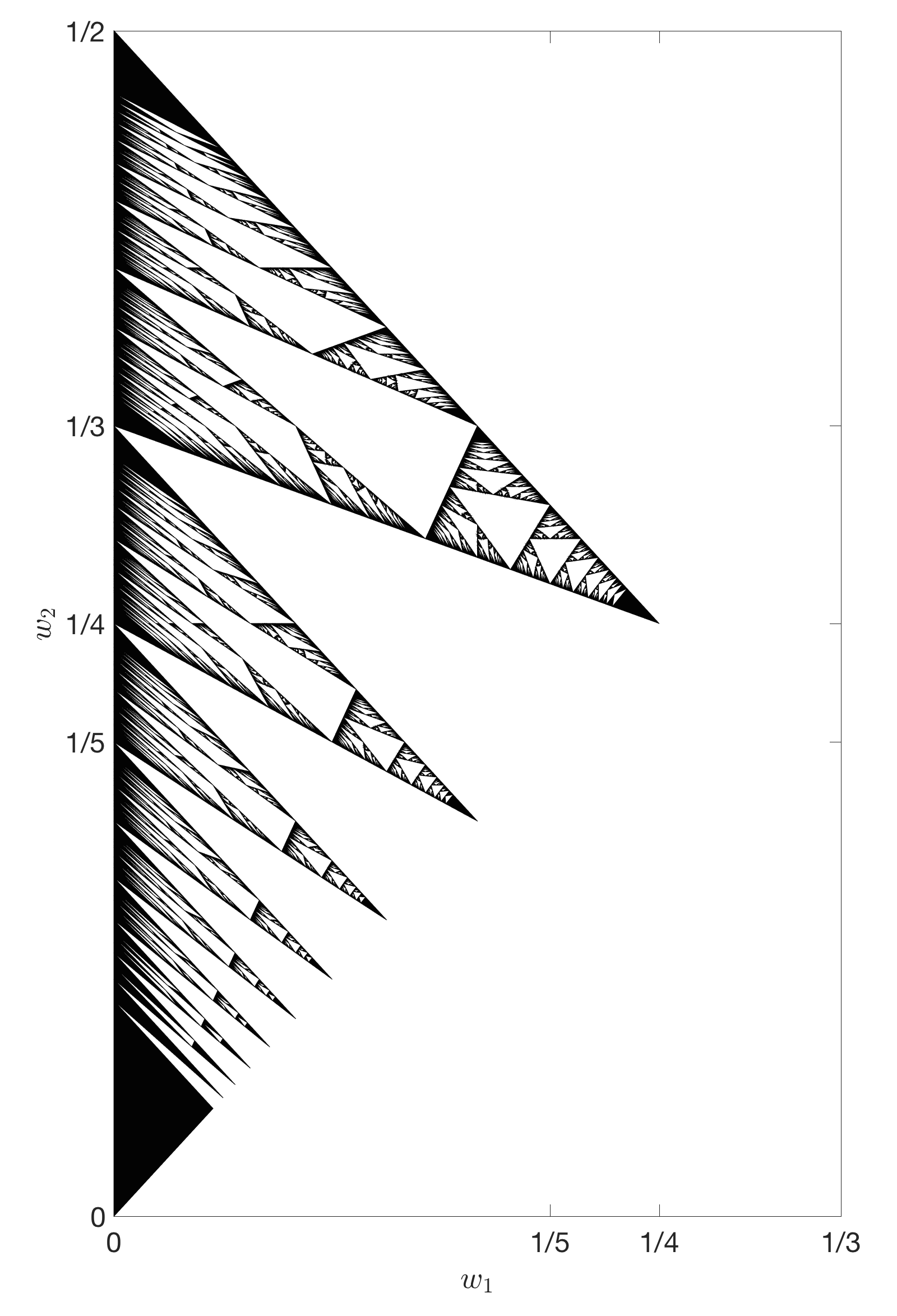}
\par\end{centering}

\centering{}\protect\caption{\label{fig:fractal}In the same setting ($n=4$) of Figure \ref{fig:Color-image},
we show in black the images of $\Delta$ under the action of all words
of length 10 in the generators $\{\gamma_{1},\gamma_{2},\gamma_{3}\}$.}
\end{figure}

In Section \ref{sec:comparison-section} we reduce Theorem \ref{thm:main-counting}
to a counting theorem for orbits of the semigroup $\Gamma'.$ The relevant counting quantity is defined by
\begin{equation}
N(y,a)\equiv\sum_{\g\in\Gamma'\cup\{e\}}\mathbf{1}\{\log(\gamma.y)_{n}-\log(y)_{n}\leq a\}\label{eq:N-of-y-a-definition}
\end{equation}
for $y\in\H-0$ and $a\geq0.$ We prove
\begin{thm}
\label{thm:-linear-counting}There is a positive bounded $C^{1}$
function $h$ on $\H$ that is invariant under the action of $\R_{+}$
and such that
\[
N(y,a)=h(y)e^{\beta a}(1+o_{a\to\infty}(1))
\]
for all $y\in\H-0,$ where the implied function in the small $o$
does not depend on $y.$ Moreover, $h$ satisfies the recursion

\begin{equation}
\sum_{\g\in T_{\Gamma}}\left(\frac{(\gamma.y)_{n}}{y_{n}}\right)^{-\beta}h(\gamma.y)=h(y).\label{eq:h-recursion}
\end{equation}
The constant $\beta$ is the same as in Theorem \ref{thm:Baragar-theorem}.
\end{thm}

\begin{rem}
\label{rem:betabig}The embedding of the $(n-1)$-dimensional version
of $\H$ inside the $n$-dimensional version implies by Theorem \ref{thm:-linear-counting}
that $\beta(n)\geq\beta(n-1)$ and in particular that $\beta(n)\geq2$
for all $n\geq3.$ 
\end{rem}

\subsection{Connection to simple closed curves and character varieties}

Theorem \ref{thm:zagier} can be rephrased
as a counting result for the number of simple\footnote{This means there are no self crossings.}
closed geodesics of length $\leq\log R$ on the modular torus. This is
the topological once-punctured torus that is uniformized by the quotient
of the hyperbolic plane by the group
\[
\mathbf{}\Big\langle\left(\begin{array}{cc}
1 & 1\\
1 & 2
\end{array}\right),\left(\begin{array}{cc}
1 & -1\\
-1 & 2
\end{array}\right)\Big\rangle \leq \mathrm{PSL}_2(\R).
\]

McShane and Rivin \cite{RIVINMCSHANE} actually obtain the analogous
counting result to Theorem \ref{thm:zagier} for simple closed geodesics on
arbitrary hyperbolic once punctured tori, by use of a special norm
on the first homology of the surface. Mirzakhani proved in \cite{MIRZSIMPLE}
an asymptotic counting result, without explicit error term, for simple
closed geodesics on any finite area complete Riemann surface. These
asymptotics have recently been extended by Mirzakhani \cite{MIRZCOUTING}
to more general orbits of the mapping class group. In Mirzakhani's
results the exponents of growth are dimensions of Teichm\"{u}ller
spaces. It is interesting to compare this to our characterization
of Theorem \ref{thm:beta-characterization}.

In \cite{HN}, Huang and Norbury showed that when $n=a=4$ and $k=0$,
 $V(\R_+)$ is a parametrization of the Teichm\"{u}ller space of finite area
hyperbolic structures on $\R P^2$ minus three points, and moreover
the coordinates of points on $V(\R_+)$ are functions of the lengths
of \emph{one-sided}\footnote{This means a thickening of the geodesic
is homeomorphic to a M\"{o}bius band.} simple closed geodesics in the relevant hyperbolic structure.
From these facts they deduce from Baragar's Theorem \ref{thm:Baragar-theorem}
that the number $n^{(1)}_J(L)$ of one sided simple closed geodesics of length $\leq L$ in a hyperbolic structure $J$ on $ \R P^2$ minus three points
satisfies 
\begin{equation*}
\lim_{L \to \infty}\frac{ \log  n^{(1)}_J(L) }{ \log L } =  \beta(4).
\end{equation*}
The second author (Magee) of this paper has recently shown \cite{Mageecurves} that the methods here
can be extended to prove that $n^{(1)}_J(L)$ is asymptotic to $c L ^\beta$, for some $c=c(J)>0$, somewhat in analogy to Mirzakhani's results.

We also mention the recent work of Hu, Tan and Zhang \cite{HPZ} that
describes some regions in $\C^{n}$ where the group of automorphisms
of (\ref{eq:markoff-hurwitzeq}) acts properly discontinuously. This
extends previous work of Goldman \cite{GOLDMAN} that describes ranges
of $k$ in the case of $n=3$ where the group $\Aut(V)$
act ergodically or properly discontinuously (or some combination thereof,
on different components of the variety). Quite strikingly, for certain
ranges of $k$ the action of $\Aut(V)$ is ergodic on $V(\R)$
yet preserves the infinite discrete subset $V(\Z)$.
In \cite{HPZ} the authors also prove a `McShane identity' that gives
a closed form expression for $1$ in terms of an infinite sum over
an orbit of the semigroup; see \cite{MCSHANETHESIS, MCSHANEIDENTITY} for McShane's
original identity.

\subsection{Paper organization}

We prove our theorems in the order we have stated them with earlier
parts of the paper depending on later parts. In Section \ref{sec:MH-tuples-and-moves}
we describe the passage from $V(\Z)$ to $V(\Z_+)$ and
 describe in full the action of the Markoff-Hurwitz moves on $V(\Z_{+}).$
At the end of Section \ref{sec:MH-tuples-and-moves} we have fixed
a large compact region of $V(\Z_{+})$ outside of which the orbits
of the action  of  Markoff-Hurwitz generators  are a disjoint union of a finite number of
orbits that we understand well. In Section \ref{sec:converting-linear-to-nonlinear}
we fit the counting of these orbits to certain counts for the linear
semigroup $\Gamma'.$ Using Theorem \ref{thm:-linear-counting} as
a black box, we prove Theorem \ref{thm:main-counting}. In Section
\ref{sec:linear-semigroup-count}, we prove Theorems \ref{thm:-linear-counting}
and \ref{thm:beta-characterization} given Proposition \ref{prop:uniformly-contracting}
that says the action of $\Gamma'$ on projective space is contracting.
It is at this point we establish the connection with Baragar's exponent
of growth that we call $\beta.$ Finally, in Section \ref{sec:dynamics}
we prove Proposition \ref{prop:uniformly-contracting}.

\subsection{Notation}

For the reader's convenience we describe the notation we use in this
paper. We will use $\mathbf{1}$ for an indicator function. A vector
with an entry $\hat{\bullet}$ with a hat means that that entry is
omitted. We use Vinogradov notation $O,o,\ll,\gg$ in the standard
way. Any implied constants may depend on $n,a,k$ that we view as
fixed throughout much of the paper. If there is any dependence of
an implied constant on a variable we denote this as a subscript e.g.
$\ll_{\e}$, and we also use subscripts to indicate which variable
is tending to a limit, e.g. $o_{a\to\infty}(1).$ For the sake of
convenience, we take the liberty of applying functions to vectors,
which means we apply the function component-wise, and we write inequalities
between vectors to mean that the inequality holds at every component.
For a set $S$ in a semigroup we may write $S^{(k)}$ for the $k$-fold
product of the set with itself. We also write $\R_{+},\R_{\geq0}$
for the sets of positive (resp. nonnegative) real numbers, and similar
for integers. We write $\{x\}$ for the fractional part of a real
number $x$, that is, $x=n+\{x\}$ for $n\in\Z$ and $\{x\}\geq0.$

\subsection*{Acknowledgements }

We would like to thank Peter Sarnak, Giulio Tiozzo, and Peter Whang for helpful
conversations about this work.

\section{Markoff-Hurwitz tuples and moves}

\label{sec:MH-tuples-and-moves}

\subsection{Basic properties of the Markoff-Hurwitz equation}

\subsubsection*{The automorphism group}
By an automorphism of $V$ we mean a polynomial automorphism of $V(\C)$. We write $\Aut(V)$ for the group of all such maps. By results of Horowitz \cite{Horowitz} when $n=3$ and Hu, Tan and Zhang \cite[Theorem 1.1]{HPZ} for $n\geq4$, one has
 \[
\Aut(V) = \mathcal{G} \rtimes(N\rtimes S_{n})
\]
where
\begin{enumerate}
\item $N$ is the group of transformations that change the sign of an even
number of variables. Hence $|N| = 2^{n-1}$.
\item $S_{n}$ is the symmetric group on $n$ letters that acts by permuting
the coordinates of $\C^{n}$.
\item $\mathcal{G}$ is the nonlinear group generated by the Markoff-Hurwitz moves $m_j$ discussed in the Introduction.
\end{enumerate}
One important corollary of this classification is that $V(\Z)$ is invariant under $\Aut(V)$.

\label{sub:basic-properties}

\subsubsection*{Exceptional solutions}

For $a=1$ and $a=2$ there are certain exceptional families of points in $V(\Z)$ whose
growth rate is totally different from the points we wish to count\footnote{See Silverman \cite{SILVERMANBOOK} for a discussion of a phenomenon of surfaces containing curves that have many more integral points than one would expect from the surface as a whole.}. These
appear only for certain values of $k$ and we describe them now.
\begin{defn} 
\label{def:exceptional}We say that $x\in V(\Z_{+})$ is 
a \emph{fundamental exceptional solution}
if it belongs to one of the following two families
\begin{enumerate}
\item One has $a=1$ and after reordering the coefficients of $x$ so that
$x_{1}\leq x_{2}\leq\ldots\leq x_{n}$
\[
x_{1}=x_{2}=\ldots=x_{n-3}=1,\quad x_{n-2}=2.
\]
In this case $x$ is a Markoff-Hurwitz tuple if and only if 
\begin{equation}
(x_{n-1}-x_{n})^{2}=k-n-1.\label{eq:exceptional1}
\end{equation}
\item One has $a=2$ and after reordering the coefficients of $x$ so that
$x_{1}\leq x_{2}\leq\ldots\leq x_{n}$ 
\[
x_{1}=x_{2}=\ldots=x_{n-2}=1.
\]
In this case $x$ is a Markoff-Hurwitz tuple if and only if 
\begin{equation}
(x_{n-1}-x_{n})^{2}=k-n+2.\label{eq:exceptional2}
\end{equation}
\end{enumerate}
We say that $x\in V(\Z)$ is an \emph{exceptional solution} if $x$ is in the $\Aut(V)$-orbit of a fundamental exceptional solution. We write $\E$ for the collection of exceptional solutions in $V(\Z)$. If $x \in V(\Z)$ is not an exceptional solution we say it is an \emph{unexceptional solution}.
\end{defn}

Note that if (\ref{eq:exceptional1}) or (\ref{eq:exceptional2})
occur then they occur in an infinite family for that given $n,a,k.$
In either case, all sufficiently 
  large
positive
   integers appear as the maximal
entry of some fundamental exceptional solution and this maximal entry determines the tuple up
to reordering. Therefore for some $c>0$ there are $cR+O(1)$
fundamental exceptional solutions
with maximal entry $\leq R$. This is not the type of growth we want
to study (cf. Theorem \ref{thm:main-counting}). 
 It is also clear, but useful to note, that the property of being exceptional (respectively, unexceptional) in $V(\Z)$ is $\Aut(V)$-invariant.

\subsubsection*{Passage from $V(\Z)$ to $V(\Z_+)$}

We now describe the relationship between asymptotic counting of $V(\Z)-\E$ and $V(\Z_+)-\E$. Recall that $n\geq 3$, $a\geq 1$ and $k$ are fixed integers, and $N$ is the group of automorphisms of $V=V_{n,a,k}$ that change the sign of an even number of the coordinates. We decompose the action of $N$ on $V(\Z)-\E$ as follows.

Let $X_0$ be the elements of $V(\Z)-\E$ with at least one coordinate equal to $0$. If $k <0$ then $X_0$  is empty, and if $k\geq0$ then one obtains for $(x_1,\ldots,x_n)\in X_0$ the equation
\begin{equation*}
x_1^2 + \ldots + x_n^2 = k
\end{equation*}
from which it is apparent that $X_0$  is finite, with a bound on its size depending on $n$ and $k$. To indicate this we write $|X_0| = O_{n,k}(1)$.

Now let $X(R) = (V(\Z)- \E - X_0 )\cap B(R)$, the unexceptional elements of $V(\Z)$ with norm $\leq R$ and no zero coordinate. The group $N$ acts freely on $X(R)$. Therefore
\begin{equation*}
2^{n-1} |N \backslash X(R) | = |X(R) | .
\end{equation*}
The orbits of $N$ on $X(R)$  fall into two categories, according to which we decompose 
\begin{equation*}
N \backslash X(R) = Y_+(R) \sqcup Y_-(R)
\end{equation*}
where $Y_+(R)$ are orbits with a  unique representative with all coordinates positive,  and $Y_-(R)$ the remaining orbits, which have a unique representative with $x_1 < 0$ and $x_i>0$ for $i\geq 2$.

	We now argue that $|Y_-(R)|$ is bounded independently of $R$. To see this, consider $N.x\in Y_-(R)$, where $x$ is the  representative described before with $x_1$ the only negative coordinate. Let $\tilde{x}_1 = -x_1$ and $\tilde{x}_i = x_i$ for $i\geq 2$ be the coordinates of $\tilde{x}$. The parametrization $x \to \tilde{x}$ is obviously 1:1 and
\begin{equation*}
\tilde{x}_1^2 + \ldots \tilde{x}_n^2 + a\tilde{x}_1 \tilde{x}_2 \ldots \tilde{x}_n = k.
\end{equation*}
Because all the $\tilde{x}_i > 0$ and $a\geq 1$, this equation has no solutions when $k\leq0$ and only finitely many when $k>0$, with a bound depending only on $n$ and $k$. In any case, this shows $|Y_-(R)| = O_{n,k}(1)$.
 
Since $Y_+(R)$ is parametrized 1:1 by $(V(\Z_+)-\E)\cap B(R)$, the previous arguments combine to show
\begin{align*}
 | (V(\Z)-\E) \cap B(R) | =  |X(R)|+|X_0 \cap B(R) |  &= 2^{n-1} |N \backslash X(R) | + O_{n,k}(1)  \\
 & =  2^{n-1} ( |Y_+(R)| + |Y_-(R)| ) + O_{n,k}(1) \\
 &= 2^{n-1} | (V(\Z_+) -\E ) \cap B(R) | + O_{n,k}(1).
 \end{align*}

\subsubsection*{Infinite descent}

The following proposition says that outside of a compact set, the
effects of the moves $m_{i}$ on the maximal entries of unexceptional
Markoff-Hurwitz tuples are at least somewhat predictable. This is
a very special feature of the Diophantine equation (\ref{eq:markoff-hurwitzeq})
that will allow us to count solutions.
\begin{prop}
\label{prop:mhmove-dynamics}Suppose $k \in \Z$. There is a compact
set $K_{0}=K_{0}(n,a,k)$ such that for unexceptional $x\in V(\Z_{+})-K_{0}$
the following hold:
\begin{enumerate}
\item \label{enu:decrease-largest}If $x_{j}$ is the largest coordinate
of $x$ then the largest entry of $m_{j}(x)$ is smaller than $x_{j}$,
that is, $(m_{j}(x))_{i}<x_{j}$ for all $i.$
\item \label{enu:The-largest-entry-unique}The largest entry of $x$ appears
in exactly one coordinate.
\item \label{enu:new-largest} If $x_{j}$ is not the largest coordinate
of $x$ then it becomes the largest after the move $m_{j}$, that
is, $(m_{j}(x))_{j}>(m_{j}(x))_{i}$ for all $i\neq j.$ (This property
holds for all $x\in V(\Z_{+})$.)
\item \label{enu:dinstinct-forward}If $x_{j}$ is not the largest coordinate
of $x$, then the number of distinct entries of $m_{j}(x)$ is at
least the number of distinct entries of $x.$ In particular, if $x$
has distinct entries then $m_{j}(x)$ has distinct entries.
\item \label{enu:preserve-positivity} Every move $m_j$ maps $V(\Z_+)-K_0$ into $V(\Z_+)$.
\end{enumerate}
The compact $K_{0}$ can be taken to be a closed ball about the origin
in the $\ell^{\infty}$ norm on $\R^{n},$ and the result still holds
after increasing the radius of $K_{0}.$
\end{prop}

\begin{proof}[Proof of Proposition \ref{prop:mhmove-dynamics}]

\emph{Part \ref{enu:decrease-largest}. }Suppose without loss of generality
that $x_{1}\leq x_{2}\leq\ldots\leq x_{n-1}\leq x_{n}$. Adapting
a proof of Cassels from \cite[pg. 27]{CASSELS}, consider the quadratic
polynomial in $x_{n}$ given by 
\[
f(T)=T^{2}-ax_{1}x_{2}\ldots x_{n-1}T+x_{1}^{2}+x_{2}^{2}+\ldots+x_{n-1}^{2}-k.
\]
Then $f$ has roots at $x_{n}$ and $x'_{n}$ where $x'_{n}$ is the
last entry of $m_{n}(x)$. The conclusion of Part \ref{enu:decrease-largest}
holds unless
\[
x_{n-1}\leq x_{n}\leq x'_{n}
\]
or 
\[
x'_{n}<x_{n-1}=x_{n}.
\]
In either case, since the coefficient of $T^{2}$ is positive it follows
that $f(x_{n-1})\geq0.$ Then
\begin{eqnarray*}
0 & \leq & f(x_{n-1})=-ax_{1}x_{2}\ldots x_{n-1}^{2}+x_{1}^{2}+x_{2}^{2}+\ldots+2x_{n-1}^{2}-k.\\
 & \leq & (n-ax_{1}x_{2}\ldots x_{n-2})x_{n-1}^{2}-k
\end{eqnarray*}
so
\[
ax_{1}x_{2}\ldots x_{n-2} + \frac{k}{x_{n-1}^{2}} \leq n.
\]
By an easy argument (cf. Section \ref{incK}) it is possible to increase the radius of $K_0$  so that for $x \in V(\Z_+) - K_0$ ordered as we assume,  $x_{n-1} \geq \left(\frac{x_n}{2a}\right)^{\frac{1}{n-1}}$ . In particular, we can increase the radius of $K_0$ so that under the ongoing assumptions on $x$, $x_{n-1}^2 > |k|$.  Then, since $a, n, x_1, x_2, \dots, x_{n-1}$ are positive integral, we have
\[
ax_{1}x_{2}\ldots x_{n-2} \leq n.
\]
This
means there are a finite number of possibilities for $x_{1},x_{2},\ldots,x_{n-2}.$ 

In the case $x'_{n}\geq x_{n}$ this directly implies
\[
ax_{1}x_{2}\ldots x_{n-2}x_{n-1}-x_{n}\geq x_{n}
\]
so 
\[
ax_{1}x_{2}\ldots x_{n-2}x_{n-1}x_{n}\geq2x_{n}^{2}.
\]
Then from (\ref{eq:markoff-hurwitzeq}) 
\[
x_{n}^{2}\leq x_{1}^{2}+\ldots+x_{n-2}^{2}+x_{n-1}^{2}-k
\]
and it follows that
\[
(x_{n}+x_{n-1})(x_{n}-x_{n-1})\leq x_{1}^{2}+\ldots+x_{n-2}^{2}-k.
\]
If $x_{n}-x_{n-1}>0$ then the finite number of possibilities for
$x_{1},x_{2},\ldots,x_{n-2}$ yield a finite number of possible $x$. 

The alternative is that $x_{n}=x_{n-1}$, and the following logic
also applies to the case $x'_{n}<x_{n-1}=x_{n}.$ Then $x_{n}$ is
a root of one of finitely many quadratic polynomials

\[
(2-ax_{1}\ldots x_{n-2})x_{n}^{2}+x_{1}^{2}+\ldots+x_{n-2}^{2}-k=0.
\]
Again, this yields finitely many possibilities for $x$ aside from
those having $x_{1},\ldots,x_{n-2}$ such that $2-ax_{1}\ldots x_{n-2}=0$
and $x_{1}^{2}+\ldots+x_{n-2}^{2}-k=0.$ Note that if $k \leq 0$  we have
exhausted the possibilities. Otherwise we must have either $a=1$
and $k=(n-3)1+4$ in which case

\[
x_{1}=x_{2}=\ldots=x_{n-3}=1,\quad x_{n-2}=2,
\]
or $a=2$ and $k=n-2,$ in which case

\[
x_{1}=x_{2}=\ldots=x_{n-2}=1.
\]
These are precisely the 
fundamental
 exceptional solutions that are ruled out by
hypothesis. Therefore for any given $n,a,k$ only finitely many unexceptional
$x$ do not satisfy Part \ref{enu:decrease-largest} of the Proposition.

\emph{Part \ref{enu:The-largest-entry-unique}. }If the largest entry
of $x$ is not unique then performing the move at one of the largest
entries does not decrease the largest entry, contradicting Part \ref{enu:decrease-largest}.

\emph{Part \ref{enu:new-largest}. }Suppose $x_{1}\leq x_{2}\leq\ldots<x_{n}$
and let $x'=(x'_{1},\ldots,x'_{n})=m_{j}(x)$ with $j<n.$ The coefficient
$x'_{j}$ is given by
\[
x'_{j}-x_{n}=a\prod_{i\neq j}x_{i}-x_{j}-x_{n}=x_{n}\left(a\prod_{i\neq j,n}x_{i}-1\right)-x_{j.}
\]
If $a\geq2$ then the right hand side is $\geq x_{n}-x_{j}>0$ so
we are done. If $a=1$ and $x_{n-2}\geq2$ then we are also done by
a similar argument. 

The remaining scenario is $a=1$ and $x_{1}=x_{2}=\ldots=x_{n-2}=1.$
In this case $x$ satisfies the equation

\[
x_{n-1}^{2}+x_{n}^{2}-x_{n-1}x_{n}=k-n+2.
\]
The form on the left hand side is positive definite so only finitely
many possible solutions exist for $(x_{n-1},x_{n})$ given $n$ and
$k.$ Add these to the compact set of Part 1.

\emph{Part \ref{enu:dinstinct-forward}. }This follows from Part \ref{enu:new-largest}
since if $x'=m_{j}(x)$ as in the Proposition, then all the entries
of $x'_{i}$ with $i\neq j$ are distinct, but $x'_{j}$ is larger
than all of these.

\emph{Part \ref{enu:preserve-positivity}.}
By Part \ref{enu:new-largest} it suffices to check that we can increase the radius of $K_0$ so that for $x\in V(\Z_+) -K_0$ with  $x_1 \leq x_2 \leq \ldots \leq x_n$, $m_n(x)_n>0$. If not, one obtains $ax_1 \ldots x_{n-1} - x_n \leq 0$ from which it follows $a x_1 x_2 \ldots x_n \leq x_n^2$. The Markoff-Hurwitz equation then gives
\begin{equation}\label{eq:sum-of-squares}
x_1^2 + \ldots + x_{n-1}^2  \leq k.
\end{equation}
 As in Part \ref{enu:decrease-largest}, we can increase the radius of $K_0$ so that under the ongoing assumptions on $x$, $x_{n-1}^2\geq |k|$. It follows then that \eqref{eq:sum-of-squares} cannot occur outside of $K_0$.

\end{proof}

\begin{cor}[Infinite descent]
\label{cor:infinite-descent}Any unexceptional Markoff-Hurwitz tuple
can be algorithmically reduced to one in the compact set $K_{0}$ by a
series of Markoff-Hurwitz moves that strictly decrease maximal entries. 
\end{cor}

Corollary \ref{cor:infinite-descent} was established by Markoff \cite{MARKOFF}
in the case $n=a=3$ and $k=0.$ In that case, every Markoff triple
can be reduced to $(1,1,1)$ by a series of Markoff moves. Hurwitz
\cite{HURWITZ} showed the analogous result for $n=a>3$ and $k=0$
and showed more generally that when $k=0,$ the Markoff-Hurwitz tuples
can be reduced to a finite set of fundamental solutions. These fundamental
solutions were characterized by Baragar in \cite{BARAGAR1} whenever
$a\geq2(n-1)^{1/2}$; he also presented two different constructions yielding sequences
of equations whose sets of fundamental solutions grow without bound.

\subsection{The polynomial semigroup}

We now perform a normalization that allows us to treat all parameters
$a,k$ with a semigroup action that only depends on $n.$ For $x\in V(\Z_{+})$
let

\[
z=a^{\frac{1}{n-2}}x.
\]
Note that $a^{\frac{1}{n-2}}\geq1$ with equality if and only if $a=1.$
Then $z=(z_{1},\ldots,z_{n})$ satisfies the equation

\begin{equation}
z_{1}^{2}+z_{2}^{2}+\ldots+z_{n}^{2}=z_{1}z_{2}\ldots z_{n}+k'\label{eq:markoff-hurwitz-normalized}
\end{equation}
where

\[
k'=ka^{\frac{2}{n-2}}.
\]
Say that $z$ is exceptional/unexceptional if $x$ has the corresponding
property. We will also work with ordered tuples $z$ so that

\[
z_{1}\leq z_{2}\leq\ldots\leq z_{n}.
\]
Write $\M$ for the set of all such ordered tuples $z\in a^{\frac{1}{n-2}}\Z_{+}^{n}$
satisfying (\ref{eq:markoff-hurwitz-normalized}). Counting

\[
\M\cap B(R)
\]
 is not equivalent to counting $V(\Z_{+})\cap B(a^{-\frac{1}{n-2}}R)$
due to the presence of elements with duplicate entries. We will return
to treat this point in Section \ref{sub:Multiplicities}. Let 
\begin{eqnarray}
K & = & a^{\frac{1}{n-2}}K_{0}\label{eq:KvsK0}
\end{eqnarray}
 where $K_{0}$ is the compact set from Proposition \ref{prop:mhmove-dynamics}. 

The Markoff-Hurwitz moves $\{m_{j}\}$ induce the moves

\begin{equation}
\lambda_{j}(z_{1},\ldots,z_{n})=\left(z_{1},\ldots,\widehat{z_{j}},\ldots,z_{n},\prod_{i\neq j}z_{i}-z_{j}\right),\quad1\leq j\leq n-1,\label{eq:moves}
\end{equation}
where $\widehat{\bullet}$ denotes omission\footnote{The reason we now have $n-1$ moves instead of $n$ is that we never
perform the move that will decrease the maximal entry, therefore moving
us towards $K$. This eliminates backtracking from our `random walk'.}. Since $K$ is a closed ball about $0$ in the $\ell^{\infty}$ norm,
Part \ref{enu:new-largest} of Proposition \ref{prop:mhmove-dynamics}
implies that the $\{\lambda_{j}\}$ preserve $\M-K.$ Let $\Lambda=\Lambda(n)$
denote the semigroup of piecewise polynomial self-maps of $\C^{n}$
generated by the $\lambda_{j}.$ In Section \ref{sub:Multiplicities}
we will reduce Theorem \ref{thm:main-counting} to an orbital counting
estimate. For $z_{0}\in\M-K$ let 

\[
\Lambda.z_{0}\subset\M-K
\]
denote the orbit of $z_{0}$ under $\Lambda$. 
\begin{lem}
\label{lem:Lambda-is-free}If $z_{0}\in\M-K$ has distinct entries
then the map $\Lambda\to\M-K$ given by

\[
\lambda\mapsto\lambda(z_{0})
\]
is injective. It follows that the semigroup $\Lambda$ is free\footnote{ As a semigroup of polynomial maps.}
 on
the generators $\{\lambda_{j}\}.$ \end{lem}
\begin{proof}
For the first part, if the map is not injective then at some point
there must be $\lambda_{1}\in\Lambda$ and some $j_{1}\neq j_{2}$
such that

\begin{equation}
\lambda_{j_{1}}\lambda_{1}(z_{0})=\lambda_{j_{2}}\lambda_{1}(z_{0}).\label{eq:clash}
\end{equation}
Since by Proposition \ref{prop:mhmove-dynamics}, Part \ref{enu:dinstinct-forward}
the entries of $\lambda_{1}(z_{0})$ are distinct we find $z=\lambda_{1}(z_{0})$
with distinct entries so that $\lambda_{j_{1}}z=\lambda_{j_{2}}z.$
But this cannot be the case since e.g. the sets $\{z_{1},\ldots,\widehat{z_{j_{1}}},\ldots,z_{n}\}$
and $\{z_{1},\ldots,\widehat{z_{j_{2}}},\ldots,z_{n}\}$ are not the
same.

For the second part it is enough to find some $a$ and $k$ so that
there is a point in $V(\Z_{+})-K$ with all entries distinct. Then
freeness of $\Lambda$ follows from the first part of the proof. 

Given $n$ we first choose some $a$ and $k$ so that $V(\Z_{+})$
contains an infinite orbit. For example, the orbit of $(1,1,\ldots,1)$
in the case $a=n$ and $k=0$ is infinite and contains no exceptional
points. Then we may find $z_{0}$ outside $K$ with distinct entries,
since it is possible to 
increase
the number of distinct entries by application
of 
$\lambda_i$, using  Proposition \ref{prop:mhmove-dynamics},
Part \ref{enu:new-largest}. 
\end{proof}

\subsection{Multiplicities\label{sub:Multiplicities}}

In the rest of the paper we will count in orbits of the free semigroup
$\Lambda.$ It is extremely useful to be able to work with a fixed free
semigroup for each $n.$ The cost of this, however, is that $\Lambda$
acts on ordered tuples. Since the original problem was to count points
in $V(\Z_{+})$ we therefore need to take into account the multiplicity
of the $\order$ map $V(\Z_{+})\to\M$.

This is best done in relation to the moves $m_{j}.$ Given
  $x\in V(\Z_{+})-K_{0},$
we say that a sequence

\[
j_{1},j_{2},j_{3},\ldots,j_{l},\ldots
\]
is \emph{admissible }for $x$ if for all $l,$ $j_{l}$ is not the
largest coordinate of 

\[
x^{(l-1)}=m_{j_{l-1}}m_{j_{l-2}}\ldots m_{j_{2}}m_{j_{1}}x.
\]
Notice then that the largest entries of $x^{(l)}$ are increasing
in $l$ and therefore $x^{(l)}\in V(\Z_{+})-K_{0}$ for all $l\geq1.$
Also, a sequence is admissible if and only if $j_{1}$ is not the
largest coordinate and $j_{l}\neq j_{l-1}$ for any $l\leq2.$ Write
$\Sigma^{*}(x)$ for the set of all finite admissible sequences for
$x.$
\begin{lem}
\label{lem:injective-walk}Given $x\in V(\Z_{+})-K_{0}$ the map $\phi_{x}:\Sigma^{*}(x)\to V(\Z_{+})$
given by 

\[
\phi_{x}(j_{1},j_{2},j_{3},\ldots,j_{l})=m_{j_{l}}m_{j_{l-1}}m_{j_{l-2}}\ldots m_{j_{2}}m_{j_{1}}x
\]
is injective. Note this is regardless of whether $x$ has duplicate
entries. Moreover, for any $x$, $x'\in V(\Z_{+})-K_{0},$ the images
of $\phi_{x}$ and $\phi_{x'}$ are disjoint unless either $x'\in\image(\phi_{x})$
or $x\in\image(\phi_{x'})$.
\end{lem}

\begin{proof}
It is clear from Proposition \ref{prop:mhmove-dynamics}, Part \ref{enu:new-largest}
that the $m_{j_{1}}x$ with $j_{1}$ admissible are distinct. It is
then enough to show $\phi_{x}$ is injective to show that there are
no $x\neq x'\in V(\Z_{+})-K_{0}$ and $j,j'$ admissible for the respective
$x,x'$ so that $m_{j}(x)=m_{j'}(x').$ But since $m_{j}(x)$ has
a distinct largest entry by Proposition \ref{prop:mhmove-dynamics}
Part \ref{enu:The-largest-entry-unique}, it has to be the case that
$j=j'$. Then applying $m_{j}$ gives $x=x'.$

Now suppose $x'\notin\image(\phi_{x})$ and $x\notin\image(\phi_{x'}).$
If $\image(\phi_{x})\cap\image(\phi_{x'})\neq\emptyset$ then at some
point there must have been $x^{(3)}\neq x^{(4)}\in V(\Z_{+})-K_{0}$
and $j,j'$ admissible for $x^{(3)},x^{(4)}$ respectively so that
$m_{j}(x^{(3)})=m_{j'}(x^{(4)})$. But we have already established
this cannot happen.
\end{proof}

\begin{lem}
\label{lem:bijection between paths}Let $x\in V(\Z_{+})-K_{0}$ and
$z=\order(a^{\frac{1}{n-2}}x)$ the corresponding element of $\M-K.$
There exists a bijection

\[
\Theta_{x}:\Sigma^{*}(x)\to\Lambda
\]
that is an intertwiner for the map $x'\mapsto z(x')=\order(a^{\frac{1}{n-2}}x')$
in the sense that 

\[
\Theta_{x}(j_{1},j_{2},\ldots,j_{l}).z(x)=z(\phi_{x}(j_{1},j_{2},j_{3},\ldots,j_{l}))
\]
for all $(j_{1},\ldots,j_{l})\in\Sigma^{*}(x).$
\end{lem}

\begin{proof}
We'll show for all $x'$ there is a one to one correspondence between
the admissible sequences $(j)$ of length $1$ and $\{\lambda_{j}:1\leq j\leq n-2\}$
so that $\Theta_{x}(j).z(x)=z(\phi_{x'}(j)).$ This is clear if $x'_{1}\leq x'_{2}\leq\ldots<x'_{n}$
is ordered (send $j\mapsto\lambda_{j})$. Otherwise pick an ordering
of $x'.$ The general result follows by repeating this process.
\end{proof}

Lemma \ref{lem:injective-walk} implies that the set $V(\Z_{+})$
decomposes into the finite set $K_{0}$ and a finite number of orbits
of the form
\[
\phi_{x^{(0)}}(\Sigma^{*}(x^{(0)})).
\]
Each one of these orbits has either all its points exceptional or unexceptional. Since we assume throughout the rest of the paper that $V(\Z)-\E$ is infinite, it follows that the collection  $\mathcal{U}$ of unexceptional basepoints $x^{(0)}$ is finite and nonempty. Summing up,
\[
V(\Z_+) -\E -K_0 = \coprod_{x^{(0)} \in \mathcal{U} } \phi_{x^{(0)}}(\Sigma^{*}(x^{(0)})),
\]
so 
\begin{eqnarray*}
|(V(\Z_+) -\E) \cap B(R)| & =& O_{n,a,k}(1) + \sum_{x^{(0)} \in \mathcal{U} } \sum_{s \in \Sigma^*(x^{(0)})} \mathbf{1}\left\{  \max( \phi_{x^{(0)}}(s) )\leq R\right\}  \\
&=& O_{n,a,k}(1) + \sum_{x^{(0)} \in \mathcal{U} } \sum_{s \in \Sigma^*(x^{(0)})} \mathbf{1}\left\{  z( \phi_{x^{(0)}}(s) )_n\leq a^\frac{1}{n-2} R\right\} .
\end{eqnarray*}
Applying Lemma \ref{lem:bijection between paths} to the above sum, one obtains
\[
 O_{n,a,k}(1) + \sum_{x^{(0)} \in \mathcal{U} } \sum_{\lambda\in \Lambda} \mathbf{1}\left\{  (\lambda.z(x^{(0)}))_n\leq a^\frac{1}{n-2} R\right\} .
\]
Therefore,
Theorem \ref{thm:main-counting} will follow from asymptotic estimates
for the quantity

\begin{equation}
\sum_{\lambda\in\Lambda}\mathbf{1}\left\{ (\lambda.z^{(0)})_{n}\leq R\right\} \label{eq:count-form}
\end{equation}
where $z^{(0)}\in z(\mathcal{U}) \subset \M-K.$ These estimates are taken up in the next
section. We draw the reader's attention to the fact that the count
is over $\Lambda$ and not over $\M.$

\subsection{Increasing the size of $K$}\label{incK}

Before we begin the count we increase the size of $K.$ Recall that
$K$ and $K_{0}$ are balls with center $0$ in the $\ell^{\infty}$
norm with radii coupled by (\ref{eq:KvsK0}) and that we are free
to increase their radii (maintaining the relationship (\ref{eq:KvsK0})).
The following can be thought of as regularizing the dynamics of $\M$
at a fixed scale depending on $n,a,k.$ We state our requirements
in terms of $z=(z_{1},\ldots,z_{n}).$ 

First we make sure $z_{n-1}$ is reasonably large compared to $z_{n}.$
Suppose $z_{n-1}\leq cz_{n}^{\frac{1}{n-1}}$. Then $z_{1}\leq z_{2}\leq\ldots\leq z_{n-1}\leq cz_{n}^{\frac{1}{n-1}}.$
Then (\ref{eq:markoff-hurwitz-normalized}) gives

\[
z_{n}^{2}\leq c^{n-1}z_{n}^{2}+k'
\]
 which is a contradiction for $c<1$ and $z_{n}$ large enough depending
on $k'$. We increase the radius of $K$ so that 
\begin{equation}
z_{n-1}\geq\frac{1}{2}z_{n}^{\frac{1}{n-1}}\label{eq:zn-1-not-too-small}
\end{equation}
 for all $z\in\M-K.$

Now we make sure $z_{n}$ is large enough so certain inequalities
hold. Note that

\begin{equation}
\frac{(n-1)\log(1-2z_{n}^{-1/(n-1)})-(n-1)\log2}{\log z_{n}}\label{eq:function-of-z_n}
\end{equation}
tends to $0$ as $z_{n}\to\infty.$ So we increase the radius of $K$
so that 
\begin{equation}
\eqref{eq:function-of-z_n}\geq-1/2\label{eq:function-of-zn-lower bound}
\end{equation}
for all $z\in\M-K.$ It will also be convenient for the sake of simplifying
arguments to assume that 
\begin{equation}
z_{n}\geq10\label{eq:ten}
\end{equation}
for all $z\in\M-K.$ 
Furthermore by increasing the radius of $K$, using  \eqref{eq:zn-1-not-too-small} we can also ensure
\begin{equation}
z_{n-1}> 2 \label{eq:2lowerbound}
\end{equation}
and 
\begin{equation}
z_{1}^{2}+\ldots+z_{n-1}^{2}-k'\geq0\label{eq:C(z)-minus-kdash-nonnegative}
\end{equation}
for $z \in \M - K$.

\section{Converting the linear count to the nonlinear count}

\label{sec:converting-linear-to-nonlinear}

\subsection{Acceleration }

\label{sub:acceleration}

In the last Section \ref{sec:MH-tuples-and-moves} we reduced our
Main Theorem \ref{thm:main-counting} to obtaining an asymptotic for
the count

\begin{equation}
\sum_{\lambda\in\Lambda}\mathbf{1}\left\{ (\lambda.z^{(0)})_{n}\leq R\right\} \label{eq:archimedeanball}
\end{equation}
where $z^{(0)}$
 is one of a finite set of unexceptional points in
$\M-K$. For the rest of the paper we view $z^{(0)}$ as fixed.

There is a general framework in which to count over the tree-like
$\Lambda,$ called the \emph{renewal method. }This was first used
in counting by Lalley \cite{Lalley88} in the setting of self-similar fractals and subsequently 
extended by him \cite{LALLEY} to the setting of Schottky groups. The essence of the method is a recursion over $\Lambda.$ Our departure
from other uses of renewal in counting problems is that we perform
what we call \emph{acceleration. }Concretely, we replace the generators
$\{\lambda_{j}:1\leq j\leq n\}$ of $\Lambda$ with the countably
infinite set of generators
\[
S=S_{\Lambda}=\left\{ \lambda_{n-1}^{A}\lambda_{j}\::\:A\in\Z_{\geq0},\:1\leq j\leq n-2\:\right\} .
\]
It is easy to see that $S_{\Lambda}$ are free generators for the
subsemigroup 
\[
\Lambda'=\cup_{j=1}^{n-2}\Lambda.\lambda_{j}\subset\Lambda
\]
that contains the words beginning with $\lambda_{j}$, $1\leq j\leq n-2.$
This acceleration is crucial for our method and has two advantages:
\begin{enumerate}
\item The quality of our fitting the nonlinear count for $\Lambda$ to a
linear count to $\Gamma$ depends on the size of the quantity
\[
\alpha(z)=\prod_{j=1}^{n-2}z_{j},
\]
cf. Lemma \ref{lem:Log-vs-f} below. This quantity can be small for
long words with respect to the generators $\{\lambda_{j}\}$, because
$\lambda_{n-1}$ does not alter $\alpha(z).$ On the other hand, we
prove in Lemma \ref{lem:growth-of-alpha-in-L} that $\alpha(z)$ grows
doubly exponentially in the word length with respect to the generators
$S_{\Lambda}$.
\item When we eventually arrive at the dynamics of $\G'$ on $P(\R_{\geq0}^{n}),$
the unaccelerated system would be non-uniformly contracting and therefore
we could not expect there to be a finite invariant measure for this
system. On the other hand, the acceleration we perform leads to uniformly
contracting dynamics (cf. Proposition \ref{prop:uniformly-contracting})
and in turn to a nice description of the invariant measure and leading
eigenfunction for the transfer operator in the Ruelle-Perron-Frobenius
Theorem (Theorem \ref{thm:RPF}).
\end{enumerate}
Now, the orbit $\Lambda.z^{(0)}$ breaks up into the countable union
of orbits

\begin{equation}
\Lambda.z^{(0)}=\bigcup_{A_{0}=0}^{\infty}\Lambda'.\lambda_{n-1}^{A_{0}}z^{(0)}.\label{eq:breaking-up-Lambda}
\end{equation}
It is clear that an asymptotic formula for (\ref{eq:archimedeanball})
is equivalent to an asymptotic formula for

\begin{equation}
M_{0}(z,a)\equiv\sum_{\lambda\in\Lambda\cup\{e\}}\mathbf{1}\{\log\log(\lambda.z)_{n}-\log\log z_{n}\leq a\}\label{eq:M0defn}
\end{equation}
when $z=z^{(0)}.$ On the other hand, our methods can prove an asymptotic
formula for the following quantity

\begin{equation}
M(z,a)=\sum_{\lambda\in\Lambda'\cup\{e\}}\mathbf{1}\{\log\log(\lambda.z)_{n}-\log\log z_{n}\leq a\}\label{eq:M-definition}
\end{equation}
for arbitrary unexceptional $z\in\M-K.$ Precisely, we will obtain
the following proposition.
\begin{prop}
\label{prop:orbital-counting-Lambda'-main-prop}For all unexceptional
$z\in\M-K$ there is a positive constant $c_{\star}$ such that as
$a\to\infty,$

\[
M(z,a)=e^{\beta a}(c_{\star}(z)+o(1)),
\]
where $\beta>1$ is the constant from Theorem \ref{thm:Baragar-theorem}
and the rate of decay in the small $o$ does not depend on $z.$ Moreover,
the $c_{\star}(z)$ have a uniform bound depending only on $n.$
\end{prop}

The proof of Proposition \ref{prop:orbital-counting-Lambda'-main-prop}
will occupy the rest of this Section. Before beginning, we show how
Proposition \ref{prop:orbital-counting-Lambda'-main-prop} implies
our main Theorem \ref{thm:main-counting}. This passage relies on
the following elementary lemma.
\begin{lem}
\label{lem:lambda-n-1-growth}For unexceptional $z\in\M-K$ we have

\[
(\lambda_{n-1}^{A}z)_{n}\geq2^{A}z_{n}.
\]
\end{lem}
\begin{proof}
One can calculate easily that for $z=(z_{1},\ldots,z_{n}),$ $\lambda_{n-1}^{A}z$
is obtained by $A$ applications of the matrix

\[
g_{\alpha(z)}=\left(\begin{array}{cc}
0 & 1\\
-1 & \alpha(z)
\end{array}\right)
\]
to the last two entries of $z,$ where $\alpha(z)=\prod_{j\leq n-2}z_{j}.$
This quantity will appear repeatedly in the rest of the paper. If
$z=z(x)$ with $x_{1}\leq x_{2}\leq\ldots\leq x_{n}$ then 

\[
\alpha(z)=ax_{1}x_{2}\ldots x_{n-2}\in\Z_{+}.
\]
If $\alpha(z)=1$ then this matrix is torsion and this contradicts
the maximal entries of $\lambda_{n-1}^{A}z$ growing with $A$ (since
$z\in\M-K)$. If $\alpha(z)=2$ then $z$ must be an exceptional solution.
Otherwise $\alpha(z)\geq3$ and if we let 
$Z_{A}=(\lambda_{n-1}^{A}z)_n$
then the $Z_{A}$ satisfy the recurrence

\[
Z_{A+1}=\alpha(z)Z_{A}-Z_{A-1}\geq2Z_{A}.
\]
Therefore $(\lambda_{n-1}^{A}z)_{n}\geq2^{A}z_{n}.$
\end{proof}

\begin{proof}[Proof of Theorem \ref{thm:main-counting} given Proposition \ref{prop:orbital-counting-Lambda'-main-prop}]

By our previous discussion it suffices to prove an asymptotic formula
for $M_{0}(z^{(0)},a)$ for a fixed $z^{(0)}.$ But using (\ref{eq:breaking-up-Lambda})
gives

\begin{equation}
M_{0}(z^{(0)},a)=\sum_{A_{0}=1}^{\infty}M(\lambda_{n-1}^{A_{0}}z^{(0)},a-\log\log(\lambda_{n-1}^{A_{0}}z^{(0)})_{n}+\log\log z_{n}^{(0)}).\label{eq:tempsum}
\end{equation}
By using Lemma \ref{lem:lambda-n-1-growth}, the value $A_{0}=A_{\max}$
where $a-\log\log(\lambda_{n-1}^{A_{0}}z^{(0)})_{n}+\log\log z_{n}^{(0)}$
first becomes negative is bounded by 
\[
A_{\max}\leq\frac{\log z_{n}^{(0)}e^{a}}{\log2}.
\]
Let the small $o$ term in Proposition \ref{prop:orbital-counting-Lambda'-main-prop}
be bounded in absolute value by a positive function $F(a)$ that tends to $0$
as $a\to\infty$. Let $\kappa$ be a small positive constant
to be chosen. The $A_{0}$ such that $a-\log\log(\lambda_{n-1}^{A_{0}}z^{(0)})_{n}+\log\log z_{n}^{(0)}\geq\kappa a$
contribute 

\[
\log z_{n}^{(0)}e^{\beta a}\sum_{A_{0}:a-\log\log(\lambda_{n-1}^{A_{0}}z^{(0)})_{n}+\log\log z_{n}^{(0)})\geq\kappa a}\frac{c_{\star}(\lambda_{n-1}^{A_{0}}z^{(0)})}{(\log(\lambda_{n-1}^{A_{0}}z^{(0)})_{n})^{\beta}}(1+O(\sup_{a'\geq\kappa a}F(a)).
\]
to (\ref{eq:tempsum}) by Proposition \ref{prop:orbital-counting-Lambda'-main-prop}.
Furthermore, by Lemma \ref{lem:lambda-n-1-growth}, 

\[
\sum_{A_{0}:a-\log\log(\lambda_{n-1}^{A_{0}}z^{(0)})_{n}+\log\log z_{n}^{(0)})\geq\kappa a}\frac{c_{\star}(\lambda_{n-1}^{A_{0}}z^{(0)})}{(\log(\lambda_{n-1}^{A_{0}}z^{(0)})_{n})^{\beta}}\leq\sum_{A_{0}}\frac{c_{\star}(\lambda_{n-1}^{A_{0}}z^{(0)})}{(A_{0}\log2)^{\beta}}
\]
 converges to some limit $c_{\infty}(z^{(0)})$ as $a\to\infty,$
using $\beta>1.$ Therefore the terms we have discussed so far give
a contribution of

\[
\log z_{n}^{(0)}c_{\infty}(z^{(0)})e^{\beta a}(1+o(1))
\]
 to $M_{0}(z^{(0)},a)$ via (\ref{eq:tempsum}).

For the remaining $A_{0}$ such that $a-\log\log(\lambda_{n-1}^{A_{0}}z^{(0)})_{n}+\log\log z_{n}^{(0)}<\kappa a$
we use Proposition \ref{prop:orbital-counting-Lambda'-main-prop}
in a coarser way to get $M(z,a)\leq Ce^{\beta a}$ for some constant
$C,$ uniformly over unexceptional $z\in\M-K.$ Then any remaining
$A_{0}$ contributes at most $Ce^{\beta\kappa a}$ to (\ref{eq:tempsum}).
Therefore the remaining contributions are in total at most 

\[
A_{\max}Ce^{\beta\kappa a}\leq\frac{\log z_{n}^{(0)}Ce^{(1+\beta\kappa)a}}{\log2}
\]
which is negligible when $1+\beta\kappa<\beta,$ and we can find such
a $\kappa$ since $\beta>1.$
\end{proof}

\subsection{The renewal equation for $M$}

We now take up the proof of Proposition \ref{prop:orbital-counting-Lambda'-main-prop}.
While the statement of Proposition  \ref{prop:orbital-counting-Lambda'-main-prop} is uniform over all unexceptional $z\in \M- K$, our previous arguments show that the unexceptional elements of $\M-K$ break up into finitely many orbits of $\Lambda$. Therefore it is sufficient for us to establish Proposition  \ref{prop:orbital-counting-Lambda'-main-prop} for $z=\lambda_{0}z^{(0)}$, where $z^{(0)} \in z(\mathcal{U})$ is a fixed unexceptional basepoint and $\lambda_0$ is an arbitrary element of $\Lambda$. We therefore view $z^{(0)}$ as fixed from now on, and we will prove Proposition \ref{prop:orbital-counting-Lambda'-main-prop} for $z=\lambda_{0}z^{(0)}$, with uniformity over $\lambda_0\in \Lambda$.

We now describe the renewal equation, for which we need some new concepts.
Define the \emph{shift $s:\Lambda'\to\Lambda'\cup\{e\}$ }by

\[
s(\lambda_{n-1}^{A_{l}}\lambda_{j_{l}}\lambda_{n-1}^{A_{l-1}}\lambda_{j_{l-1}}\ldots\lambda_{n-1}^{A_{2}}\lambda_{j_{2}}\lambda_{n-1}^{A_{1}}\lambda_{j_{1}})\equiv\lambda_{n-1}^{A_{l-1}}\lambda_{j_{l-1}}\ldots\lambda_{n-1}^{A_{2}}\lambda_{j_{2}}\lambda_{n-1}^{A_{1}}\lambda_{j_{1}}.
\]
Now extend this definition so that $s(\lambda\lambda_{0})=s(\lambda)\lambda_{0}$
for all $\lambda\in\Lambda'$\emph{ }and\emph{ }$\lambda_{0}\in\Lambda\cup\{e\}.$
We define the \emph{distortion function $\tau_{\star}:\Lambda'.(\Lambda\cup\{e\})\to\R_{\geq0}$
}by

\[
\tau_{\star}(\lambda)\equiv\log\log(\lambda.z^{(0)})_{n}-\log\log(s(\lambda).z^{(0)})_{n}.
\]
This depends on the constant $z^{(0)}.$ One also has the iterated
version of distortion

\begin{equation}
\tau_{\star}^{N}(\lambda)=\sum_{p=0}^{N-1}\tau_{\star}(s^{p}(\lambda))=\log\log(\lambda.z^{(0)})_{n}-\log\log(s^{N}(\lambda).z^{(0)})_{n}.\label{eq:tau-star-iterated}
\end{equation}
for any $\lambda\in s^{-N}(\Lambda).$ The \emph{renewal equation
}for $M$ is then

\begin{equation}
M(\lambda z^{(0)},a)=\sum_{\lambda'\in S_{\Lambda}}M(\lambda'\lambda z^{(0)},a-\tau_{\star}(\lambda'\lambda))+\mathbf{1}\{0\leq a\}\label{eq:M-renewal}
\end{equation}
for all $\lambda\in\Lambda.$ Note that the summation above is finite
since the $\lambda'$ act to strictly increase maximal entries in
$\M$.

\subsection{Iteration}

The eventual goal is to compare the asymptotics of $M(\lambda z^{(0)},a)$
to those of an analogous quantity for the linear semigroup $\Gamma$
introduced in the Introduction. Before this happens, a regularization
must occur. In our approach\footnote{In Zagier's approach in \cite{ZAGIER} for the case $n=a=3$, there
is a special mapping arising from the close connection between the
Markoff equation and hyperbolic geometry. This mapping offers a much
better fit to the linear semigroup count than is available in general.
See footnote \ref{fn:approximation-quality-footnote} for more on
this.}, the quality of the comparison to the linear semigroup depends on
the size of 

\[
\alpha(z_{1},\ldots,z_{n})=\prod_{j\leq n-2}z_{j}.
\]
It is clear that no $\lambda\in\Lambda$ decreases $\alpha(z).$ To
pass to the case that $\alpha(\lambda'.z^{(0)})$ is large, we iterate
the renewal equation (\ref{eq:M-renewal}) $L$ times. This yields

\begin{equation}
M(\lambda z^{(0)},a)=\sum_{\lambda':s^{L}(\lambda')=\lambda}M(\lambda'z^{(0)},a-\tau_{\star}^{L}(\lambda'))+\sum_{l=1}^{L-1}\sum_{\lambda':s^{l}(\lambda')=\lambda}\mathbf{1}\left\{ \tau_{\star}^{l}(\lambda')\leq a\right\} +\mathbf{1}\left\{ 0\leq a\right\} ,\label{eq:M-iterated-renewal}
\end{equation}
recalling the definition of $\tau_{\star}^{L}$  from (\ref{eq:tau-star-iterated}).
We now show that for suitable $L$ the last two summations in (\ref{eq:M-iterated-renewal})
are negligible. The following lemma is used at several points in the
rest of the paper.
\begin{lem}
\label{lem:iterated-error-term-small}There are constants $c_{0}$
and $c_{1}$ depending only on $n$ such that for all $L\in\N,$ $x\geq0$
\begin{equation}
\sum_{\lambda':s^{L}(\lambda')=\lambda}\mathbf{1}\left\{ \:\tau_{\star}^{L}(\lambda')\leq x\:\right\} \leq c_{1}^{L}(c_{0}+x)^{L}e^{x}.\label{eq:combinatorial-bound-for-tau-star-leq-x}
\end{equation}
As a consequence, for any $\delta>0$, there is $c=c(\delta)>0$ such
that when $L=\Big\lceil\frac{ca}{\log a}\Big\rceil$ one has

\begin{equation}
\sum_{l=1}^{L-1}\sum_{\lambda':s^{l}(\lambda')=\lambda}\mathbf{1}\left\{ \tau_{\star}^{l}(\lambda')\leq a\right\} =O(e^{(1+\delta)a})\label{eq:iterated-tail-bound}
\end{equation}
and
\begin{equation}
c_{1}^{L}(c_{0}+x)^{L}\leq e^{\delta x}\label{eq:bound-for-IBP}
\end{equation}
for all $x\geq a/2.$\end{lem}
\begin{proof}
For the first part of this proof, let $\tilde{\lambda}$ denote an arbitrary element of $\Lambda'$, and $z := \tilde{\lambda}.z^{(0)}$.
The proof of Lemma \ref{lem:lambda-n-1-growth} can be easily adapted
to show that  for arbitrary unexceptional $z' \in\M-K$ 
\[
(\lambda_{n-1}^{A}z')_{n}\geq(\alpha(z')-1)^{A}z'_{n}.
\]
This gives,  setting $z' = \lambda_j z$ 
\begin{eqnarray*}
\tau_{\star}(\lambda_{n-1}^{A}\lambda_{j}\tilde{\lambda}) & = & \log\log(\lambda_{n-1}^{A}\lambda_{j}z)_{n}-\log\log z_{n}\\
 & \geq & \log\log((\alpha(\lambda_{j}(z))-1)^{A}(\lambda_{j}z)_{n})-\log\log z_{n}.
\end{eqnarray*}
Now,
\begin{eqnarray*}
\alpha(\lambda_{j}(z)) & = & \prod_{1\leq i\leq n-1,i\neq j}z_{i}=a\prod_{1\leq i\leq n-1,i\neq j}x_{i}
\end{eqnarray*}
where $x$ is an integer solution to (\ref{eq:markoff-hurwitzeq})
corresponding to $z$. By using (\ref{eq:zn-1-not-too-small}) we
get $\alpha(\lambda_{j}(z))\geq z_{n-1}\geq\frac{1}{2}z_{n}^{\frac{1}{n-1}}$
and hence using $(\lambda_{j}z)_{n}\geq z_{n}$,
\begin{eqnarray}
\tau_{\star}(\lambda_{n-1}^{A}\lambda_{j} \tilde{\lambda} ) & \geq & \log\log(\frac{1}{2^{A}}z_{n}^{A/(n-1)}(1-2z_{n}^{-1/(n-1)})^{A}z_{n})-\log\log z_{n}\nonumber \\
 & \geq & \log\left(1+\frac{A}{n-1}\left(1+\frac{(n-1)\log(1-2z_{n}^{-1/(n-1)})-(n-1)\log2}{\log z_{n}}\right)\right)\nonumber \\
 & \geq & \log\left(1+\frac{A}{2(n-1)}\right).\label{eq:tau-star-lower-one}
\end{eqnarray}
where the last inequality is by the previously prepared (\ref{eq:function-of-zn-lower bound}).

Now, if $\lambda=\lambda_{n-1}^{A_{l}}\lambda_{j_{l}}\lambda_{n-1}^{A_{l-1}}\lambda_{j_{l-1}}\ldots\lambda_{n-1}^{A_{2}}\lambda_{j_{2}}\lambda_{n-1}^{A_{1}}\lambda_{j_{1}}$
then by $l$ applications of (\ref{eq:tau-star-lower-one}) we get
\begin{eqnarray*}
\tau_{\star}^{l}(\lambda) & = & \sum_{p=0}^{l-1}\tau_{\star}(s^{p}(\lambda))\geq\sum_{q=1}^{l}\log\left(1+\frac{A_{q}}{2(n-1)}\right).
\end{eqnarray*}
Therefore the number of $\lambda'$ that can contribute to (\ref{eq:combinatorial-bound-for-tau-star-leq-x})
is bounded by the size of the set

\begin{equation}
\left\{ (A_{1},A_{2},A_{3},\ldots,A_{L})\in\Z_{\geq0}^{L}:\sum_{q=1}^{L}\log\left(1+\frac{A_{q}}{2(n-1)}\right)\leq x\right\} .\label{eq:big-set}
\end{equation}
times the number of possible choices for $j_{1},\ldots j_{L}$. The
latter can be crudely bounded by $(n-2)^{L}.$

\emph{Claim. }The size of the set in (\ref{eq:big-set}) is bounded
by $(2(n-1)(c_{0}+x))^{L}e^{x}$ for some positive constant $c_{0}$.

\emph{Proof of Claim. }We prove this by induction on $L.$ The base
case $(L=1)$ is clear. For the induction, after choosing the first
$A_{1}$ the remaining $A_{2},\ldots,A_{L}$ must satisfy
\[
\sum_{q=2}^{L}\log\left(1+\frac{A_{q}}{2(n-1)}\right)\leq x-\log\left(1+\frac{A_{1}}{2(n-1)}\right).
\]
So the size of the set in (\ref{eq:big-set}) is bounded by
\begin{align*}
 & \sum_{A_{1}=1}^{\lfloor2(n-1)e^{x}\rfloor}(2(n-1))^{L-1}\left(c_{0}+x-\log\left(1+\frac{A_{1}}{2(n-1)}\right)\right)^{L-1}e^{x}\frac{1}{1+\frac{A_{1}}{2(n-1)}}\\
\leq & (2(n-1))^{L}(c_{0}+x)e^{x}\sum_{A_{1}=1}^{\lfloor2(n-1)e^{x}\rfloor}\frac{1}{2(n-1)+A_{1}}.
\end{align*}
The final sum is within a constant $c_{0}$ of $x.$ This completes
the proof of the Claim.

So in total we obtain that the sum in (\ref{eq:combinatorial-bound-for-tau-star-leq-x})
is bounded by $c_{1}^{L}(c_{0}+x){}^{L}e^{x}$ with $c_{1}=2(n-2)(n-1).$
As for the stated consequence, we get

\[
\sum_{l=1}^{L-1}\sum_{\lambda':s^{l}(\lambda')=\lambda}\mathbf{1}\left\{ \tau_{\star}^{l}(\lambda')\leq a\right\} \ll c_{1}^{L}(c_{0}+a)^{L}e^{a}.
\]
If we choose $L\approx$ $ca/\log(1+a)$ with $c$ small enough depending
on $\delta$ we obtain our result.
\end{proof}

Since we expect $M(\lambda z^{(0)},a)\approx e^{\beta a}$ with $\beta=\beta(n)>1$,
choosing parameters as in Lemma \ref{lem:iterated-error-term-small}
gives

\begin{equation}
M(\lambda z^{(0)},a)=\sum_{\lambda':s^{L}(\lambda')=\lambda}M(\lambda'z^{(0)},a-\tau_{\star}^{L}(\lambda'))+O(e^{(1+\delta)a})\label{eq:M-iterated-expression}
\end{equation}
 and the big $O$ term is truly an error term when $\delta$ is small.
The benefits to our iteration in (\ref{eq:M-iterated-expression})
can be quantified by the following result.
\begin{lem}
\label{lem:growth-of-alpha-in-L}There is some $C>0$ such that for
all $\lambda\in\Lambda\cup\{e\}$ and $\lambda'$ such that $s^{L}(\lambda')=\lambda,$
we have both

\begin{equation}
\alpha(\lambda'z^{(0)})\geq\frac{1}{2}\exp(C\phi^{L})\label{eq:alpha-growth}
\end{equation}
and
\begin{equation}
(\lambda'z^{(0)})_{n}\geq\exp(C\phi^{L})\label{eq:last-coordinate-growth}
\end{equation}
where $\phi=\frac{1+\sqrt{5}}{2}>1$ is the golden ratio.\end{lem}
\begin{proof}
For $1\leq j\leq n-2$

\[
(\lambda_{j}z)_{n}=\prod_{i\neq j}z_{i}-z_{j}=z_{n}z_{n-1}\prod_{i\neq j,n-1,n}z_{i}-z_{j}\geq(z_{n}-1)z_{n-1}
\]
since $z_{i}\geq1$ for all $i$ and $z_{n-1}\geq z_{j}.$ So then
for any $A\geq0$

\[
(\lambda_{n-1}^{A}\lambda_{j}z)_{n}\geq(\lambda_{j}z)_{n}\geq(z_{n}-1)z_{n-1}.
\]
Then 
\begin{eqnarray*}
(\lambda_{n-1}^{A_{2}}\lambda_{j_{2}}\lambda_{n-1}^{A_{1}}\lambda_{j_{1}}z)_{n} & \geq & ((\lambda_{n-1}^{A_{1}}\lambda_{j_{1}}z)_{n}-1)(\lambda_{n-1}^{A_{1}}\lambda_{j_{1}}z)_{n-1}\\
 & \geq & ((\lambda_{n-1}^{A_{1}}\lambda_{j_{1}}z)_{n}-1)z{}_{n}
\end{eqnarray*}
using the inequality $(\lambda z)_{n-1}\geq z_{n}$ for any $\lambda\in\Lambda.$
Therefore the numbers

\[
Z_{p}=(\lambda_{n-1}^{A_{p}}\lambda_{j_{p}}\ldots\lambda_{n-1}^{A_{2}}\lambda_{j_{2}}\lambda_{n-1}^{A_{1}}\lambda_{j_{1}}z)_{n}\geq10
\]
(cf. (\ref{eq:ten})) satisfy the two stage recursive estimate $Z_{p}\geq(Z_{p-1}-1)Z_{p-2}$
for $p\geq2.$ Then an elementary argument gives the existence of
$C$ such that

\[
Z_{p}\geq\exp(C\phi^{p}).
\]
This gives the required (\ref{eq:last-coordinate-growth}).

On the other hand

\[
\alpha(\lambda_{n-1}^{A}\lambda_{j}z)\geq\alpha(\lambda_{j}z)\geq z_{n-1}\geq\frac{1}{2}z_{n}^{\frac{1}{n-1}}
\]
 where the last inequality is by (\ref{eq:zn-1-not-too-small}). The
result (\ref{eq:alpha-growth}) now follows after replacing $C$ with
a suitable smaller constant.
\end{proof}

In the sequel we choose

\[
L=\Big\lceil c\frac{a}{\log a}\Big\rceil
\]
 so that (\ref{eq:M-iterated-expression}) and \eqref{eq:bound-for-IBP}
hold with\footnote{We know by Remark \ref{rem:betabig} that $\beta\geq2.$}
\begin{equation}
\delta=\min\left(\frac{1}{10},\frac{\beta-1}{2}\right).\label{eq:delta-definition}
\end{equation}
Then for all $\lambda'z^{(0)}$ appearing in (\ref{eq:M-iterated-expression})
we have
\begin{equation}
\alpha(\lambda'z^{(0)})\geq\frac{1}{2}\exp(C\phi^{ca/\log a})\label{eq:what-alpha-is}
\end{equation}
by Lemma \ref{lem:growth-of-alpha-in-L}.

\subsection{Comparison to the linear count}

\label{sec:comparison-section}

Now we relate the terms $M(\lambda'z^{(0)},a)$ appearing in (\ref{eq:M-iterated-expression})
to orbital counting for $\Gamma$, the linear semigroup defined in
the Introduction. We begin with the expression for $M(\lambda'z^{(0)},a)$
in (\ref{eq:M-definition}). Denoting by  $S^{(N)}$ the $N$-fold
product\footnote{That is, $S^{(N)}$ is the elements of $\Lambda'$ that are a product
of $N$ generators. We extend this definition to $S^{(0)}=\{e\}.$} of the countable generating set $S$ for $\Lambda'$, then we can
write
\begin{equation}
M(\lambda'z^{(0)},a)=\sum_{N=0}^{\infty}\sum_{\lambda^{(2)}\in S^{(N)}}\mathbf{1}\{\tau_{\star}^{N}(\lambda^{(2)}\lambda')\leq a\}.\label{eq:M-lambda-prime-expression-for-comparison-to-linear}
\end{equation}
We will proceed by
\begin{enumerate}
\item \label{enu:fitting-the-point}Matching $\lambda'z^{(0)}$ with some
element of $\H\subset\R_{+}^{n}$ that is very close to\footnote{When we write $\log$ of a vector we always mean take $\log$ of each
coordinate.} $\log(\lambda'z^{(0)}).$
\item \label{enu:Matching--group-elements}Matching each $\lambda^{(2)}$
with an element $\gamma^{(2)}$ of $\Gamma$ in the obvious way.
\end{enumerate}

With Part \ref{enu:fitting-the-point} in mind, we define for $z\in\M$

\[
f(z)\equiv(\log z_{1},\log z_{2},\ldots,\log z_{n-1},\sum_{j=1}^{n-1}\log z_{i}).
\]
The reason to use this map over just taking $\log$ of coordinates
is that we expect $\log(z)$ to be very close to the hyperplane $\H$
defined in (\ref{eq:hyperplane-defn}), so we just go ahead and fit
$\log(z)$ to this plane. The following lemma (cf. Lemma 2 in \cite{ZAGIER}) says that when $\alpha(z)$
is big, $f(z)$ is a good\footnote{\label{fn:approximation-quality-footnote}Although our $f$ is not
even close to being as good as Zagier's function $f$ from \cite{ZAGIER}:
the quality of fit of Zagier's $f$ improves with the size of $z_{n-1}$
whereas we need $z_{n-2}$ to be big. This is one reason we must accelerate.} fit to $\log(z).$ In this paper, we write inequalities between vectors
to mean they hold at every coordinate.

\begin{lem}
\label{lem:Log-vs-f}There are constants $C_{1}$ and $C_{2}$ depending only
on $n$ such that when $z\in\M-K$ with $\alpha(z)>C_{1}$

\begin{equation}
\log(z)\leq f(z)\leq\log(z)+C_{2}\alpha(z)^{-2}(0,0,0,\ldots,0,1).\label{eq:f-log-sandwich}
\end{equation}
\end{lem}

\begin{proof}
Since $z$ satisfies the equation (\ref{eq:markoff-hurwitz-normalized}),
and $z_{n}$ is always the larger of the two quadratic roots of the
resulting quadratic in $z_{n},$ we have

\[
z_{n}=\frac{A(z)}{2}\left(1+\sqrt{1-4\frac{C(z)-k'}{A(z)^{2}}}\right)
\]
where 
\[
A(z)=\prod_{i=1}^{n-1}z_{i},\quad C(z)=\sum_{i=1}^{n-1}z_{i}^{2}
\]
and $k'\geq0$ is the constant from (\ref{eq:markoff-hurwitz-normalized}).
Now the first inequality of (\ref{eq:f-log-sandwich}) follows from
(\ref{eq:C(z)-minus-kdash-nonnegative}).

For the second inequality we estimate 
\[
\frac{C(z)-k'}{A(z)^{2}}\leq\sum_{i=1}^{n-1}\frac{z_{n-1}^{2}}{\prod_{j\neq n}z_{i}^{2}}\leq(n-1)\frac{1}{\prod_{j\leq n-2}z_{j}^{2}}=(n-1)\alpha(z)^{-2}.
\]
We can then choose $C_{1}$ large enough so that when $\alpha(z)>C_{1}$
we have 

\[
z_{n}=A(z)(1+O_{n}(\alpha(z)^{-2})),
\]
by increasing $C_{1}$ again if necessary we obtain

\[
\log(z_{n})=\log(A(z))+O_{n}(\alpha(z)^{-2})=f(z)_{n}+O_{n}(\alpha(z)^{-2}).
\]
\end{proof}

The following adapts an idea of Zagier from \cite[Proof of Lemma 3]{ZAGIER}
to our setting. While the strength of approximation is different,
we take the same approach in noting that if $f(z)$ is close to $y$
then $f(\lambda_{j}z)$ will be close to $\gamma_{j}y$. Of course
this is designed to be iterated.
\begin{lem}
\label{lem:boundingygrowth}There are $C_{1},C_{2}$ depending only
on $n$ such that for all $\e>0,$ for $z\in\M-K,$ 
$\alpha(z)>\max(C_{1},2C_{2}^{1/2}\epsilon^{-1/2})$,
and for $y^{(1)},y^{(2)}\in\H,$ if 

\begin{equation}
y^{(1)}+\epsilon(0,0,\ldots0,\frac{1}{2},\frac{1}{2},1)<f(z)\leq y^{(2)}\label{eq:epsilon-setup}
\end{equation}
then 

\begin{equation}
\gamma_{j}y^{(1)}+\e(0,0,\ldots0,\frac{1}{2},\frac{1}{2},1)<f(\lambda_{j}z)\leq\gamma_{j}y^{(2)}\label{eq:epsiloncarry}
\end{equation}
for all $1\leq j\leq n-1.$
\end{lem}

\begin{proof}
We first prove the upper bound for $f(\lambda_{j}z)$ from (\ref{eq:epsiloncarry}).
The inequality $f(z)\leq y^{(2)}$ implies that $\log(z_{i})\leq y_{i}^{(2)}$
for $i\leq n-1.$ By Lemma \ref{lem:Log-vs-f} we get $\log(z_{n})\leq f(z)_{n}\leq y_{n}$
as well. Then $f(\lambda_{j}z)\leq\gamma_{j}y^{(2)}$ follows.

For the other inequality, $f(z)>y^{(1)}+\epsilon(0,0,\ldots0,1/2,1/2,1)$
implies $\log(z_{i})>y_{i}^{(1)}$ for all $i\leq n-3$ and $\log(z_{i})>y_{i}^{(1)}+\e/2$
for $i=n-2,n-1$. By Lemma \ref{lem:Log-vs-f}, $\log(z_{n})\geq f(z)_{n}-C_{2}\alpha(z)^{-2}\geq y_{n}^{(1)}+\e-C_{2}\alpha(z)^{-2}.$
Since $\alpha(z)>2C_{2}^{1/2}\epsilon^{-1/2}$ we get 
\[
\log(z_{n})\geq y_{n}^{(1)}+3\epsilon/4.
\]
When $i\leq n-3$ we have $f(\lambda_{j}z)_{i}\geq(\gamma_{j}y^{(1)})_{i}$
quite clearly. If $j\leq n-2$ we have $f(\lambda_{j}z)_{n-2}=\log z_{n-1}\geq y_{n-1}^{(1)}+\e/2=(\gamma_{j}y^{(1)})_{n-2}+\epsilon/2$
and if $j=n-1$ then $f(\lambda_{j}z)_{n-2}=\log z_{n-2}\geq y_{n-2}^{(1)}+\e/2=(\gamma_{j}y^{(1)})_{n-2}+\epsilon/2$.
At the $(n-1)$st coordinate we have $f(\lambda_{j}z)_{n-1}=\log z_{n}\geq y_{n}^{(1)}+3\e/4=(\gamma_{j}y^{(1)})_{n-1}+3\epsilon/4$
which is sufficient. It remains to check the last coordinate. Here,

\[
f(\lambda_{j}z)_{n}=\sum_{i\neq j}\log z_{i}\geq\sum_{i\neq j}y_{i}^{(1)}+5\epsilon/4=(\gamma_{j}y^{(1)})_{n}+5\epsilon/4.
\]
The inequality above is due to the fact that at least one of $\log z_{n-2},\log z_{n-1}$
appear on the left hand side (giving $\epsilon/2)$ and $\log z_{n}$
also appears (giving $3\epsilon/4).$
\end{proof}

We can now accomplish Parts \ref{enu:fitting-the-point} and \ref{enu:Matching--group-elements}
of our plan above. Recall we have some fixed $z^{(0)}\in\M-K.$ For
each given $\lambda'\in\Lambda$ (in particular, those that occur
in (\ref{eq:M-iterated-expression})) we define 

\[
y(\lambda')=f(\lambda'z^{(0)}).
\]
We choose our parameters as follows: let $C_{2}$ be the constant
from Lemma \ref{lem:boundingygrowth} and set 
\begin{equation}
\epsilon=\epsilon(a)=16C_{2}\exp(-2C\phi^{ca/\log a}).\label{eq:epsilon-choice}
\end{equation}
so that by (\ref{eq:what-alpha-is})

\[
4C_{2}\alpha(\lambda'z^{(0)})^{-2}\leq\epsilon
\]
for all $\lambda'$ appearing in (\ref{eq:M-iterated-expression}).
\begin{lem}[Completing Part \ref{enu:fitting-the-point}]
We have 
\[
(1-\epsilon)y(\lambda')+\epsilon(0,0,\ldots0,\frac{1}{2},\frac{1}{2},1)<f(\lambda'z^{(0)})=y(\lambda').
\]
\end{lem}

\begin{proof}
For any nonidentity map $\lambda'\in \Lambda$, 
\[
(\lambda' z^{(0)})_{n-2} \geq (z^{(0)})_{n-1} > 2,
\]
using \eqref{eq:2lowerbound}. Therefore $f(\lambda'z^{(0)})_{n-2} \geq \log(2) > 1/2$. Since $f(\lambda'z^{(0)}) \in \H$ it follows that
\[
\e f(\lambda'z^{(0)})\geq\epsilon(0,0,\ldots0,\frac{1}{2},\frac{1}{2},1),
\]
from which the lemma is a direct consequence.
\end{proof}
Now for each 
\[
\lambda^{(2)}=\lambda_{n-1}^{A_{N}}\lambda_{j_{N}}\lambda_{n-1}^{A_{N-1}}\lambda_{j_{N-1}}\ldots\lambda_{n-1}^{A_{2}}\lambda_{j_{2}}\lambda_{n-1}^{A_{1}}\lambda_{j_{1}}\in S^{(N)},\quad1\leq j_{i}\leq n-2\:\forall i
\]
 appearing in (\ref{eq:M-lambda-prime-expression-for-comparison-to-linear}),
we set

\begin{equation}
\g^{(2)}=\gamma^{(2)}(\lambda^{(2)})=\gamma_{n-1}^{A_{N}}\gamma_{j_{N}}\gamma_{n-1}^{A_{N-1}}\gamma_{j_{N-1}}\ldots\gamma_{n-1}^{A_{2}}\gamma_{j_{2}}\gamma_{n-1}^{A_{1}}\gamma_{j_{1}}\in\Gamma'\cup\{e\}.\label{eq:gamma-lambda-matching}
\end{equation}
This is the matching of Part \ref{enu:Matching--group-elements}.
Since $\Lambda'$ and $\Gamma'$ are free, this gives a bijective
correspondence. 

The key point now is that by iterating Lemma \ref{lem:boundingygrowth}
we obtain for all coupled $\lambda^{(2)},\gamma^{(2)},$
\[
(1-\epsilon)\gamma^{(2)}.y(\lambda')+\epsilon(0,0,\ldots0,\frac{1}{2},\frac{1}{2},1)<f(\lambda^{(2)}\lambda'z^{(0)})\leq\gamma^{(2)}.y(\lambda')
\]
where we have used the linearity of the action of $\Gamma$ to pull
out the factor of $(1-\epsilon).$ Using Lemma \ref{lem:Log-vs-f}
we get
\[
\log(\lambda^{(2)}\lambda'z^{(0)})_{n}\leq f(\lambda^{(2)}\lambda'z^{(0)})_{n}\leq(\gamma^{(2)}.y(\lambda'))_{n}
\]
and
\[
\log(\lambda^{(2)}\lambda'z^{(0)})_{n}\geq f(\lambda^{(2)}\lambda'z^{(0)})_{n}-\frac{\epsilon}{4}\geq(1-\epsilon)(\gamma^{(2)}.y(\lambda'))_{n}.
\]
Then taking logarithms gives
\begin{equation}
\log\log(\lambda^{(2)}\lambda'z^{(0)})_{n}\leq\log(\gamma^{(2)}.y(\lambda'))_{n}\leq\log\log(\lambda^{(2)}\lambda'z^{(0)})_{n}+2\epsilon\label{eq:sandwich-log-terms}
\end{equation}
using $2\epsilon+\log(1-\epsilon) > 0$ for $\epsilon\ll 1.$

Note (\ref{eq:sandwich-log-terms})
also holds when $\g^{(2)}=e$, $\lambda^{(2)}=e.$ Now we claim we
can reasonably compare each of the $M(\lambda'z^{(0)},a-\tau_{\star}^{L}(\lambda'))$
from (\ref{eq:M-iterated-expression}) to $N(y(\lambda'),a')$ defined
in (\ref{eq:N-of-y-a-definition}) with $a'$ very close to $a-\tau_{\star}^{L}(\lambda')$.

\begin{lem}
\label{lem:sandwich-N-M}We have

\[
N(y(\lambda'),a-\tau_{\star}^{L}(\lambda')-\epsilon)\leq M(\lambda'z^{(0)},a-\tau_{\star}^{L}(\lambda'))\leq N(y(\lambda'),a-\tau_{\star}^{L}(\lambda')+\epsilon).
\]

\end{lem}

\begin{proof}
We write out

\[
N(y,a')=\sum_{\g^{(2)}\in\Gamma'\cup\{e\}}\mathbf{1}\{\log(\gamma^{(2)}.y(\lambda'))_{n}-\log y(\lambda'){}_{n}\leq a'\}
\]
and compare to

\[
M(\lambda'z^{(0)},a-\tau_{\star}^{L}(\lambda'))=\sum_{\lambda^{(2)}\in\Lambda'\cup\{e\}}\mathbf{1}\left\{ \log\log(\lambda^{(2)}\lambda'z^{(0)})_{n})-\log\log(\lambda'z^{(0)})_{n}\leq a-\tau_{\star}^{L}(\lambda')\right\} 
\]
term by term, matching $\g^{(2)}$ with $\lambda^{(2)}$ as in (\ref{eq:gamma-lambda-matching}).
By (\ref{eq:sandwich-log-terms}) we have

\[
\log(\gamma^{(2)}.y(\lambda'))_{n}-\log y(\lambda'){}_{n}\leq\log\log(\lambda^{(2)}\lambda'z^{(0)})_{n}-\log\log(\lambda'z^{(0)})_{n}+2 \epsilon
\]
and
\[
\log\log(\lambda^{(2)}\lambda'z^{(0)})_{n})-\log\log(\lambda'z^{(0)})_{n}-2 \epsilon\leq\log(\gamma^{(2)}.y(\lambda'))_{n}-\log y(\lambda'){}_{n}
\]
from which the result follows.
\end{proof}

\subsection{Using the linear semigroup count to prove Proposition \ref{prop:orbital-counting-Lambda'-main-prop}}

We now use Theorem \ref{thm:-linear-counting}, whose proof will be
deferred to Section \ref{sec:linear-semigroup-count}. Let $y'=y(\lambda')=f(\lambda'z^{(0)}).$
\begin{lem}
\label{lem:Mlambda-a-preparation}Let $\delta$ be the small constant
from (\ref{eq:delta-definition}). We have

\begin{eqnarray*}
M(\lambda z^{(0)},a) & = & \left(1+o(1)\right)e^{\beta a}\sum_{\lambda':s^{L}(\lambda')=\lambda}e^{-\beta\tau_{\star}^{L}(\lambda')}h(y')\\
 & + & O\left(\exp(\beta a^{\delta}+(1+\delta)a)\right).
\end{eqnarray*}
The big and small $o$ terms have implied constant and decay rates
that are independent of $\lambda z^{(0)}.$\end{lem}
\begin{proof}
Using Lemma \ref{lem:sandwich-N-M} in the expression (\ref{eq:M-iterated-expression})
gives that up to a negligible $O(e^{(1+\delta)a}),$

\begin{equation}
\sum_{\lambda':s^{L}(\lambda')=\lambda}N(y(\lambda'),a-\e-\tau_{\star}(\lambda'))\leq M(\lambda z^{(0)},a)\leq\sum_{\lambda':s^{L}(\lambda')=\lambda}N(y(\lambda'),a+\e-\tau_{\star}(\lambda'))\label{eq:big-sandwich}
\end{equation}
where $y(\lambda')=f(\lambda'z^{(0}).$

We want to carefully use Theorem \ref{thm:-linear-counting} that
says that along with $h,\beta$ there is some function $F(a)$ such
that

\[
|N(y,a)-e^{\beta a}h(y)|\leq F(a)e^{\beta a}h(y)
\]
and $F(a)\to0$ as $a\to\infty.$ The minor problem with using this
in (\ref{eq:big-sandwich}) is that there may be terms with $a'=a\pm\epsilon-\tau_{\star}(\lambda')$
close to zero, or less than zero. Letting $\delta$ be the same small
parameter as before, we note that if $a'\leq a^{\delta}$ then there
is some constant $C_{3}\geq1$ such that

\[
|N(y,a')-e^{\beta a'}h(y)|\leq C_{3}e^{\beta a'}
\]
which follows from Theorem \ref{thm:-linear-counting} when $0\leq a'\leq a^{\delta}$
and is trivial when $a'<0$ since then $N(y,a')=0.$

Therefore, working with the right hand inequality of (\ref{eq:big-sandwich})
we get 

\begin{eqnarray*}
M(\lambda z^{(0)},a) & \leq & \sum_{\lambda':s^{L}(\lambda')=\lambda}\left(e^{\beta a'}h(y')+\mathbf{1}\{a'\leq a^{\delta}\}C_{3}e^{\beta a'}+\mathbf{1}\{a'>a^{\delta}\}F(a')e^{\beta a'}h(y')\right)
\end{eqnarray*}
where we write $a'=a'(\lambda')=a+\epsilon-\tau_{\star}^{L}(\lambda')$
and $y'=y(\lambda').$ Therefore

\begin{eqnarray}
M(\lambda z^{(0)},a) & \leq & \left(1+\sup_{b\geq a^{\delta}}F(b)\right)\sum_{\lambda':s^{L}(\lambda')=\lambda}e^{\beta a'}h(y')+C_{3}\sum_{\begin{subarray}{c}
\lambda':s^{L}(\lambda')=\lambda\\
a'\leq a^{\delta}
\end{subarray}}e^{\beta a'}.\label{eq:temp1}
\end{eqnarray}
For the first term in (\ref{eq:temp1}) note that

\begin{eqnarray*}
\sum_{\lambda':s^{L}(\lambda')=\lambda}e^{\beta a'}h(y') & = & e^{\beta a}\sum_{\lambda':s^{L}(\lambda')=\lambda}e^{\beta\epsilon}e^{-\beta\tau_{\star}^{L}(\lambda')}h(y')\\
 & = & (1+O(\exp(-2C\phi^{\frac{ca}{\log a}})))e^{\beta a}\sum_{\lambda':s^{L}(\lambda')=\lambda}e^{-\beta\tau_{\star}^{L}(\lambda')}h(y').
\end{eqnarray*}
The last term in (\ref{eq:temp1}) can be bounded by

\[
\ll e^{\beta a}\sum_{\begin{subarray}{c}
\lambda':s^{L}(\lambda')=\lambda\\
\tau_{\star}^{L}(\lambda')\geq a+\epsilon-a^{\delta}
\end{subarray}}e^{-\beta\tau_{\star}^{L}(\lambda')}.
\]
The contributions to the sum above from $M-1\leq\tau_{\star}^{L}(\lambda')\leq M$
are bounded by 

\[
\sum_{\lambda:s^{L}(\lambda')=\lambda}\mathbf{1}\{M\geq\tau_{\star}^{L}(\lambda')\geq M-1\}e^{-\beta\tau_{\star}^{L}(\lambda')}\leq c_{1}^{L}(c_{0}+M)^{L}e^{M}e^{-\beta(M-1)}
\]
by Lemma \ref{lem:iterated-error-term-small}, equation (\ref{eq:combinatorial-bound-for-tau-star-leq-x}).
Summing this quantity over natural numbers from $M_{0}=\lfloor a-a^{\delta}-1\rfloor$
to infinity, using the bound (\ref{eq:bound-for-IBP}) to replace $c_{1}^{L}(c_{0}+M)^{L}$
by $e^{\delta M}$, gives
\[
\sum_{\begin{subarray}{c}
\lambda':s^{L}(\lambda')=\lambda\\
\tau_{\star}^{L}(\lambda')\geq a+\epsilon-a^{\delta}
\end{subarray}}e^{-\beta\tau_{\star}^{L}(\lambda')}\ll e^{-(\beta-1-\delta)(a-a^{\delta})};
\]
so we get for the last term in (\ref{eq:temp1})

\[
\sum_{\begin{subarray}{c}
\lambda':s^{L}(\lambda')=\lambda\\
a'\leq a^{\delta}
\end{subarray}}e^{\beta a'}\ll\exp((\beta-1-\delta)a^{\delta}+(1+\delta)a).
\]
Therefore it can be absorbed into the error stated in the lemma. The
lower bound for $M(\lambda z^{(0)},a)$ is similar. Notice that our
constants and rates of decay do not depend on $\lambda z^{(0)}.$
\end{proof}
Proposition \ref{prop:orbital-counting-Lambda'-main-prop} will now
follow from Lemma \ref{lem:Mlambda-a-preparation} and the following
proposition.
\begin{prop}
\label{prop:Cauchy}For fixed $\lambda$ and $z^{(0)}$ there is a
constant $c_{\star}(\lambda z^{(0)})$ such that

\[
a_{L}(\lambda z^{(0)})=\sum_{\lambda':s^{L}(\lambda')=\lambda}h(y(\lambda'))e^{-\beta\tau_{\star}^{L}(\lambda')}=c_{\star}(\lambda z^{(0)})+o(1)
\]
as $L\to\infty,$ with a rate of decay that is independent of $\lambda$.
The values $c_{\star}(\lambda z^{(0)})$ are bounded by some constant
independent of $\lambda.$
\end{prop}

\begin{proof}
We are going to prove the sequence is Cauchy with a very fast rate.
Consider the difference of consecutive terms. Again we write $y'=y(\lambda').$
For $\lambda''\in S_{\Lambda}$ we write $y''=y''(\lambda'',\lambda')=f(\lambda''\lambda'z^{(0)})$.
We suppress the dependence of these variables on others to improve
readability.

We obtain

\begin{eqnarray}
a_{L+1}-a_{L} & = & \sum_{\lambda^{(2)}:s^{L+1}(\lambda^{(2)})=\lambda}h(y'')e^{-\beta\tau_{\star}^{L+1}(\lambda^{(2)})}-\sum_{\lambda':s^{L}(\lambda')=\lambda}h(y')e^{-\beta\tau_{\star}^{L}(\lambda')}\nonumber \\
 & = & \sum_{\lambda':s^{L}(\lambda')=\lambda}e^{-\beta\tau_{\star}^{L}(\lambda')}\left(\left(\sum_{\lambda''\in S_{\Lambda}}h(y'')e^{-\beta(\tau_{\star}^{L+1}(\lambda''\lambda')-\tau_{\star}^{L}(\lambda'))}\right)-h(y')\right)\nonumber \\
 & = & \sum_{\lambda':s^{L}(\lambda')=\lambda}e^{-\beta\tau_{\star}^{L}(\lambda')}\left(\left(\sum_{\lambda''\in S_{\Lambda}}h(y'')e^{-\beta\tau_{\star}(\lambda''\lambda')}\right)-h(y')\right)\nonumber \\
 & = & \sum_{\lambda':s^{L}(\lambda')=\lambda}e^{-\beta\tau_{\star}^{L}(\lambda')}\left(\left(\sum_{\lambda''\in S_{\Lambda}}h(y'')\left(\frac{\log(\lambda'z^{(0)})_{n}}{\log(\lambda''\lambda'z^{(0)})_{n}}\right)^{\beta}\right)-h(y')\right).\label{eq:difference-form}
\end{eqnarray}
The point is that the terms in parentheses should be close to zero
by the recursion (\ref{eq:h-recursion}) satisfied by $h$ over $\Gamma'.$
We will use Lemma \ref{lem:Log-vs-f} which gives a bound when $\alpha(\lambda'z^{(0)})>C_{1}$.
On the other hand by Lemma \ref{lem:growth-of-alpha-in-L} there is
some $L_{0}$ such that when $L\geq L_{0}$ and $s^{L}(\lambda')=\lambda$
then $\alpha(\lambda'z^{(0)})>C_{1}$.

We use the natural bijection

\[
S_{\Lambda}\to T_{\Gamma},\quad\lambda''\mapsto\gamma(\lambda'').
\]
When $L>L_{0},$ repeating the arguments of the previous section leading
up to (\ref{eq:sandwich-log-terms}) gives the bounds

\begin{equation}
\log(\lambda'z^{(0)})_{n}\leq y'_{n}\leq(1+O(\alpha(\lambda'z^{(0)})^{-2}))\log(\lambda'z^{(0)})_{n}\label{eq:temp2}
\end{equation}

\begin{equation}
\log(\lambda''\lambda'z^{(0)})_{n}\leq(\gamma(\lambda'').y')_{n}\leq(1+O(\alpha(\lambda'z^{(0)})^{-2}))\log(\lambda''\lambda'z^{(0)})_{n}\label{eq:temp3}
\end{equation}
where the implied constant depends only on $n.$ Moreover, using Lemma
\ref{lem:Log-vs-f} gives 

\begin{equation}
\log(\lambda''\lambda'z^{(0)})\leq y''\leq\log(\lambda''\lambda'z^{(0)})+C_{2}\alpha(\lambda'z^{(0)})^{-2}(0,0,\ldots,0,1)\label{eq:temp4}
\end{equation}
whenever $L>L_{0}.$

Suppose $L>L_{0}.$ We must estimate the cost of replacing $y''$
by $\gamma(\lambda'')y'$ and $\left(\frac{\log(\lambda'z^{(0)})_{n}}{\log(\lambda''\lambda'z^{(0)})_{n}}\right)^{\beta}$
by $\left(\frac{y'_{n}}{(\gamma(\lambda'')y')_{n}}\right)^{\beta}$
in (\ref{eq:difference-form}). Since using (\ref{eq:temp3}) and
(\ref{eq:temp4}) gives that $y''$ is within\\
 $O(\alpha(\lambda'z^{(0)})^{-2}\log(\lambda''\lambda'z^{(0)})_{n})$
of $\gamma(\lambda'').y'$ and $h$ is $C^{1},$ we get

\[
h(y'')=h(\gamma(\lambda'')y')+O(\alpha(\lambda'z^{(0)})^{-2}\log(\lambda''\lambda'z^{(0)})_{n}).
\]
Using (\ref{eq:temp2}) and (\ref{eq:temp3}) gives

\[
\left(\frac{y'_{n}}{(\gamma(\lambda'')y')_{n}}\right)^{\beta}(1+O(\alpha(\lambda'z^{(0)})^{-2}))^{-\beta}\leq\left(\frac{\log(\lambda'z^{(0)})_{n}}{\log(\lambda''\lambda'z^{(0)})_{n}}\right)^{\beta}\leq\left(\frac{y'_{n}}{(\gamma(\lambda'')y')_{n}}\right)^{\beta}(1+O(\alpha(\lambda'z^{(0)})^{-2}))^{\beta}.
\]
Using that $h$ and $\left(\frac{\log(\lambda'z^{(0)})_{n}}{\log(\lambda''\lambda'z^{(0)})_{n}}\right)^{\beta},$
$\left(\frac{y'_{n}}{(\gamma(\lambda'')y')_{n}}\right)^{\beta}$ are
bounded we get

\begin{eqnarray*}
\sum_{\lambda''\in S_{\Lambda}}h(y'')\left(\frac{\log(\lambda'z^{(0)})_{n}}{\log(\lambda''\lambda'z^{(0)})_{n}}\right)^{\beta} & = & \sum_{\gamma''\in T_{\Gamma}}h(\gamma(\lambda'')y')\left(\frac{y'_{n}}{(\gamma(\lambda'')y')_{n}}\right)^{\beta}+O(\alpha(\lambda'z^{(0)})^{-2})\\
 & = & h(y')+O(\alpha(\lambda'z^{(0)})^{-2})
\end{eqnarray*}
where the last equality uses the recursion (\ref{eq:h-recursion}).
Therefore for $L\geq L_{0}$

\[
|a_{L+1}-a_{L}|\ll\alpha(\lambda'z^{(0)})^{-2}\sum_{\lambda':s^{L}(\lambda')=\lambda}e^{-\beta\tau_{\star}^{L}(\lambda')}.
\]

It is possible to use a fortiori estimates to prove the sum above
is universally bounded, for example by using the work of Baragar \cite{BARAGAR2}
in the case of $k=0.$ To keep things self contained, since we only
need a coarse bound we instead use Lemma \ref{lem:iterated-error-term-small}
to prove

\begin{equation}
\sum_{\lambda':s^{L}(\lambda')=\lambda}e^{-\beta\tau_{\star}^{L}(\lambda')}\ll\exp(C_{4}L^{1+\eta})\label{eq:fixed-L-exp-minusbeta-sum}
\end{equation}
for some constant $C_{4}$ and small $\eta$. However, $\alpha(\lambda'z^{(0)})^{-2}$
is much smaller than this: by Lemma \ref{lem:growth-of-alpha-in-L}
we have $\alpha(\lambda'z^{(0)})^{-2}\ll\exp(-2C\phi^{L})$ where
$\phi>1$ so not only is 
\[
|a_{L+1}-a_{L}|\ll\exp(C_{4}L^{1+\eta}-2C\phi^{L})
\]
very small but we can sum the differences to get a Cauchy sequence.
Indeed $C_{4}L^{1+\eta}-2C\phi^{L}\leq C_{5}-C_{6}\phi^{L}$ for some
$C_{5},C_{6}>0.$ Therefore for $L_{1}\geq L_{0}$

\begin{equation}
\sum_{L=L_{1}}^{\infty}|a_{L+1}-a_{L}|\ll\sum_{L=L_{1}}^{\infty}\exp(-C_{6}\phi^{L})=o_{L_{1}\to\infty}(1)\label{eq:fast-Cauchy}
\end{equation}
so the sequence converges at a uniform rate to its limit $c_{\star}(\lambda z^{(0)})$.
The uniform boundedness of $c_{\star}(\lambda z^{(0)})$ will follow
from the uniform boundedness of $a_{L_{0}}(\lambda z^{(0)})$ given
(\ref{eq:fast-Cauchy}), and $a_{L_{0}}(\lambda z^{(0)})$ is uniformly
bounded by using that $h$ is bounded and the already established
(\ref{eq:fixed-L-exp-minusbeta-sum}). This finishes the proof.
\end{proof}

Putting Proposition \ref{prop:Cauchy} and Lemma \ref{lem:Mlambda-a-preparation}
together proves Proposition \ref{prop:orbital-counting-Lambda'-main-prop}
given Theorem \ref{thm:-linear-counting}. In the rest of the paper
we prove Theorem \ref{thm:-linear-counting}.

\section{The linear semigroup count}

\label{sec:linear-semigroup-count}

\subsection{Renewal (again)}

Now we discuss renewal for the quantity $N(y,a)$ that appears in
Theorem \ref{thm:-linear-counting}. The renewal equation for $N(y,a)$
says 

\begin{equation}
N(y,a)=\sum_{\gamma\in T_{\Gamma}}N(\gamma.y,a-\log(\gamma.y)_{n}+\log y_{n})+\mathbf{1}\{0\leq a\}.\label{eq:N-renewal}
\end{equation}

Notice from its Definition in (\ref{eq:N-of-y-a-definition}) that
the function $N(y,a)$ is invariant under multiplication of the $y$
variable by $\R_{+}$. With this in mind, we are going to consider

\[
P(\R_{\geq0}^{n})=\R_{\geq0}^{n}/\R_{+},
\]
the quotient of $\R_{\geq0}^{n}$ by the multiplicative action of
positive real numbers. Let $\Delta\subset P(\R_{\geq0}^{n})$ denote
the projection of $\H.$ We will from now on use a coordinate

\[
w=(w_{1},w_{2},\ldots,w_{n-1},1)
\]
with $w_{1}\leq w_{2}\leq\ldots\leq w_{n-1}$ and $\sum_{j=1}^{n-1}w_{j}=1$
to uniquely represent a point in $\Delta.$ We now view $N(w,a)$
as a function on $\Delta\times\R_{\geq0}.$ Note that equation (\ref{eq:N-renewal})
descends to $(w,a)\in\Delta\times\R_{\geq0}.$

Now, for the first time in the paper, we start the full argument of
the renewal method\footnote{Previously we just used an iteration of a renewal equation to perform
a linearization.}. This begins with taking a Laplace transform which we define for
general $f$ of suitable decay by 

\[
\hat{f}(s)=\int_{-\infty}^{\infty}e^{-sx}f(x)dx.
\]
The outcome of taking a Laplace transform of the renewal equation
(\ref{eq:N-renewal}) in the $a$ variable, ignoring issues of convergence\footnote{These issues are worked out in Lemma \ref{lem:transferwelldefinedoncontinuous}.},
is that

\begin{equation}
\hat{N}(w,s)=\sum_{\gamma\in T_{\Gamma}}\left(\frac{w_{n}}{(\gamma.w)_{n}}\right)^{s}\hat{N}(\gamma.w,s)+\frac{1}{s}\label{eq:fourier-transform-of-renewal}
\end{equation}
for all $w\in\Delta,$ where $\hat{N}(w,s)$ is the Laplace transform
$\widehat{N(w,\bullet)}$ in the $a$ variable. Thus $s$ is a frequency
parameter dual to the counting parameter $a.$ Notice that the function

\[
(\gamma,w)\mapsto\frac{w_{n}}{(\gamma.w)_{n}}
\]
 descends from $T_{\Gamma}\times\H$ to a well defined real valued
function on $T_{\Gamma}\times\mbox{\ensuremath{\Delta}.}$

Now we introduce the transfer operator that will play a crucial role in this section. For a function
$f$ on $\Delta$ we define

\begin{equation}
\L_{s}[f](w)=\sum_{\gamma\in T_{\Gamma}}\left(\frac{w_{n}}{(\gamma.w)_{n}}\right)^{s}f(\gamma.w)\label{eq:transfer-definition}
\end{equation}
whenever the sum is pointwise absolutely convergent on $\Delta$.
Then (\ref{eq:fourier-transform-of-renewal}) can be rephrased as

\begin{equation}
\hat{N}(\bullet,s)=s^{-1}(1-\L_{s})^{-1}\mathbf{1},\label{eq:resolvent-formula}
\end{equation}
whenever the resolvent operator $(1-\L_{s})^{-1}$ exists in such
a way it can act on the constant function $\mathbf{1}.$ 

There is a procedure due to Lalley to convert (\ref{eq:resolvent-formula})
together with a sufficiently complete description of the spectrum
of $\L_{s}$ on a suitable Banach space into Theorem \ref{thm:-linear-counting}.
More specifically we will appeal to the perturbation theory and Fourier
analysis developed in \cite[Sections 7 and 8]{LALLEY}. In the next
section we will lay out the necessary spectral theory of $\L_{s}.$
Before that, let us calculate explicitly the sum in (\ref{eq:transfer-definition}).
\begin{lem}
\label{lem:branchcalcs}An element $\gamma_{n-1}^{A}\gamma_{j}$ of
$T_{\Gamma}$ acts on $\Delta$ by 

\begin{eqnarray}
\gamma_{n-1}^{A}\gamma_{j}.[w_{1},\ldots,w_{n-1},1]\nonumber \\
= & [w_{1},\ldots,\widehat{w_{j}},\ldots,w_{n-1},1+A(1-w_{j}),1+(A+1)(1-w_{j})];\label{eq:longformula-1}
\end{eqnarray}
in particular,
\begin{equation}
(\gamma_{n-1}^{A}\gamma_{j}.(w_{1},\ldots w_{n-1},1))_{n}=1+(A+1)(1-w_{j}).\label{eq:new-nth-coefficient}
\end{equation}
\end{lem}
\begin{proof}
This is a direct calculation.\end{proof}
\begin{example}[Gauss map]
\label{exa:Gaussexample1}When $n=3,$ the only inverse branches
are of the form

\[
\g_{2}^{A}\g_{1}(w_{1},w_{2})=(1+(A+1)w_{2})^{-1}(w_{2},1+Aw_{2}).
\]
With the change of variables $x=w_{1}/w_{2},$ these are precisely
the inverse branches of the Gauss map $x\mapsto\{\frac{1}{x}\}$:

\[
\g_{2}^{A}\g_{1}:x\mapsto\frac{1}{x+A+1},\quad A\in\Z_{\geq0}.
\]

\end{example}

\subsection{Spectral theory of the transfer operator}

\label{sub:spectral-theory-transfer-operator}

In this section, we give a full account of the spectral theory of
$\L_{s}$.
 A good reference for the spectral theory of transfer operators is the book of Baladi \cite{BALADI}.
We begin with the following lemma.
\begin{lem}
\label{lem:transferwelldefinedoncontinuous}When $\Re(s)>1$ the summation
in the defining equation (\ref{eq:transfer-definition}) of $\L_{s}$
is absolutely and uniformly convergent on $\Delta$ and so gives a
well defined continuous map of Banach spaces\footnote{$C^{0}$ is the Banach space of continuous functions with the supremum
norm.}

\[
\L_{s}:C^{0}(\Delta)\to C^{0}(\Delta).
\]
\end{lem}
\begin{proof}
Substituting Lemma \ref{lem:branchcalcs}, equation (\ref{eq:new-nth-coefficient})
in the Definition (\ref{eq:transfer-definition}), the summation amounts
to

\begin{equation}
\L_{s}[f](w)=\sum_{j\in[n-2]}\sum_{A\in\N}\frac{1}{(1+(A+1)(1-w_{j}))^{s}}f(\g_{n-1}^{A}\g_{j}.w).\label{eq:transder-defn-explicit}
\end{equation}
Since $w_{j}\leq1/2$ for $j\in [n-2]$ and $f$ is bounded, each sum in $L$ converges
uniformly absolutely on $\Delta$ for $\Re(s)>1$. The limit is then
continuous and bounded by a constant multiple, depending on $s$,
of $\|f\|_{\infty}.$
\end{proof}

We obtain the following consequence of Lemma \ref{lem:transferwelldefinedoncontinuous}
by a standard application of the Schauder-Tychonoff Theorem.
\begin{cor}[Existence of eigenmeasures]
\label{lem:conformalexistence}Let $\L_{s}^{*}$ denote the dual
of $\L_{s}.$ For each real $s>1$ there is a number $\lambda_{s}>0$
and a probability measure $\nu_{s}$ such that $\L_{s}^{*}\nu_{s}=\lambda_{s}\nu_{s}.$\end{cor}
\begin{example}[Transfer operator for the Gauss map]
Let $n=3.$ Carrying on from Example \ref{exa:Gaussexample1}, we
have in the coordinate $x=w_{1}/w_{2}$

\[
\L_{s}[f](x)=\sum_{A\in\N}\frac{(x+1)^{s}}{(x+A+2)^{s}}f\left(\frac{1}{x+1+A}\right).
\]
\emph{This is not the usual transfer operator for the Gauss map. }However,
letting $M_{(x+1)^{s}}$ denote the operator of multiplication by
$(x+1)^{s},$ we get

\[
M_{(x+1)^{s}}^{-1}\L_{s}M_{(x+1)^{s}}[f](x)=\sum_{A\in\N}\frac{1}{(x+A+1)^{s}}f\left(\frac{1}{x+1+A}\right)=\L_{s}^{\mathrm{Gauss}}[f](x),
\]
the classical transfer operator for the Gauss map. This coincides
with the Perron-Frobenius operator for the Gauss map when $s=2.$
The leading eigenfunction of $\L_{2}^{\mathrm{\mathrm{Gauss}}}$ corresponds
to a multiplicity 1 eigenvalue $1$ and eigenfunction 

\[
h(x)=\frac{1}{1+x}.
\]
This eigenfunction was known to Gauss \cite{GAUSS}, and its invariance property
was formally proved by Kuzmin \cite{KUZMIN}. Correspondingly, the
leading eigenfunction of $\L_{2}$ is $[M_{(x+1)^{2}}h](x)=(x+1)=\frac{1}{w_{2}}$
with eigenvalue 1.

\end{example}

 Our functional analysis takes place on the Banach space $C^{1}(\Delta)$
which consists of continuously differentiable functions on $\Delta$
with the norm

\[
\|f\|_{C^{1}}=\|f\|_{\infty}+\|\nabla f\|_{\infty}.
\]
We use the standard Euclidean metric on $\Delta$ given by the coordinates
$w_{1},\ldots,w_{n-1}.$
\begin{lem}
\label{lem:holomorphic-bounded}In the region $\Re(s)>1$, the mapping
$s\mapsto\L_{s}$ gives a holomorphic family of bounded operators
on the Banach space $C^{1}(\Delta).$ In particular, for $\Re(s)>1$,
$\L_{s}$ is bounded on $C^{1}(\Delta).$ 
\end{lem}
We will prove the following version of the Ruelle-Perron-Frobenius
Theorem. 
\begin{thm}[Ruelle-Perron-Frobenius]
\label{thm:RPF}Let $s\in(1,\infty)$ be a real parameter for the
transfer operator $\L_{s}:C^{1}(\Delta)\to C^{1}(\Delta)$.
\begin{enumerate}
\item \label{enu:RPFspectrum}The eigenvalue $\lambda_{s}$ is multiplicity
one and the rest of the spectrum of $\L_{s}$ in contained in a ball
of radius $R(s)$ strictly less than $\lambda_{s}.$ For any compact
interval $I\subset(1,\infty)$ there is an $\epsilon(I)>0$ such that
$\lambda_{s}-R(s)\geq\epsilon$ for $s\in I.$
\item There is a unique probability measure $\nu_{s}$ such that $\L_{s}^{*}\nu_{s}=\lambda_{s}\nu_{s}.$
\item The unique eigenfunction $h_{s}\in C^{1}(\Delta)$ for the eigenvalue
$\lambda_{s}$ with $\nu_{s}(h_{s})=1$ is positive.
\end{enumerate}
\end{thm}

In the case of the Gauss map, a version of Theorem \ref{thm:RPF}
was first proved by Wirsing \cite{WIRSING}.
In the case of $n=4$, when there is a close connection between the Rauzy gasket and the dynamics of $\Gamma'$ on $\Delta$ as explained in Example \ref{ex:Rauzy}, a version of Theorem \ref{thm:RPF} was proved by Avila, Hubert, and Skripchenko in  \cite[Proof of Theorem 22]{AHS2}. There are slight differences; in \cite{AHS2} the authors work in a symbolic setting, so their function space is not the same as ours, whereas we need to know that $h\in C^1(\Delta)$, for example, in order to state Theorem \ref{thm:-linear-counting}.

It is well-known that Theorem \ref{thm:RPF} follows from eventually
contracting dynamics for example, by the use of Birkhoff cones and
contraction of a Hilbert projective metric as in the paper of Liverani
\cite{LIVERANI}. The only thing that is possibly nonstandard about
our setting is the presence of both countably many branches and a
semigroup action for which we expect the invariant set to have non
full Hausdorff dimension (cf. Figures \ref{fig:Color-image} and \ref{fig:fractal}).
We explain the proof of Lemma \ref{lem:holomorphic-bounded} and Theorem
\ref{thm:RPF} in Section \ref{sub:consequences and RPF}.

These proofs depend crucially on our dynamics being uniformly contracting,
which we make precise in Proposition \ref{prop:uniformly-contracting}. 
 We freely make use of this property henceforth. 
Let $T_\Gamma^{\Z_+}$ denote the set of all positively indexed sequences $(\gamma^{(1)}, \gamma^{(2)} ,\ldots )$ with each $\gamma^{(j)} \in T_\Gamma$.
Because the elements of $T_\Gamma$ uniformly contract $\Delta$, one obtains for any fixed $w_0 \in \Delta$ a map
\def\limit{\mathrm{limit}}

\[ 
\limit : T_\Gamma^{\Z_+} \to \Delta ,\quad \limit( \gamma^{(1)}, \gamma^{(2)} ,\ldots ) :=\lim_{j \to \infty} \gamma^{(1)} \ldots \gamma^{(j)} . w_0 ;
\]
in fact, this map does not depend on the choice of $w_0$. The image of this map is the attractor of the iterated function system 
given by the elements of $T_\Gamma$, which we also call the  \emph{limit set} of $\Gamma'$, and denote it by $\mathfrak{K}(\Gamma')$. Then $\mathfrak{K}(\Gamma')$ is a compact $\Gamma'$-invariant subset of $\Delta$. 

The Ruelle-Perron-Frobenius Theorem is not enough for input to Lalley's
framework of complex analysis. One must also know that there is some
non trivial spectral bound for $\L_{s}$ on the vertical line $s=\beta+it$,
the trivial bound being that the spectral radius is no greater than $\lambda_{\beta}.$
In the context of subshifts of finite type, this was investigated
by Pollicott in \cite{POLLICOTT} who found a cohomological criterion
for a nontrivial spectral bound. We make the following definition
as in Pollicott \cite[pg. 139]{POLLICOTT}, adapted to the current
setting.
\begin{defn}
We say that a function $f=u+iv$ with
\[
u,v:T_{\Gamma}\times\Delta\to\R
\]
is \emph{regular} if there is no $r\in\R$ and 
bounded\footnote{ It is possible to impose more regularity on $G$ in this definition but it is not necessary for our purposes.}
 function $G : \mathfrak{K}(\Gamma') \to \R$
such that
\[
v(\gamma,w)-G(\gamma . w)+G(w)-r\in2\pi\Z
\]
for all $\gamma\in T_{\Gamma}$ and $w\in
\mathfrak{K}(\Gamma')
.$ In other words,
there is no $r\in\R$ so that $v-r$ is \emph{cohomologous} on 
$\mathfrak{K}(\Gamma')$
to a $2\pi\Z$-valued function.
\end{defn}
The following theorem can be viewed as an an extension of a result
of Wielandt \cite{WIELANDT} on the spectrum of finite dimensional complex matrices.
It was proved by Pollicott \cite[Theorem 2]{POLLICOTT} in the context
of shifts of finite type in symbolic dynamics. The proof goes through
perfectly well in our context\footnote{ The main point is that our definition of regular function is strong enough to rule out $\L_s$ having an eigenvalue of modulus $\lambda_{\Re(s)}$. This fact is supplemented by compactness arguments  relying on the Ionescu Tulcea-Marinescu type inequality that we
establish in Lemma \ref{lem:Marinescu-Tulcea}.}.  to give 
\begin{thm}[Wielandt's Theorem, after Pollicott]
\label{thm:Wielandt}If 
\begin{equation}
F_{s}(\gamma,w)\equiv-s\log\left(\frac{(\gamma.w)_{n}}{w_{n}}\right)\in C^{1}(\Delta;\C)\label{eq:cocycle}
\end{equation}
 is regular, 
 $\Im(s) \neq 0$,
   and $\Re(s)>1$ then the spectral radius of the operator
$\L_{s}:C^{1}(\Delta)\to C^{1}(\Delta)$ is strictly less than $\lambda_{\Re(s)}.$
\end{thm}
This is applicable in the present setting:
\begin{prop}
\label{prop:minus-s-tau-regular}For all $s\in\C-\R$, the function
in (\ref{eq:cocycle}) is regular.\end{prop}
\begin{proof}
It is enough to show that for
\[
\tau(\gamma,w)=\log\left(\frac{(\gamma.w)_{n}}{w_{n}}\right)=\log(\gamma.w)_{n}-\log w_{n}
\]
there is no
bounded $G$ on $\mathfrak{K}(\Gamma')$
 such that the values of
\[
\tau'(\gamma,w):=\tau(\gamma,w)-G(\gamma .w)+G(w)
\]
for $(\g,w)\in T_{\Gamma}\times
\mathfrak{K}(\Gamma')
$ are contained in a translate
of a discrete subgroup of $\R$. So it is also enough to show that
for any such $\tau',$ the gaps between distinct values of $\tau'$
are not bounded below.

The fundamental simple fact we use is that for $\gamma\in T_{\Gamma}$
and 
$w$ such that $\gamma . w=w$, (from which it follows $w\in \mathfrak{K}(\Gamma')$)
\[
\tau'(\gamma,w)=\tau(\gamma,w)-G(\gamma . w)+G(w)=\tau(\gamma,w).
\]
Then it remains to show that gaps between distinct values of $\tau$
on the fixed points of $\gamma\in T_{\Gamma}$ are not bounded below.
We compute that
\[
\gamma_{n-1}^{A}\gamma_{n-2}=\left(\begin{array}{cccccc}
1 &  &  &  &  & 0\\
 & \ddots &  &  &  & \vdots\\
 &  & 1 & 0 & 0 & 0\\
0 & \cdots & 0 & 0 & 1 & 0\\
A & \cdots & A & 0 & A & 1\\
A+1 & \cdots & A+1 & 0 & A+1 & 1
\end{array}\right),
\]
so (using the block lower triangular structure) 
\[
\det(\gamma_{n-1}^{A}\gamma_{n-2}-TI_{n})=(1-T)^{n-3}(-T)(T^{2}-(A+1)T-1).
\]
Consequently, the eigenvalues aside from $0$ and $1$ are 
\[
T=\frac{A+1\pm\sqrt{(A+1)^{2}+4}}{2}.
\]
Let $T_{+}$ be the largest, that is, $T_{+}=\frac{A+1+\sqrt{(A+1)^{2}+4}}{2}>A+1.$
One can find an eigenvector $v_{+}$ for $T_{+}$ where
\[
v_{+}=(0,0,\ldots,0,1,T_{+},T_{+}(T_{+}-A))>0,
\]
moreover, $v_{+}\in\H.$ Now, $(\gamma.v_{+})_{n}/(v_{+})_{n}=T_{+}$
and so $\tau(\gamma_{n-1}^{A}\gamma_{n-2},[v_{+}])=\log T_{+}.$ Writing
$T_{+}=T_{+}(A),$ we have

\begin{eqnarray*}
\log T_{+}(A+1)-\log T_{+}(A) & = & \log\left(\frac{A+2+\sqrt{(A+2)^{2}+4}}{A+1
+
\sqrt{(A+1)^{2}+4}}\right)\\
 & = & \log\left(\left(1+\frac{1}{A+1}\right)\frac{1+\frac{1}{A+2}+\sqrt{1+\frac{4}{(A+2)^{2}}}}{1+\frac{1}{A+1}+\sqrt{1+\frac{4}{(A+1)^2}}}\right)\to0
\end{eqnarray*}
as $A\to\infty.$ As the terms are easily seen to be non-zero, this
completes the proof.
\end{proof}
The contour shifting argument of Lalley hinges on the behavior
of the eigenvalue $\lambda_{s}$ and, in particular, on the location of
the possible real value $\beta$ such that $\lambda_{\beta}=1.$ Since
our dynamics is suitably uniformly contracting, if such a value exists
it is unique:
\begin{prop}
\label{prop:The-eigenfunction-decreases}The eigenvalue $\lambda_{s}$
is a real analytic function of $s$ that is strictly decreasing on
$(1,\infty).$ We have $\lambda_{s}<1$ for sufficiently large $s.$
As such, any value $\beta_{0}\in(1,\infty)$ such that $\lambda_{\beta_{0}}=1$
is unique, and if no such $\beta_{0}$ exists then $\lambda_{s}<1$
for all $s\in(1,\infty)$.
\end{prop}

As we will discuss momentarily, such a $\beta_{0}$ does exist, and
it coincides with Baragar's $\beta$ from Theorem \ref{thm:Baragar-theorem}.
Note that when $s=\beta$ we obtain from Theorem \ref{thm:RPF} a
unique measure such that $\L_{\beta}^{*}\nu_{\beta}=\nu_{\beta}$.
Then we will show $\nu_{\beta}$ is the \emph{conformal measure }of
Theorem \ref{thm:beta-characterization}. Proposition \ref{prop:The-eigenfunction-decreases}
will be proved in Section \ref{sub:eigenvalue-behaviour}.

\subsection{Proofs of Theorem \ref{thm:beta-characterization} and \ref{thm:-linear-counting}
given the spectral theorems}

Here we make a sketch of the passage from the spectral theory outlined
in Section \ref{sub:spectral-theory-transfer-operator} to Theorems
\ref{thm:beta-characterization} and \ref{thm:-linear-counting} via
(\ref{eq:resolvent-formula}) and the techniques of Lalley from \cite{LALLEY}.
Firstly, if there is no value $\beta_{0}$ such that $\lambda_{\beta_{0}}=1$
then Proposition \ref{prop:The-eigenfunction-decreases} together
with Lemma \ref{lem:holomorphic-bounded} imply that the resolvent
$(1-\L_{s})^{-1}$ exists as a holomorphic family of bounded operators
on $C^{1}(\Delta)$ in the region $\Re(s)>1$. This would imply by
standard contour shifting arguments in combination with (\ref{eq:resolvent-formula})
that for any $\eta>0$ 

\begin{equation}
N(w,a)=O(e^{(1+\eta)a}).\label{eq:contradictory-count}
\end{equation}
But this can be used along with the arguments of Section \ref{sec:converting-linear-to-nonlinear}
to show for some $z$ in an infinite orbit of $\Lambda$ that $M(z,a)=O(e^{(1+\eta)a})$,
in contradiction to Baragar's result (Theorem \ref{thm:Baragar-theorem})
when $\eta$ is small. Here we use the fact that for any $n,$ there
is an infinite orbit in $V(\Z_{+})$ when $n=a$ and $k=0$ coming
from the tuple $(1,1,\ldots,1).$ In fact, for small $\eta,$ (\ref{eq:contradictory-count})
is already in contradiction to some of Baragar's results from \cite{BARAGAR2}
on orbits of the linear semigroup $\Gamma.$

Now suppose there is such a $\beta_{0}>1$ as in Proposition \ref{prop:The-eigenfunction-decreases}.
Then Lalley's method of proof of his analog of Theorem \ref{thm:-linear-counting}
is by a contour shifting argument involving control on meromorphic
behavior of $(1-\L_{s})^{-1}$ in the following two ways:
\begin{enumerate}
\item By standard results in Linear Perturbation Theory \cite[Sections 4.3 and 7.1]{KATO},
Lemma \ref{lem:holomorphic-bounded} and Part \ref{enu:RPFspectrum} of Theorem \ref{thm:RPF} 
 imply that the functions
\[
s\mapsto\lambda_{s},\:s\mapsto h_{s},\:s\mapsto\nu_{s}
\]
extend to holomorphic functions on a neighborhood of the real line
segment $(1,\infty)$ in $\Re(s)>1$ such that
\[
\lambda_{s}\neq0,\:\L_{s}h_{s}=\lambda_{s}h_{s},\:\L_{s}^{*}\nu_{s}=\lambda_{s}\nu_{s},\:\nu_{s}(h_{s})=1.
\]
By suitable spectral decomposition of $\L_{s}$, one finds a neighborhood
$U$ of $s=\beta_{0}$ and an operator $\L_{s}'$ such that $(1-\L'_{s})^{-1}$
is a holomorphic family of bounded operators on $C^{1}(\Delta)$ for
$s\in U$ and moreover
\[
(1-\L_{s})^{-1}g=(1-\lambda_{s})^{-1}\nu_{s}(g)h_{s}+(1-\L_{s}')^{-1}g
\]
for $s\in U-\{\beta_{0}\}.$ This is the analog of \cite[Proposition 7.2]{LALLEY}.
\item By use of Theorem \ref{thm:Wielandt} along with its supplement Proposition
\ref{prop:minus-s-tau-regular}, we obtain that
\[
s\mapsto(1-\L_{s})^{-1}
\]
 is holomorphic in a neighborhood of every $s$ with $\Re(s)=\beta_{0}$,
with the exception of $s=\beta_{0}.$
\end{enumerate}

The outcome of Lalley's argument is that 
\[
N(w,a)=h_{\beta_{0}}(w)e^{\beta_{0}a}+o(e^{\beta_{0}a})
\]
where the decay in the small $o$ does not depend on $w.$  Our argument
of Section \ref{sec:comparison-section} converts this into a version
of Theorem \ref{thm:main-counting} with $\beta$ replaced by $\beta_{0}.$
Finally, this contradicts Baragar's Theorem \ref{thm:Baragar-theorem}
unless $\beta=\beta_{0}.$ Then Theorem \ref{thm:-linear-counting}
is proved, assuming the theorems of Section \ref{sub:spectral-theory-transfer-operator}.

Theorem \ref{thm:beta-characterization} is now a direct consequence
of the following fact:
\begin{lem}
For all $\gamma\in\G$ we have

\[
\frac{(\gamma.w)_{n}}{w_{n}}=|\Jac_{w}(\gamma)|^{-\frac{1}{n-1}}
\]
where $|\Jac_{w}(\gamma)|$ is the absolute value of the Jacobian determinant of $\gamma$
acting on $\Delta=\H/\R_{+}$ at the point $w.$
\end{lem}

This can be checked by a direct calculation on general grounds as
in \cite[Lemma 2.1]{POLLICOTTRVZ}, or by using explicit formulae
that appear later in this paper, e.g. by calculating the determinants
of total derivatives we calculate in Section \ref{sec:dynamics}.

\subsection{Consequences of uniformly contracting dynamics}

\label{sub:consequences and RPF}

The spectral theorems of the previous section all rely on the action
of $\G'$ on $\Delta$ being by contractions. That can be summarized
in the following proposition.
\begin{prop}
\label{prop:uniformly-contracting}There are constants $D>0$ and
$\rho<1$ such that for all $\gamma^{(1)},\gamma^{(2)},\ldots,\gamma^{(N)}\in T_{\Gamma}$
we have

\[
\|d_{w}[\gamma^{(1)}\gamma^{(2)}\ldots\gamma^{(N)}]\|_{\op}\leq D\rho^{N}.
\]
Here we view $\gamma^{(1)}\gamma^{(2)}\ldots\gamma^{(N)}$ as self-maps
of $\Delta,$ using the fixed Euclidean metric on $\Delta$, $d_{w}$
is the total derivative of the map at $w\in\Delta,$ and $\|\bullet\|_{\op}$
is the operator norm of the map between tangent spaces (using the
$\ell^{2}$ norms coming from the metric).
\end{prop}
When $n=4$, modulo translation between the Rauzy gasket and our dynamical system, a proof of Proposition  \ref{prop:uniformly-contracting} was outlined by Arnoux and Starosta in \cite[Lemma 2]{AS} and given in more detail by Avila, Hubert, and Skripchenko in \cite[Lemma 13]{AHS1}. 

We will prove Proposition \ref{prop:uniformly-contracting} 
for all $n\geq 4$
in Section
\ref{sec:dynamics}. The dynamical Proposition \ref{prop:uniformly-contracting}
gets brought into play by the following two-norm inequality with origins
in the work of Ionescu Tulcea and Marinescu \cite{IT-M}.
\begin{lem}
\label{lem:Marinescu-Tulcea}There is $C>0$ such that for any $Q\in\N$
and $\Re(s)>1$

\[
\|\nabla\L_{s}^{Q}[f](w)\|_{2}\leq C|s|\L_{s}^{Q}[|f|](w)+D\rho^{Q}\L_{s}^{Q}[\|\nabla f\|_{2}](w)
\]
for all $w\in\Delta.$ We write $\|\bullet\|_{2}$ for the pointwise
$\ell^{2}$ norm in an individual tangent space fiber.\end{lem}
\begin{proof}
This is standard given Proposition \ref{prop:uniformly-contracting}:
it essentially boils down to the chain rule. The only thing to take
care with are the infinite sums that arise, but these are all absolutely
convergent when $\Re(s)>1.$
\end{proof}

We can now prove Lemma \ref{lem:holomorphic-bounded}.
\begin{proof}[Proof of Lemma \ref{lem:holomorphic-bounded}]
We are proving $s\mapsto\L_{s}$ is a holomorphic mapping to bounded
operators on $C^{1}(\Delta).$ If we truncate the summation going
into the expression (\ref{eq:transder-defn-explicit}) for $\L_{s}$
at some fixed $L$ to form 
\[
\L_{s}^{(L)}=\sum_{j\in[n-2]}\sum_{A\leq L}\frac{1}{(1+(A+1)(1-w_{j}))^{s}}f(\g_{n-1}^{A}\g_{j}.w);
\]
the resulting $\L_{s}^{(L)}$ is easily seen to be holomorphic by taking
a complex derivative. So it remains to show that $\L_{s}^{L}\to\L_{s}$
uniformly on compact sets, say in the norm topology. But the tail
consists of $n-2$ terms of the form
\[
(\L_{s}-\L_{s}^{(L)})[f](w)=\sum_{A>L}\frac{1}{(1+(A+1)(1-w_{j}))^{s}}f(\g_{n-1}^{A}\g_{j}.w).
\]
Then $\|\L_{s}-\L_{s}^{(L)}\|_{C^{0}}\to0$ as $L\to\infty$ and this
is uniform for $s$ in $W$, a compact subset of $\Re(s)>1.$ On the
other hand, the proof of Lemma \ref{lem:Marinescu-Tulcea} also applies
to $\L_{s}-\L_{s}^{(L)},$ so applying it when $Q=1$ gives
\[
\|\nabla(\L_{s}-\L_{s}^{(L)})[f]\|_{\infty}\leq C|s|\|(\L_{s}-\L_{s}^{(L)})[|f|]\|_{\infty}+D\rho\|(\L_{s}-\L_{s}^{(L)})[\|\nabla f\|_{2}]\|_{\infty}.
\]
This implies
\[
\|\L_{s}-\L_{s}^{(L)}\|_{C^{1}(\Delta)}\ll_{W}\|\L_{s}-\L_{s}^{(L)}\|_{C^{0}(\Delta)},
\]
which we've established goes to zero uniformly on $W.$
\end{proof}

The proof of the Ruelle-Perron-Frobenius Theorem \ref{thm:RPF} now
proceeds either via use of Birkhoff cones as in Liverani's paper \cite{LIVERANI}
or by a more direct approach as in Pollicott \cite[Lemma 2.3]{POLLICOTTRVZ}.
The classical proof of this Theorem for subshifts of finite type can
be found in \cite[Theorem 2.2]{PP}. In any approach Lemma \ref{lem:Marinescu-Tulcea}
is the key input. The uniform spectral gap stated in Part \ref{enu:RPFspectrum}
of Theorem \ref{thm:RPF} is a consequence of the uniformity of Lemma
\ref{lem:Marinescu-Tulcea} for $s$ in a fixed compact subinterval
of $(1,\infty).$

\subsection{Behavior of the eigenvalue}

\label{sub:eigenvalue-behaviour}

In this section we prove Proposition \ref{prop:The-eigenfunction-decreases}.
The statement that $\lambda_{s}$ is real analytic on $(1,\infty)$
follows from the fact we noted in the previous Section \ref{sub:spectral-theory-transfer-operator}
that by perturbation theory in combination with Theorem \ref{thm:RPF}
Part \ref{enu:RPFspectrum} 
\[
s\mapsto\lambda_{s}
\]
is holomorphic in a neighborhood of $(1,\infty)$ in $\Re(s)>1.$
Notice that we have the bound
\begin{eqnarray*}
\L_{s}[f](w) & = & \sum_{j\in[n-2]}\sum_{A\in\N}\frac{1}{(1+(A+1)(1-w_{j}))^{s}}f(\g_{n-1}^{A}\g_{j}.w)\\
 & \leq & (n-2)\|f\|_{\infty}\sum_{A\in\N}\frac{1}{(1+\frac{1}{2}(A+1))^{s}}\leq2(n-2)\|f\|_{\infty}\sum_{A\in\N}\frac{1}{(3+A)^{s}}.
\end{eqnarray*}
Letting $f=h_{s}$ and $w$ such that $h_{s}(w)=\|h_{s}\|_{\infty}$
gives 
\[
\lambda_{s}\leq2(n-2)\sum_{A\in\N}\frac{1}{(3+A)^{s}}
\]
so $\lambda_{s}\to0$ as $s\to\infty.$

It remains to show that $\lambda_{s}$ is strictly decreasing in $s.$
Let $I$ be a fixed compact subinterval of $(1,\infty)$. By Theorem
\ref{thm:RPF} $\lambda_{s}^{-N}\L_{s}^{N}1$ converges in $C^{1}$
norm to $h_{s}$ and this convergence is uniform for $s\in I.$ This
implies
\begin{equation}
\log\lambda_{s}=\frac{\log\left(\L_{s}^{N}[1](w)\right)}{N}+o(1)\label{eq:lambda_s-approx}
\end{equation}
where the error is uniform in $s\in I$ and $w\in\Delta$. We calculate
\[
\L_{s}^{N}[1](w)=\sum_{\g\in(T_{\Gamma})^{N}}\left(\frac{(\gamma.w)_{n}}{w_{n}}\right)^{-s},\quad\frac{d}{ds}\L_{s}^{N}[1](w)=\sum_{\g\in(T_{\Gamma})^{N}}-\log\left(\frac{(\gamma.w)_{n}}{w_{n}}\right)\left(\frac{(\gamma.w)_{n}}{w_{n}}\right)^{-s}.
\]
Now we make the \emph{Claim}:\emph{ }There is some $c>0$ such that
\[
\log\left(\frac{(\gamma.w)_{n}}{w_{n}}\right)\geq cN.
\]
for all $\gamma\in(T_{\Gamma})^{N}.$ Assuming the Claim we get
\[
\frac{d}{ds}\L_{s}^{N}[1](w)\leq-cNL_{s}^{N}[1](w)
\]
 and hence
\[
\frac{d}{ds}\log\L_{s}^{N}[1](w)\leq-cN.
\]
This means $\log\lambda_{s}$ is a uniform limit of functions with
derivatives bounded above by a negative constant, so $\lambda_{s}$
must be strictly decreasing as required. 

To prove the Claim it is enough to show (by expanding $\log(\g.w)_{n}-\log w_{n}$
as a telescoping sum) that for all $w\in 
\Delta
$ and $\gamma'=\gamma_{n-1}^{A}\gamma_{j}\in T_{\Gamma}$ 

\[
\frac{(\gamma'.w)_{n}}{w_{n}}=1+(A+1)(1-w_{j})\geq c.
\]
This is true with $c=3/2$ since $w_{j}\leq1/2.$ This completes the
proof of Proposition \ref{prop:The-eigenfunction-decreases}.

\section{Proof of uniform contraction}

\label{sec:dynamics}

In this section we prove Proposition \ref{prop:uniformly-contracting}
asserting that the elements of $T_{\Gamma}$ eventually uniformly contract
$\Delta.$

\subsection{Setup}

We define the sets 
\[
\Delta\equiv\{(w_{1},w_{2},\ldots,w_{n-2},w_{n-1}):0\leq w_{1}\leq w_{2}\leq 
\ldots
\leq w_{n-2}\leq w_{n-1}\leq1,\sum_{i\in[n-1]}w_{i}=1\},
\]
\[
\core\equiv\{(w_{1},w_{2},\ldots,w_{n-2},w_{n-1})\in\Delta:0\leq w_{n-1}-\sum_{j\in[n-2]}w_{j}\leq w_{n-2}\},
\]
and 
\[
\cusp\equiv\{(w_{1},w_{2},\ldots,w_{n-2},w_{n-1})\in\Delta:w_{n-1}-\sum_{j\in[n-2]}w_{j}\geq w_{n-2}\}
\]
where we use the notation $[N]=\{1,2,\ldots,N\}$. We also define
the set 
\[
\deltao\equiv\core\cup\cusp.
\]

Recall that the elements of $T_{\Gamma}$ are all of the form $\gamma=\gamma_{n-1}^{L}\gamma_{i}$
where $L\in\mathbb{N}$ and $i=1,2,\ldots,n-2$. Note that for each
$w\in\Delta$, we have $\gamma_{i}(w)\in\core$ for $i=1,2,\ldots,n-2$
and $\gamma_{n-1}(w)\in\cusp$. In particular, $\gamma(w)\in\deltao$
for all $\gamma\in T_{\Gamma}$ and $w\in\Delta.$

From now on, we choose to use $n-2$ coordinates in $\Delta$ instead
of $n-1$, using the relationship $w_{n-1}=1-{\displaystyle \sum_{i\in[n-2]}w_{i}}$.

Note that on $\deltao$ we have ${\displaystyle \sum_{j\in[n-2]}w_{j}\leq\frac{1}{2}}$
by combining the conditions that 
\[
{\displaystyle w_{n-1}\geq\sum_{j\in[n-2]}w_{j}}
\]
and 
\[
1-w_{n-1}={\displaystyle \sum_{j\in[n-2]}w_{j}}.
\]
Similarly, it is easy to show that on $\core$ we have $w_{n-2}\leq\frac{1}{2}$
and $w_{n-3}\leq\frac{1}{4}$, while on $\cusp$ we have $w_{n-2}\leq\frac{1}{3}$
and $w_{n-3}\leq\frac{1}{5}$.
\begin{rem}
\label{rem:norms-equivalent}It is clear that Proposition \ref{prop:uniformly-contracting}
can be proved with the local $\ell^{2}$ operator norms replaced by
local $\ell^{1}$ norms, since the norms are equivalent, possibly
at the expense of a larger $N.$
\end{rem}

\subsection{Overview of the proof of Proposition \ref{prop:uniformly-contracting}}

We will now prove Proposition \ref{prop:uniformly-contracting} (the
$\ell^{1}$ norm variant). We will appeal to the following bounds.
\begin{gather}
\onenorm{d\gamma_{i}}\leq\dfrac{2}{2-w_{i}}\leq\begin{cases}
\frac{6}{5} & \on\cusp\\
\frac{4}{3} & \on\core
\end{cases},1\leq i\leq n-3\label{i}\\
\onenorm{d\gamma_{n-1}}=\dfrac{1+2(w_{1}+w_{2}+\ldots w_{n-2})-2w_{1}}{(1+w_{1}+w_{2}+\ldots w_{n-2})^{2}}\leq1\on\deltao\label{n-1}\\
\onenorm{d(\gamma_{i}\circ\gamma_{j})}\leq \dfrac{2}{4-2w_{j}-w_{i}}\leq\frac{4}{5}\on\deltao,1\leq i<j\leq n-2\label{smalli,j}\\
\onenorm{d(\gamma_{i}\circ\gamma_{j})}\leq \dfrac{2}{4-2w_{j}-w_{i+1}}\leq\frac{4}{5}\on\deltao,1\leq j\leq i<n-2\label{bigi,j}\\
\onenorm{d(\gamma_{n-2}\circ\gamma_{j})} \leq \dfrac{4+2(w_{1}+\ldots+w_{n-2})-2w_{1}-3w_{j}}{3+(w_{1}+\ldots+w_{n-2})-2w_{j}}\leq\frac{4}{5}\on\deltao,1\leq j\leq n-2\label{n-2,j}\\
\onenorm{d(\gamma_{n-1}\circ\gamma_{i})} \leq \dfrac{2}{3-2w_{i}}\leq\begin{cases}
\frac{10}{13} & \on\cusp\\
\frac{4}{5} & \on\core
\end{cases},1\leq i\leq n-3\label{n-1,i}\\
\onenorm{d(\gamma_{n-1}\circ\gamma_{n-2})} \leq \dfrac{2}{3-2w_{n-2}}\leq\begin{cases}
\frac{6}{7} & \on\cusp\\
1 & \on\core
\end{cases}\label{n-1,n-2}\\
\onenorm{d(\gamma_{i}\circ\gamma_{n-1}\circ\gamma_{n-2})} \leq \dfrac{2}{6-4w_{n-2}-w_{i}}\leq\frac{4}{7}\on\deltao,1\leq i\leq n-3\label{i,n-1,n-2}\\
\onenorm{d(\gamma_{n-2}\circ\gamma_{n-1}\circ\gamma_{n-2})} \leq \dfrac{7+2(w_{1}+\ldots+w_{n-2})-2w_{1}-6w_{n-2}}{5+(w_{1}+\ldots+w_{n-2})-4w_{n-2}}\leq\frac{32}{49}\on\deltao\label{n-2,n-1,n-2}\\
\onenorm{d(\gamma_{n-1}\circ\gamma_{n-1}\circ\gamma_{n-2})}\leq \dfrac{2}{4-3w_{n-2}}\leq\begin{cases}
\frac{2}{3} & \on\cusp\\
\frac{4}{5} & \on\core
\end{cases}\label{n-1,n-1,n-2}
\end{gather}

We will prove these bounds below by direct calculation. Using these
bounds we can prove the following result for any $n\geq3$ which implies
Proposition \ref{prop:uniformly-contracting} via Remark \ref{rem:norms-equivalent}.
\begin{lem}
Given the bounds (\ref{i})-(\ref{n-1,n-1,n-2}), $\onenorm{d(\gamma_{n-1}^{L}\circ\gamma_{i}\circ\gamma_{n-1}^{K}\circ\gamma_{j})\big|_{\deltao}}\leq\frac{24}{25}$
for each $L,K\in\mathbb{N}$, and each $i,j=1,2,\ldots,n-2$. \end{lem}
\begin{proof}
\emph{T}hroughout this proof, we repeatedly use the fact that $\gamma_{k}(w)\in\core$
for $k=1,2,\ldots,n-2$ and $\gamma_{n-1}(w)\in\cusp$. We distinguish
$3$ cases.

\noindent \textbf{Case I: $L\geq1,K\geq1$:} \\
Using equations (\ref{n-1}), (\ref{n-1,i}) and (\ref{n-1,n-2}),
we have 
\begin{multline*}
\onenorm{d(\gamma_{n-1}^{L}\circ\gamma_{i}\circ\gamma_{n-1}^{K}\circ\gamma_{j})\big|_{\deltao}}\\
\leq\onenorm{d\gamma_{n-1}^{L-1}\big|_{\cusp}}\onenorm{d(\gamma_{n-1}\circ\gamma_{i})\big|_{\cusp}}\onenorm{d\gamma_{n-1}^{K-1}\big|_{\cusp}}\onenorm{d(\gamma_{n-1}\circ\gamma_{j})\big|_{\deltao}}\leq1\cdot\frac{6}{7}\cdot1\cdot1<\frac{24}{25}.
\end{multline*}
\\
 \textbf{Case II: $L\geq0,K=0$:} \\
Using equations (\ref{n-1}), (\ref{smalli,j}), (\ref{bigi,j}),
(\ref{n-2,j}), we have 
\[
\onenorm{d(\gamma_{n-1}^{L}\circ\gamma_{i}\circ\gamma_{j})\big|_{\deltao}}\leq\onenorm{d\gamma_{n-1}^{L}\big|_{\core}}\onenorm{d(\gamma_{i}\circ\gamma_{j})\big|_{\deltao}}\leq1\cdot\frac{4}{5}<\frac{24}{25}.
\]
\\
 \textbf{Case III: $L=0,K\geq1$:} \\
We first suppose that $j\leq n-3$. Then by equations (\ref{i}),
(\ref{n-1}), (\ref{n-1,i}) we have 
\[
\onenorm{d(\gamma_{i}\circ\gamma_{n-1}^{K}\circ\gamma_{j})\big|_{\deltao}}\leq\onenorm{d\gamma_{i}\big|_{\cusp}}\onenorm{d\gamma_{n-1}^{K-1}\big|_{\cusp}}\onenorm{d(\gamma_{n-1}\circ\gamma_{j})\big|_{\deltao}}\leq\frac{6}{5}\cdot1\cdot\frac{4}{5}=\frac{24}{25}.
\]
Finally, if $j=n-2$ we are left with two subcases. If $K=1$, then
by equations (\ref{i,n-1,n-2}) and (\ref{n-2,n-1,n-2}) we have 
\[
\onenorm{d(\gamma_{i}\circ\gamma_{n-1}\circ\gamma_{n-2})\big|_{\deltao}}\leq\frac{32}{49}<\frac{24}{25}.
\]
Otherwise, we have $K\geq2$ and by equations (\ref{i}), (\ref{n-1}),
(\ref{n-1,n-1,n-2}) we have 
\begin{eqnarray*}
\onenorm{d(\gamma_{i}\circ\gamma_{n-1}^{K}\circ\gamma_{n-2})\big|_{\deltao}} & \leq & \onenorm{d\gamma_{i}\big|_{\cusp}}\onenorm{d\gamma_{n-1}^{K-2}\big|_{\cusp}}\onenorm{d(\gamma_{n-1}\circ\gamma_{n-1}\circ\gamma_{n-2})\big|_{\deltao}}\\
 & \leq & \frac{6}{5}\cdot1\cdot\frac{4}{5}=\frac{24}{25}.
\end{eqnarray*}
\end{proof}
In the remainder of
 this section
   we prove equations (\ref{i})-(\ref{n-1,n-1,n-2})
by direct calculation. In all following sections, we define 
\[
 w\equiv(w_{1},w_{2},\ldots,w_{n-2},1-\sum_{k=1}^{n-2}w_{k},1),
\]
and
\[
 \wsum\equiv{\displaystyle \sum_{k=1}^{n-2}w_{k}}.
\]
Also recall that the $\onenorm{\cdot}$ of a matrix is equal to the
maximum over columns of the matrix of the sum of the absolute values
of the column. From now on, we call such a sum an \emph{absolute column
sum}.

\subsection{Proof of equations (\ref{i})-(\ref{n-1,n-1,n-2})}

\subsubsection*{Proof of equation (\ref{i})\label{firsteqn}}

For $i=1,2,\ldots,n-2$ we have 
\[
\gamma_{i}(w)=(w_{1},w_{2},\ldots,\widehat{w_{i}},\ldots,w_{n-2},1-\wsum,1,2-w_{i}),
\]
which, after projectivizing and removing the placeholder components,
gives 
\[
\gamma_{i}(w)=\bigg(\frac{w_{1}}{2-w_{i}},\frac{w_{2}}{2-w_{i}},\ldots,\widehat{\frac{w_{i}}{2-w_{i}}},\ldots,\frac{w_{n-2}}{2-w_{i}},\frac{1-\wsum}{2-w_{i}}\bigg)
\]
which is a function in $(n-2)$ variables with $(n-2)$ components.
The $(n-2)\times(n-2)$ total derivative $d\gamma_{i}$ is given by
the following matrix 

{\small{}\begin{equation*}
\kbordermatrix{&1&2&3&\ldots&i-1&i&i+1&\ldots & n-3 & n-2\\ 1 &   \dfrac{1}{2-w_i}  & 0 & 0 & \ldots & 0 & \dfrac{w_1}{(2-w_i)^2} & 0 & \ldots & 0 & 0 \\ 2 &  0  & \dfrac{1}{2-w_i} & 0 & \ldots & 0 & \dfrac{w_2}{(2-w_i)^2} & 0 & \ldots & 0 & 0 \\ 3 &  0  & 0 & \dfrac{1}{2-w_i}  & \ldots & 0 & \dfrac{w_2}{(2-w_i)^2} & 0 & \ldots & 0 & 0 \\ \vdots & \vdots & \vdots & \vdots & \ddots & \vdots & \vdots & \vdots & \ddots & \vdots & \vdots \\ i-1 &0 & 0 & 0 & \ldots & \dfrac{1}{2-w_i} & \dfrac{w_{i-1}}{(2-w_i)^2} & 0 & \ldots & 0 & 0 \\ i & 0 & 0 & 0 & \ldots & 0 & \dfrac{w_{i+1}}{(2-w_i)^2} & \dfrac{1}{2-w_i} & \ldots & 0 & 0 \\ \vdots & \vdots & \vdots & \vdots & \ddots & \vdots & \vdots & \vdots & \ddots & \vdots & \vdots \\ n-3 & 0 & 0 & 0 & \ldots & 0 & \dfrac{w_{n-2}}{(2-w_i)^2} & 0 & \ldots & 0 & \dfrac{1}{2-w_i} \\ n-2 & \dfrac{-1}{2-w_i} & \dfrac{-1}{2-w_i} & \dfrac{-1}{2-w_i} & \ldots & \dfrac{-1}{2-w_i} & \dfrac{-1 + w_i - \wsum}{(2-w_i)^2} & \dfrac{-1}{2-w_i} & \ldots & \dfrac{-1}{2-w_i} & \dfrac{-1}{2-w_i} \\ }
\end{equation*}}where the row and column indices are indicated to the left and above
respectively. Each of these partial derivatives is immediate, except
for the $(n-2,i)$ entry which follows from an application of the
quotient rule. Note that the sign of entry $(n-2,i)$ is negative
on $\deltao$. The signs of the other entries are self-evident.

The absolute column sum for each column $k$ with $k\neq i$ is 
\[
C_{k}=\dfrac{2}{2-w_{i}}.
\]
For column $k=i$ the absolute column sum is 
\[
C_{i}=\dfrac{1+2\wsum-2w_{i}}{(2-w_{i})^{2}}.
\]
We must compute which absolute column sum is maximal on $\deltao$.
Note that on $\deltao$ we have $\wsum\leq\frac{1}{2}$. Furthermore,
we have the following equivalences: 
\[
C_{i}\leq C_{k},k\neq i\hspace{0.3in}\Leftrightarrow\hspace{0.3in}1+2\wsum-2w_{i}<4-2w_{i}\hspace{0.3in}\Leftrightarrow\hspace{0.3in}\wsum<\frac{2}{3}.
\]
Any column $k\neq i$ is maximal and $\onenorm{d\gamma_{i}} \leq \dfrac{2}{2-w_{i}}$
on $\deltao$. For each $i$, we have $w_{i}\leq\frac{1}{3}$ on $\cusp$,
and $w_{i}\leq\frac{1}{2}$ on $\core$. This gives the bound that
$\onenorm{d\gamma_{i}}\leq\frac{6}{5}$ on $\cusp$ and $\leq\frac{4}{3}$
on $\core$, proving equation (\ref{i}). 


\subsubsection*{Proof of equation (\ref{n-1})}

We have 
\[
\gamma_{n-1}(w)=(w_{1},w_{2},\ldots,w_{n-2},1,1+\wsum)
\]
which after projectivizing and removing placeholder components becomes
\[
\gamma_{n-1}(w)=\bigg(\frac{w_{1}}{1+\wsum},\frac{w_{2}}{1+\wsum},\ldots,\frac{w_{n-2}}{1+\wsum}\bigg).
\]
The $(n-2)\times(n-2)$ total derivative $d\gamma_{n-1}$ is given
by the following matrix

\begin{equation*}
\kbordermatrix{&1&2&3&\ldots& n-2\\  1 &   \dfrac{1 + \wsum - w_1}{(1 + \wsum)^2} &  \dfrac{-w_1}{(1 + \wsum)^2} &\dfrac{-w_1}{(1 + \wsum)^2}  & \ldots & \dfrac{-w_1}{(1 + \wsum)^2} \\ 2&  \dfrac{-w_2}{(1 + \wsum)^2} &      \dfrac{1 + \wsum - w_2}{(1 + \wsum)^2} & \dfrac{-w_2}{(1 + \wsum)^2}  & \ldots & \dfrac{-w_2}{(1 + \wsum)^2} \\ \vdots & \vdots & \vdots & \vdots & \ddots & \vdots \\ n-2 & \dfrac{-w_{n-2}}{(1 + \wsum)^2} & \dfrac{-w_{n-2}}{(1 + \wsum)^2} & \dfrac{-w_{n-2}}{(1 + \wsum)^2}  & \ldots & \dfrac{1 + \wsum - w_{n-2}}{(1 + \wsum)^2} }.
\end{equation*}For each column $k=1,2,\ldots,n-2$ we have the absolute column sum
\[
C_{k}=\frac{1+2\wsum-2w_{k}}{(1+\wsum)^{2}}.
\]
Since $w_{1}\leq w_{2}\leq\ldots\leq w_{n-2}$ on $\deltao$, we have
column $C_{1}$ is maximal. Hence, $\onenorm{d\gamma_{n-1}}=\dfrac{1+2\wsum-2w_{1}}{(1+\wsum)^{2}}$
on $\deltao$.

To bound this norm, observe 
\[
\frac{1+2\wsum-2w_{1}}{(1+\wsum)^{2}}=\frac{1+2\wsum-2w_{1}}{1+2\wsum+\wsum^{2}}\leq\frac{1+2\wsum-2w_{1}}{1+2\wsum-2w_{1}}=1.
\]
Hence $\onenorm{d\gamma_{n-1}}\leq1$ on $\deltao$, proving equation
(\ref{n-1}).

\subsubsection*{Proof of equations (\ref{smalli,j}) and (\ref{bigi,j})}

We first prove equation (\ref{smalli,j}). Assume first that $1\leq i<j\leq n-2.$
We have 
\[
\gamma_{i}\circ\gamma_{j}(w)=(w_{1},w_{2},\ldots,\widehat{w_{i}},\ldots,\widehat{w_{j}},\ldots,w_{n-2},1-\wsum,1,2-w_{j},4-2w_{j}-w_{i}).
\]
Define $\wdenomij\equiv4-2w_{j}-w_{i}.$ Then, after projectivizing
and removing placeholder components we have 
\[
\gamma{}_{i}\circ\gamma_{j}(w)=\bigg(\frac{w_{1}}{\wdenomij},\frac{w_{2}}{\wdenomij},\ldots,\widehat{\frac{w_{i}}{\wdenomij}},\ldots,\widehat{\frac{w_{j}}{\wdenomij}},\ldots,\frac{w_{n-2}}{\wdenomij},\frac{1-\wsum}{\wdenomij},\frac{1}{\wdenomij}\bigg).
\]

The $(n-2)\times(n-2)$ total derivative $d(\gamma_{i}\circ\gamma_{j})$
is given by the matrix\footnotesize

\begin{equation*}
\kbordermatrix{&1&2&\ldots&i-1&i&i+1&\ldots&j-1&j&j+1&\ldots& n-2\\ 1 & \dfrac{1}{\wdenomij} & 0 & \ldots & 0 & \dfrac{w_1}{(\wdenomij)^2} & 0 & \ldots & 0 & \dfrac{2w_1}{(\wdenomij)^2}& 0 & \ldots & 0 \\ 2& 0 & \dfrac{1}{\wdenomij} & \ldots & 0 & \dfrac{w_2}{(\wdenomij)^2} & 0 & \ldots & 0 & \dfrac{2w_2}{(\wdenomij)^2} & 0 & \ldots & 0 \\ \vdots & \vdots & \vdots & \ddots &  \vdots &  \vdots & \vdots & \ddots &  \vdots &  \vdots & \vdots & \ddots & \vdots \\ i-1& 0 & 0 & \ldots & \dfrac{1}{\wdenomij}  & \dfrac{w_{i-1}}{(\wdenomij)^2} & 0 & \ldots & 0 & \dfrac{2w_{i-1}}{(\wdenomij)^2} & 0 & \ldots & 0 \\ i& 0 & 0 & \ldots & 0 & \dfrac{w_{i+1}}{(\wdenomij)^2} & \dfrac{1}{\wdenomij}  & \ldots & 0 & \dfrac{2w_{i+1}}{(\wdenomij)^2} & 0 & \ldots & 0 \\ \vdots & \vdots & \vdots & \ddots &  \vdots &  \vdots & \vdots & \ddots &  \vdots &  \vdots & \vdots & \ddots & \vdots \\ j-2& 0 & 0 & \ldots & 0 & \dfrac{w_{j-1}}{(\wdenomij)^2} & 0 & \ldots & \dfrac{1}{\wdenomij} & \dfrac{2w_{j-1}}{(\wdenomij)^2} & 0 & \ldots & 0 \\ j-1& 0 & 0 & \ldots & 0 & \dfrac{w_{j+1}}{(\wdenomij)^2} & 0 & \ldots & 0 & \dfrac{2w_{j+1}}{(\wdenomij)^2} & \dfrac{1}{\wdenomij} & \ldots & 0 \\ \vdots & \vdots & \vdots & \ddots &  \vdots &  \vdots & \vdots & \ddots &  \vdots &  \vdots & \vdots & \ddots & \vdots \\ n-4 & 0 & 0 & \ldots & 0 & \dfrac{w_{n-2}}{(\wdenomij)^2} & 0 & \ldots & 0 & \dfrac{2w_{n-2}}{(\wdenomij)^2} & 0 & \ldots & \dfrac{1}{\wdenomij} \\ n-3 & \dfrac{-1}{\wdenomij} & \dfrac{-1}{\wdenomij} & \ldots & \dfrac{-1}{\wdenomij} & \frac{-3 - \wsum + 2w_j + w_i}{(\wdenomij)^2} & \dfrac{-1}{\wdenomij} & \ldots & \dfrac{-1}{\wdenomij} & \frac{-2 - 2\wsum + 2w_j + w_i}{(\wdenomij)^2} & \dfrac{-1}{\wdenomij} & \ldots & \dfrac{-1}{\wdenomij} \\ n-2 & 0 & 0 & \ldots & 0 & \dfrac{1}{(\wdenomij)^2} & 0 & \ldots & 0 & \dfrac{2}{(\wdenomij)^2} & 0 & \ldots & 0 \\ }.
\end{equation*}

\normalsize

The absolute column sum of each column $k\neq i,j$ is 
\[
C_{k}=\dfrac{2}{\wdenomij}=\dfrac{2(4-2w_{j}-w_{i})}{(\wdenomij)^{2}}
\]
The absolute column sum of column $i$ is 
\[
C_{i}=\dfrac{4+2\wsum-3w_{j}-2w_{i}}{(\wdenomij)^{2}}
\]
and the absolute column sum of row $j$ is 
\[
C_{j}=\dfrac{4+4\wsum-4w_{j}-3w_{i}}{(\wdenomij)^{2}}.
\]

Subtracting $C_{i}$ from $C_{j}$ we obtain $\dfrac{2\wsum-w_{j}-w_{i}}{(\wdenomij)^{2}}$
which is nonnegative on $\deltao$. Thus $C_{j}\geq C_{i}$. Furthermore
$C_{k}\geq C_{j}$ for each $k\neq i,j$ since 
\[
4+4\wsum-4w_{j}-3w_{i}\leq2(4-2w_{j}-w_{i})\hspace{0.3in}\Leftrightarrow\hspace{0.3in}4\wsum-w_{i}\leq4
\]
and the latter inequality holds since $\wsum\leq\frac{1}{2}$ and
$w_{i}\geq0$ on $\deltao$.

Thus $\onenorm{d(\gamma_{i}\circ\gamma_{j})} \leq \dfrac{2}{\wdenomij}=\dfrac{2}{4-2w_{j}-w_{i}}$
on $\deltao$. Using the bound that each $w_{k}\leq\frac{1}{2}$ on
$\deltao$ we have $\onenorm{d(\gamma_{i}\circ\gamma_{j})}\leq\frac{4}{5}$,
proving equation (\ref{smalli,j}).

To prove equation (\ref{bigi,j}), we now consider $\gamma_{i}\circ\gamma_{j}$
for $1\leq j\leq i<n-2$. We have 
\[
\gamma_{i}\circ\gamma_{j}(w)=(w_{1},w_{2},\ldots,\widehat{w_{j}},\ldots,\widehat{w_{i+1}},\ldots,w_{n-2},1-\wsum,1,2-w_{j},4-2w_{j}-w_{i+1}).
\]
The remainder of the proof for equation (\ref{bigi,j}) is nearly
identical after careful bookkeeping of indices (for example, column
$C_{i+1}$ for $1\leq j\leq i<n-2$ plays the role of $C_{i}$ for
$1\leq i<j\leq n-2$).

\subsubsection*{Proof of equation (\ref{n-2,j})}

We have 
\[
\gamma_{n-2}\circ\gamma_{j}(w)=(w_{1},w_{2},\ldots,\widehat{w_{j}},\ldots,w_{n-2},1,2-w_{j},3+\wsum-2w_{j}).
\]
We define $\wdenommm\equiv3+\wsum-2w_{j}$. Then, after projectivizing
and removing placeholder components, we have 

\[
\gamma_{n-2}\circ\gamma_{j}(w)=\bigg(\dfrac{w_{1}}{\wdenommm},\dfrac{w_{2}}{\wdenommm},\ldots,\widehat{\dfrac{w_{j}}{\wdenommm}},\ldots,\dfrac{w_{n-2}}{\wdenommm},\dfrac{1}{\wdenommm}\bigg).
\]
The $(n-2)\times(n-2)$ total derivative, $d(\gamma_{n-2}\circ\gamma_{j})$
is the matrix \scriptsize

\[ \kbordermatrix{&1&2&\ldots&j-1&j&j+1&\ldots& n-2\\ 1 & \frac{3 + \wsum - 2w_j - w_1}{(\wdenommm)^2} & \dfrac{-w_1}{(\wdenommm)^2} & \ldots & \dfrac{-w_1}{(\wdenommm)^2} & \dfrac{w_1}{(\wdenommm)^2} & \dfrac{-w_1}{(\wdenommm)^2} & \ldots & \dfrac{-w_1}{(\wdenommm)^2} \\ 2 & \dfrac{-w_2}{(\wdenommm)^2} & \frac{3 + \wsum - 2w_j - w_2}{(\wdenommm)^2} & \ldots & \dfrac{-w_2}{(\wdenommm)^2} & \dfrac{w_2}{(\wdenommm)^2} & \dfrac{-w_2}{(\wdenommm)^2} & \ldots & \dfrac{-w_2}{(\wdenommm)^2} \\ \vdots & \vdots & \vdots & \ddots & \vdots  & \vdots  & \vdots & \ddots & \vdots  \\ j-1 & \dfrac{-w_{j-1}}{(\wdenommm)^2} & \dfrac{-w_{j-1}}{(\wdenommm)^2} & \ldots & \frac{3 + \wsum - 2w_j - w_{j-1}}{(\wdenommm)^2} & \dfrac{w_{j-1}}{(\wdenommm)^2} & \dfrac{-w_{j-1}}{(\wdenommm)^2} & \ldots & \dfrac{-w_{j-1}}{(\wdenommm)^2} \\ j & \dfrac{-w_{j+1}}{(\wdenommm)^2} & \dfrac{-w_{j+1}}{(\wdenommm)^2} & \ldots & \dfrac{-w_{j+1}}{(\wdenommm)^2} & \dfrac{w_{j+1}}{(\wdenommm)^2} & \frac{3 + \wsum - 2w_j - w_{j+1}}{(\wdenommm)^2} & \ldots & \dfrac{-w_{j+1}}{(\wdenommm)^2} \\ \vdots & \vdots & \vdots & \ddots & \vdots  & \vdots  & \vdots & \ddots & \vdots  \\ n-3 & \dfrac{-w_{n-2}}{(\wdenommm)^2} & \dfrac{-w_{n-2}}{(\wdenommm)^2} & \ldots & \dfrac{-w_{n-2}}{(\wdenommm)^2} & \dfrac{w_{n-2}}{(\wdenommm)^2} & \dfrac{w_{n-2}}{(\wdenommm)^2} & \ldots & \frac{3 + \wsum - 2w_j - w_{n-2}}{(\wdenommm)^2} \\ n-2 & \dfrac{-1}{(\wdenommm)^2} & \dfrac{-1}{(\wdenommm)^2} &  \ldots & \dfrac{-1}{(\wdenommm)^2} & \dfrac{1}{(\wdenommm)^2} & \dfrac{-1}{(\wdenommm)^2} &  \ldots & \dfrac{-1}{(\wdenommm)^2} \\ }. \]

\normalsize

The absolute column sum for column $k\neq j$ is 
\[
C_{k}=\dfrac{4+2\wsum-3w_{j}-2w_{k}}{(\wdenommm)^{2}}
\]
and the absolute column sum for column $j$ is 
\[
C_{j}=\dfrac{1+\wsum-w_{j}}{(\wdenommm)^{2}}.
\]
Note that $w_{1}\leq w_{k}$ for all $k$, so $C_{1}\leq C_{k}$ for
each $k\neq j$. Furthermore, subtracting column sum $C_{j}$ from
$C_{1}$ and using the trivial bound $w_{k}\leq\frac{1}{2}$ on $\deltao$
for all $k$ we obtain 
\[
\frac{3+\wsum-2w_{j}-2w_{1}}{(\wdenommm)^{2}}\geq\frac{1+\wsum}{(\wdenommm)^{2}}\geq0.
\]
Hence, $C_{1}\geq C_{j}$, and $C_{1}$ is maximal. We have $\onenorm{d(\gamma_{n-2}\circ\gamma_{j})} \leq \dfrac{4+2\wsum-3w_{j}-2w_{1}}{(\wdenommm)^{2}}$
on $\deltao$. Separately bounding the numerator and denominator on
$\deltao$ we have 
\[
4+2\wsum-3w_{j}-2w_{i}\leq4+2\bigg(\frac{1}{2}\bigg)=5,
\]
\[
\wdenommm=3+\wsum-2w_{j}\geq3+(\wsum-w_{j})-w_{j}\geq3-\frac{1}{2}=\frac{5}{2}.
\]
Thus $\onenorm{d(\gamma_{n-2}\circ\gamma_{j})}\leq\dfrac{5}{(\frac{5}{2})^{2}}=\frac{4}{5}$
on $\deltao$. This proves equation (\ref{n-2,j}).

\subsubsection*{Proof of equation (\ref{n-1,i}) and (\ref{n-1,n-2})}

For $i=1,2,\ldots,n-2$ we have 
\[
\gamma_{n-1}\circ\gamma_{i}(w)=(w_{1},w_{2},\ldots,\widehat{w_{i}},\ldots,w_{n-2},1-\wsum,2-w_{i},3-2w_{i})
\]
which, after projectivizing and removing placeholder components, becomes
\[
\gamma_{n-1}\circ\gamma_{i}(w)=\bigg(\dfrac{w_{1}}{3-2w_{i}},\dfrac{w_{2}}{3-2w_{i}},\ldots,\widehat{\dfrac{w_{i}}{3-2w_{i}}},\ldots,\dfrac{w_{n-2}}{3-2w_{i}},\dfrac{1-\wsum}{3-2w_{i}}\bigg).
\]

The $(n-2)\times(n-2)$ total derivative $d(\gamma_{n-1}\circ\gamma_{i})$
is given by 

\[ \kbordermatrix{&1&2&\ldots&i-1&i&i+1&\ldots& n-2\\ 1 &    \dfrac{1}{3-2w_i} & 0 & \ldots & 0 &  \dfrac{2w_1}{(3-2w_i)^2}& 0  & \ldots & 0 \\ 2 & 0 & \dfrac{1}{3-2w_i} & \ldots & 0 & \dfrac{2w_2}{(3-2w_i)^2} & 0 & \ldots & 0 \\ \vdots & \vdots & \vdots & \ddots & \vdots & \vdots & \vdots & \ddots & \vdots \\ i-1 & 0 & 0 & \ldots & \dfrac{1}{3-2w_i} & \dfrac{2w_{i-1}}{(3-2w_i)^2} & 0 & \ldots & 0 \\ i & 0 & 0 & \ldots & 0 & \dfrac{2w_{i+1}}{(3-2w_i)^2} &  \dfrac{1}{3-2w_i} & \ldots & 0 \\ \vdots & \vdots & \vdots & \ddots & \vdots & \vdots & \vdots & \ddots & \vdots \\ n-3 & 0 & 0 & \ldots & 0 & \dfrac{2w_{n-2}}{(3-2w_i)^2} & 0 & \ldots & \dfrac{1}{3-2w_i} \\ n-2 & \dfrac{-1}{3-2w_i} & \dfrac{-1}{3-2w_i} & \ldots &   \dfrac{-1}{3-2w_i} &\dfrac{-1- 2\wsum +2w_i}{(3-2w_i)^2} &  \dfrac{-1}{3-2w_i} & \ldots & \dfrac{-1}{3-2w_i} \\ }. \]

The absolute column sum for column $k\neq i$ is 
\[
C_{k}=\dfrac{2}{3-2w_{i}}
\]
and the absolute column sum for column $i$ is 
\[
C_{i}=\dfrac{1+4\wsum-4w_{i}}{(3-2w_{i})^{2}}.
\]

Each column $C_{k}$ with $k\neq i$ is maximal since 
\[
C_{k}\geq C_{i}\hspace{0.3in}\Leftrightarrow\hspace{0.3in}2(3-2w_{i})\geq1+4\wsum-4w_{i}\hspace{0.3in}\Leftrightarrow\hspace{0.3in}5\geq4\wsum
\]
and $\wsum\leq\frac{1}{2}$ on $\deltao$. Thus $\onenorm{d(\gamma_{n-1}\circ\gamma_{i})} \leq \dfrac{2}{3-2w_{i}}$
on $\deltao$.

When $i=n-2$, we have $w_{n-2}\leq\frac{1}{2}$ on $\core$ and $w_{n-2}\leq\frac{1}{3}$
on $\cusp$. Thus $\onenorm{d(\gamma_{n-1}\circ\gamma_{n-2})}\leq1$
on $\core$ and $\leq\dfrac{6}{7}$ on $\cusp$. This proves equation
(\ref{n-1,n-2}).

For $i\leq n-3$, we have the stronger bound $w_{i}\leq\frac{1}{4}$
on $\core$ and $w_{i}\leq\frac{1}{5}$ on $\cusp$. This gives $\onenorm{d(\gamma_{n-1}\circ\gamma_{i})}\leq\frac{4}{5}$
on $\core$ and $\leq\frac{10}{13}$. This proves equation (\ref{n-1,i}).

\subsubsection*{Proof of equation (\ref{i,n-1,n-2})}

For each $i\leq n-3$ we have 
\[
\gamma_{i}\circ\gamma_{n-1}\circ\gamma_{n-2}(w)=(w_{1},w_{2},\ldots,\widehat{w_{i}},\ldots,w_{n-3},1-\wsum,2-w_{n-2},3-2w_{n-2},6-4w_{n-2}-w_{i})
\]
which, after projectivizing and removing placeholder components, becomes
\[
\gamma_{i}\circ\gamma_{n-1}\circ\gamma_{n-2}(w)=\bigg(\dfrac{w_{1}}{\mu(w)},\ldots,\widehat{\dfrac{w_{i}}{\mu(w)}}\ldots\dfrac{w_{n-3}}{\mu(w)},\dfrac{1-\wsum}{\mu(w)},\dfrac{2-w_{n-2}}{\mu(w)}\bigg)
\]
where $\wdenom\equiv6-4w_{n-2}-w_{i}$. Then the $(n-2)\times(n-2)$
total derivative $d(\gamma_{i}\circ\gamma_{n-1}\circ\gamma_{n-2})$
is 

\[ \kbordermatrix{&1&\ldots&i-1&i&i+1&\ldots&n-3&n-2\\ 1 & \dfrac{1}{\wdenom} & \ldots & 0 & \dfrac{w_1}{(\wdenom)^2} & 0 & \ldots & 0 & \dfrac{4w_1}{(\wdenom)^2} \\ \vdots & \vdots& \ddots & \vdots & \vdots & \vdots & \ddots & \vdots & \vdots \\ i-1 & 0 & \ldots & \dfrac{1}{\wdenom} & \dfrac{w_{i-1}}{(\wdenom)^2} & 0 & \ldots & 0 & \dfrac{4w_{i-1}}{(\wdenom)^2} \\  i & 0 & \ldots & 0 & \dfrac{w_{i+1}}{(\wdenom)^2} & \dfrac{1}{\wdenom} & \ldots & 0 & \dfrac{4w_{i+1}}{(\wdenom)^2} \\  \vdots & \vdots & \ddots & \vdots & \vdots & \vdots & \ddots & \vdots & \vdots \\ n-4 & 0 & \ldots & 0 & \dfrac{w_{n-3}}{(\wdenom)^2} & 0 & \ldots & \dfrac{1}{\wdenom} & \dfrac{4w_{n-3}}{(\wdenom)^2} \\  n-3 & \dfrac{-1}{\wdenom} & \ldots & \dfrac{-1}{\wdenom}& \frac{-5 - \wsum + 4w_{n-2} + w_i}{(\wdenom)^2} &\dfrac{-1}{\wdenom}& \ldots & \dfrac{-1}{\wdenom} & \frac{-2 - 4\wsum + 4w_{n-2} + w_i}{(\wdenom)^2} \\  n-2 & 0 & \ldots & 0 & \dfrac{2-w_{n-2}}{(\wdenom)^2} & 0 & \ldots & 0 & \dfrac{2+w_i}{(\wdenom)^2}  }. \]The
absolute column sum of column $k\neq i,n-2$ is 
\[
C_{k}=\frac{2}{\wdenom}=\frac{2(6-4w_{n-2}-w_{i})}{(\wdenom)^{2}},
\]
whereas the absolute column sum of column $i$ is 
\[
C_{i}=\frac{7+2\wsum-2w_{i}-6w_{n-2}}{(\wdenom)^{2}}
\]
and the absolute column sum of column $n-2$ is 
\[
C_{n-2}=\frac{4+8\wsum-8w_{n-2}-4w_{i}}{(\wdenom)^{2}}.
\]
Subtracting $C_{n-2}$ from $C_{i}$ we obtain 
\[
\frac{3-6\wsum+2w_{n-2}+2w_{i}}{(\wdenom)^{2}}\geq\frac{3-6(\frac{1}{2})}{(\wdenom)^{2}}\geq0
\]
on $\deltao$. This shows $C_{i}\geq C_{n-2}$. 

In fact, each column $C_{k}$ with $k\neq i,n-2$ is maximal since
\[
C_{k}\geq C_{i},k\neq i,n-2\hspace{0.2in}\Leftrightarrow\hspace{0.2in}2(6-4w_{n-2}-w_{i})\geq7+2\wsum-2w_{i}-6w_{n-2}\hspace{0.2in}\Leftrightarrow\hspace{0.2in}5\geq2\wsum+2w_{n-2}
\]
and $\wsum$, $w_{n-2}\leq\frac{1}{2}$ on $\deltao$. Hence $\onenorm{d(\gamma_{i}\circ\gamma_{n-1}\circ\gamma_{n-2})} \leq \dfrac{2}{6-4w_{n-2}-w_{i}}$
on $\deltao$. The denominator is bounded by 
\[
6-4w_{n-2}-w_{i}\geq6-5\bigg(\frac{1}{2}\bigg)=\frac{7}{2}
\]
so $\onenorm{d(\gamma_{i}\circ\gamma_{n-1}\circ\gamma_{n-2})}\leq\frac{4}{7}$
on $\deltao$. This proves equation (\ref{i,n-1,n-2}).

\subsubsection*{Proof of equation (\ref{n-2,n-1,n-2})}

We have 
\[
\gamma_{n-2}\circ\gamma_{n-1}\circ\gamma_{n-2}(w)=(w_{1},w_{2},\ldots,w_{n-3},2-w_{n-2},3-2w_{n-2},5+\wsum-4w_{n-2})
\]
which, after projectivizing and removing placeholder components, becomes
\[
\gamma_{n-2}\circ\gamma_{n-1}\circ\gamma_{n-2}(w)=\bigg(\dfrac{w_{1}}{\wdenomm},\dfrac{w_{2}}{\wdenomm},\ldots,\dfrac{w_{n-3}}{\wdenomm},\dfrac{2-w_{n-2}}{\wdenomm}\bigg)
\]
where $\wdenomm\equiv5+\wsum-4w_{n-2}$. Then the $(n-2)\times(n-2)$
total derivative $d(\gamma_{n-2}\circ\gamma_{n-1}\circ\gamma_{n-2})$
is

\[ \kbordermatrix{&1&2&3&\ldots&n-3&n-2\\ 1 & \frac{5 + \wsum - 4w_{n-2} - w_1}{(\wdenomm)^2} & \dfrac{-w_1}{(\wdenomm)^2} & \dfrac{-w_1}{(\wdenomm)^2} & \ldots &  \dfrac{-w_1}{(\wdenomm)^2} &  \dfrac{3w_1}{(\wdenomm)^2} \\ 2& \dfrac{-w_2}{(\wdenomm)^2} & \frac{5 + \wsum - 4w_{n-2} - w_2}{(\wdenomm)^2} & \dfrac{-w_2}{(\wdenomm)^2} & \ldots &  \dfrac{-w_2}{(\wdenomm)^2} &  \dfrac{3w_2}{(\wdenomm)^2} \\ 3&\dfrac{-w_3}{(\wdenomm)^2} & \dfrac{-w_3}{(\wdenomm)^2} & \frac{5 + \wsum - 4w_{n-2} - w_3}{(\wdenomm)^2} & \ldots &  \dfrac{-w_3}{(\wdenomm)^2} &  \dfrac{3w_3}{(\wdenomm)^2} \\ \vdots&\vdots & \vdots & \vdots & \ddots & \vdots &\vdots \\ n-3&\dfrac{-w_{n-3}}{(\wdenomm)^2} & \dfrac{-w_{n-3}}{(\wdenomm)^2} & \dfrac{-w_{n-3}}{(\wdenomm)^2} & \ldots & \frac{5 + \wsum - 4w_{n-2} - w_{n-3}}{(\wdenomm)^2} &  \dfrac{3w_{n-3}}{(\wdenomm)^2} \\ n-2&\dfrac{-2+w_{n-2}}{(\wdenomm)^2} & \dfrac{-2+w_{n-2}}{(\wdenomm)^2} &\dfrac{-2+w_{n-2}}{(\wdenomm)^2} & \ldots & \dfrac{-2+w_{n-2}}{(\wdenomm)^2} & \frac{1 - \wsum + w_{n-2} }{(\wdenomm)^2} }. \]

The absolute column sum for each column $k\leq n-3$ is 
\[
C_{k}=\dfrac{7+2\wsum-2w_{k}-6w_{n-2}}{(\wdenomm)^{2}}
\]
and the absolute column sum for column $n-2$ is 
\[
C_{n-2}=\dfrac{1+2\wsum-2w_{n-2}}{(\wdenomm)^{2}}.
\]

Subtracting $C_{n-2}$ from $C_{k}$ we obtain 
\[
\frac{6-2w_{k}-4w_{n-2}}{(\wdenomm)^{2}}\geq0
\]
on $\deltao$ (using the bound $w_{k}\leq\frac{1}{2}$ for all $k$).
Hence, $C_{j}\leq C_{k}$ for each $k\neq j$. Of the remaining column
sums, $C_{1}$ is maximal since $w_{1}\leq w_{k}$ for each $k$ on
$\deltao$. Thus, $\onenorm{d(\gamma_{n-2}\circ\gamma_{n-1}\circ\gamma_{n-2})} \leq \frac{7+2\wsum-2w_{1}-6w_{n-2}}{(5+\wsum-4w_{n-2})^{2}}.$

Separately bounding the numerator and the denominator we have 
\[
7+2\wsum-2w_{1}-6w_{n-2}\leq7+1\leq8
\]
\[
5+\wsum-4w_{n-2}\geq5+(\wsum-w_{n-2})-3w_{n-2}\geq5-3\bigg(\frac{1}{2}\bigg)=\frac{7}{2}
\]
so $\onenorm{d(\gamma_{n-2}\circ\gamma_{n-1}\circ\gamma_{n-2})}\leq\frac{32}{49}$
on $\deltao$, proving equation (\ref{n-2,n-1,n-2}).

\subsubsection*{Proof of equation (\ref{n-1,n-1,n-2})}

\label{lasteqn} We have 
\[
\gamma_{n-1}\circ\gamma_{n-1}\circ\gamma_{n-2}(w)=(w_{1},w_{2},\ldots,w_{n-3},1-\wsum,3-2w_{n-2},4-3w_{n-2})
\]
which, after projectivizing and removing placeholder components, becomes
\[
\gamma_{n-1}\circ\gamma_{n-1}\circ\gamma_{n-2}(w)=\bigg(\dfrac{w_{1}}{4-3w_{n-2}},\dfrac{w_{2}}{4-3w_{n-2}},\ldots,\dfrac{w_{n-3}}{4-3w_{n-2}},\dfrac{1-\wsum}{4-3w_{n-2}}\bigg).
\]

The $(n-2)\times(n-2)$ total derivative $d(\gamma_{n-1}\circ\gamma_{n-1}\circ\gamma_{n-2})$
is

\[ \kbordermatrix{&1&2&\ldots&n-3&n-2\\ 1&\dfrac{1}{4-3w_{n-2}} & 0 & \ldots & 0 & \dfrac{3w_1}{(4-3w_{n-2})^2} \\ 2&0 & \dfrac{1}{4-3w_{n-2}} & \ldots & 0 & \dfrac{3w_2}{(4-3w_{n-2})^2} \\ \vdots&\vdots & \vdots & \ddots & \vdots & \vdots \\ n-3&0 & 0 & \ldots & \dfrac{1}{4-3w_{n-2}} & \dfrac{3w_{n-3}}{(4-3w_{n-2})^2}  \\ n-2&\dfrac{-1}{4-3w_{n-2}} & \dfrac{-1}{4-3w_{n-2}} & \ldots & \dfrac{-1}{4-3w_{n-2}} &\dfrac{-1-3\wsum+3w_{n-2}}{(4-3w_{n-2})^2}  \\ }. \]

The absolute column sum for each column $k\leq n-3$ is 
\[
C_{k}=\dfrac{2}{4-3w_{n-2}}
\]
and the absolute column sum for column $n-2$ is 
\[
C_{n-2}=\dfrac{1+6\wsum-6w_{n-2}}{(4-3w_{n-2})^{2}}.
\]

Each column $C_{k}$ for $k\neq n-2$ is maximal because 
\[
C_{k}\geq C_{n-2},k\neq n-2\hspace{0.3in}\Leftrightarrow\hspace{0.3in}2(4-3w_{n-2})\geq1+6\wsum-6w_{n-2}\hspace{0.3in}\Leftrightarrow\hspace{0.3in}7\geq6\wsum.
\]

Thus, $\onenorm{d(\gamma_{n-1}\circ\gamma_{n-1}\circ\gamma_{n-2})} \leq \frac{2}{4-3w_{n-2}}$
on $\deltao$. Since $w_{n-2}\leq\frac{1}{2}$ on $\deltao$ we have
$\onenorm{d(\gamma_{n-1}\circ\gamma_{n-1}\circ\gamma_{n-2})}=\frac{4}{5}$,
proving equation (\ref{n-1,n-1,n-2}).

\bibliographystyle{alpha}
\bibliography{database}
\begin{multicols}{2}

\noindent Alex Gamburd, \\
CUNY Graduate Center, \\
New York, NY, USA \\
{\tt agamburd@gc.cuny.edu}\\

\noindent Michael Magee, \\
Durham University,\\
Durham, UK \\ 
{\tt michael.r.magee@durham.ac.uk}\\

\columnbreak

\noindent Ryan Ronan, \\
Baruch College (CUNY), \\
New York, NY, USA \\
{\tt ryan.ronan@baruch.cuny.edu}\\

\end{multicols}

\end{document}